\documentclass[11pt]{amsart}
\usepackage {amsmath, amssymb, a4wide,  enumerate, psfrag, color}
\usepackage{epsfig}
\usepackage[latin1]{inputenc}
\usepackage[english]{babel}

\date{\today}

\keywords{}
 \author{Romain Dujardin}
 \address{LAMA  \\
Université Paris-Est Marne-la-Vallée  \\
5 boulevard Descartes \\
77454 Champs sur Marne  \\
France}
\email{romain.dujardin@u-pem.fr} 
 
\author{Misha Lyubich}
\address{Mathematics Department and IMS \\ Stony Brook University \\ Stony Brook \\ NY 11794\\ USA.}
\email{mlyubich@math.sunysb.edu}

\title[Bifurcations of polynomial automorphisms of $\cd$]{Stability
  and bifurcations for dissipative polynomial automorphisms of $\cd$
  }

\newcommand{\cc}{\mathbb{C}}
\newcommand{\bb}{\mathbb{B}}
\newcommand{\re}{\mathbb{R}}
\newcommand{\dd}{\mathbb{D}}
\newcommand{\zz}{\mathbb{Z}}

\newcommand{\pp}{\mathbb{P}}
\newcommand{\e}{\varepsilon}
\newcommand{\cv}{\rightarrow}

\newcommand{\fr}{\partial}
\newcommand{\om}{\Omega}
\newcommand{\set}[1]{\left\{#1\right\}}
\newcommand{\norm}[1]{\left\Vert#1\right\Vert}
\newcommand{\abs}[1]{\left\vert#1\right\vert}
\newcommand{\cd}{{\cc^2}}
\newcommand{\pd}{{\mathbb{P}^2}}
\newcommand{\pu}{{\mathbb{P}^1}}
\newcommand{\rest}[1]{ \arrowvert_{#1}}

\newcommand{\unsur}[1]{\frac{1}{#1}}

\newcommand{\lrpar}[1]{\left(#1\right)}

\newcommand{\la}{\lambda}
\newcommand{\lo}{{\lambda_0}}

\newcommand{\La}{\Lambda}

\newcommand{\G}{\mathcal{G}}

\renewcommand{\Re}{\mathrm{Re}}

\newcommand{\itm}{\item[-]}

\DeclareMathOperator{\Int}{Int}
\DeclareMathOperator{\dist}{dist}

 \newtheorem{prop} {Proposition} [section]
\newtheorem{thm}[prop] {Theorem}
\newtheorem{defi}[prop] {Definition}
\newtheorem{lem}[prop] {Lemma}
\newtheorem{cor}[prop]{Corollary}
\newtheorem{theo}{Theorem}
\newtheorem{theoprime}{Theorem}

\newtheorem{conj}{Conjecture \hspace{-.5em}}

\theoremstyle{remark}

\newtheorem{rmk}[prop]{Remark}
\newtheorem{rmknonb}{Remark}

\newcommand{\Jac}{\operatorname{Jac}\,}

\newcommand{\OO}{{\mathcal{O}}}
\newcommand{\EE}{{\mathcal{E}}}
\newcommand{\FF}{{\mathcal{F}}}
\newcommand{\GG} {{\mathcal {G}}}
\newcommand{\PP}  {{\mathcal{P}}}

\newcommand{\Om}{{\Omega}}

\newcommand{\sm}{\setminus}
\newcommand{\ra}{\rightarrow}
\newcommand{\Hyp} {{\Bbb{H}}}
\newcommand{\ssk}{\smallskip}
\newcommand{\msk}{\medskip}
\newcommand{\nin}{\noindent}
\newcommand{\cl}{\operatorname{cl}}

\newcommand{\BHM}{{BHM}\ }
\newcommand{\Saddles}{\mathfrak{S}}

\newcommand{\de}{\delta}
\newcommand{\eps}{\epsilon}

\newcommand{\bJ}{\widehat J}
\newcommand{\isom}{\approx}
\newcommand{\intr}{{i}}

\newcommand{\comm}[1]{}

\begin{document}

 \begin{abstract}
 We study stability and bifurcations in holomorphic families of polynomial automorphisms of $\cd$. 
 We say that such a family is weakly stable over some parameter domain if periodic orbits do not bifurcate 
  there. We first show that this defines a meaningful notion of stability, which parallels in many ways the classical notion of $J$-stability in one-dimensional dynamics. Define the bifurcation locus to be the complement of the weak stability locus. 
   In the second part of the paper, we prove that under an assumption of moderate 
 dissipativity, the parameters displaying homoclinic tangencies are dense in the bifurcation locus. 
 This confirms one of Palis' Conjectures in the complex setting. The proof relies on the formalism of semi-parabolic bifurcation and the construction of ``critical points" in semi-parabolic basins (which makes use of the classical Denjoy-Carleman-Ahlfors and Wiman Theorems). 
 \end{abstract}
 
 \maketitle

 \section*{Introduction}
 
One of the main goals in the modern theory of dynamical systems is
to describe the dynamics of  typical  mappings in a representative family. 
 Let us consider for instance  the space of $C^k$ diffeomorphisms ($k\geq 1$) of real compact surfaces.   
It was briefly believed in the 1960's that hyperbolicity was  generically satisfied in 
$\mathrm {Diff}^k(M)$. 
This hope  was 
discouraged fast, particularly  with the discovery by  S. Newhouse
 \cite{newhouse1, newhouse2} of  an open region ${\mathcal N}$  in
 $\mathrm{Diff}^k(M)$,  $k\geq 2$ 
containing a dense subset of maps that
 display   homoclinic tangencies.
Moreover, a generic  map in  ${\mathcal N}$ has infinitely many sinks. 
(We will refer to ${\mathcal N}$ as the {\it Newhouse region}.)

A more refined picture of typical dynamics of diffeomorphisms
then  gradually   emerged. It was articulated  by  J. Palis as a series of conjectures
 (see e.g. \cite{Palis}, \cite[Chap. 7]{palis takens}). 
The first conjecture on this list  is the following:
 
 \begin{conj}[Palis] 
 Every $f \in \mathrm{Diff}^k (M)$, $k \geq  1$, can be
 $C^k$-approximated either by a hyperbolic diffeomorphism or by one  
 exhibiting a homoclinic tangency.
 \end{conj}
 
 Here  ``homoclinic tangency'' 
means a tangency between the stable and unstable manifolds of some saddle periodic point.  
Since hyperbolic diffeomorphisms are structurally stable, this singles out homoclinic tangencies as a basic phenomenon responsible for bifurcations. 
 This  conjecture was proven for $k=1$ by E. Pujals and M. 
 Sambarino \cite{pujals sambarino}, nevertheless  it remains wide open for $k>1$. 
 More generally,  there has been an important progress in the understanding of $C^1$-generic dynamics in the past few years 
 (see  \cite{crovisier} for a recent  overview).
 
 \medskip
 
 Another  situation that has been extensively studied 
 is  One-Dimensional Dynamics, both real and complex. 
In fact, the early Density of Hyperbolicity Conjecture turned out to
be true in the real one-dimensional case  \cite{Acta,GS,KSS}. 
It is conjectured to be true in the complex case as well
(this is known as the {\it Fatou Conjecture}), but this
problem is still open. 

 Consider a holomorphic family $(f_\la)_{\la\in \La}$ of rational mappings of degree $d$
 on the Riemann sphere $\pu(\cc)$, parameterized by a complex manifold $\La$ (which may be the whole space of rational mappings 
 of degree $d$). We say that the family is {\em $J$-stable} in a connected open subset 
 $\om\subset\La$ if in $\om$ the dynamics is   structurally stable on the Julia set $J$. 
Work of R. Ma\~né, P. Sad and D. Sullivan \cite{mss} and independently 
 of the second author \cite{lyubich,lyubich-2} implies  that the $J$-stability locus is dense in $\La$. In addition, 
parameters with 
preperiodic critical points 
(which is the one-dimensional counterpart of the homoclinic tangency)
are  dense in the bifurcation locus.   
We see that the Fatou Conjecture is reduced to the
problem  
whether $J$-stability implies 
hyperbolicity (for a sufficiently 
generic  family, 
like the whole space of 
polynomials or rational maps of a given  degree).

\medskip

In this paper we deal with families of   
polynomial automorphisms of $\cd$, which shares features with both of the previous settings. 
S. Friedland and J. Milnor \cite{fm}  showed
that dynamically interesting automorphisms in $\cd$ are conjugate to 
compositions of Hénon mappings $(z,w)\mapsto (aw + p(z), az)$, where $a$ is a non-zero complex number and $p$ is a polynomial
of degree at least two. 
{\it  In what follows, we assume without
   saying that all  automorphisms under consideration are dynamically
   interesting}, and in particular, they have dynamical degree $d\geq 2$ (see \S \ref{sec:prel} for a review of this notion).

Note that a polynomial automorphism $f$ has constant complex
Jacobian $\Jac f= \det Df$. 
So 
$\Jac f$
is a well-defined quantity attached  to $f$. We work in the dissipative setting, and our main results actually 
require some stronger form of  dissipation, namely we need 
\begin{equation}
  \abs{ \Jac f} < \frac 1{d^2}, \quad \mathrm{where }\ d \text{ is the dynamical degree of }f.
  \end{equation}
We will call such maps {\it moderately dissipative}\footnote{The word ``moderately" was chosen to  contrast with the 
 very strong dissipativity assumptions that are usually made in the study of real  Hénon mappings.}

We denote by $J^*$ the
closure of saddle periodic points of $f$.
 It is unknown whether $J^*$ is always equal to the ``small Julia set'' $J$,  
which can be defined in classical terms as the locus where both families 
$\{f^n\}_{n\geq 0}$ and $\{f^n\}_{n\leq 0}$  are not normal.

From one-dimensional holomorphic dynamics we borrow the idea of focusing on $J$-stability rather than hyperbolicity, and in 
accordance with the Palis program, we explain bifurcations by the presence of homoclinic tangencies. 
Our main result is the following, in the spirit of the Palis
conjecture
 (the precise meaning of the terminology ``weakly 
 stable" will be explained shortly). 

\begin{theo}\label{theo:dense}
 Let $(f_\la)_{\la\in \La}$ be a holomorphic family of 
moderately dissipative polynomial automorphisms of $\cd$
of dynamical degree $d\geq 2$. 
Then  weakly 
stable maps, together with maps exhibiting homoclinic
tangencies form a dense subset  of 
$\La$. 
 \end{theo}

It is also true that
{\it  weakly 
stable maps, together with maps that have
  infinitely many  sinks  form a dense subset  
in   $\La$}.  
Somewhat surprisingly, this is just an  observation
obtained by analyzing the one-dimensional argument. 

The set of  locally weakly 
stable parameters 
will be simply referred to as the {\em stability locus}, and its complement is by definition the {\em bifurcation locus}. 
It is worth  mentioning here that  G. Buzzard \cite{buzzard} showed
that {\it the Newhouse region is non-empty} 
in the  space of polynomial automorphisms of sufficiently high degree.
It follows that  the  stability locus is not dense in general.

\medskip

 Let us now 
 discuss the notion of weak 
 stability. To say it briefly, a family of polynomial automorphisms is weakly 
 stable in some open set 
  if periodic points do not bifurcate there. The first part 
  of this paper is devoted to demonstrating that this defines a reasonable notion of 
 stability in this context, parallel to the usual  $J$-stability in dimension 1. 
 In particular we show that in a weakly 
 stable family: 
 \begin{itemize}
\itm there are no homoclinic bifurcations, and moreover all homoclinic and heteroclinic 
intersections can be followed holomorphically;
\itm the sets $J^*$, $J^-$, $J^+$, $K$
 move continuously in the Hausdorff topology;
\itm  connectivity properties of the Julia sets are preserved. 
\end{itemize}
Let us point out that these results are true for any dissipative family, without further assumption on the Jacobian.
Naturally,  these results are based on a generalization to two dimensions of the key  
 idea of {\em holomorphic motion}.   
A fundamental problem here is  
 that a holomorphic motion of a set $X$ in higher dimension 
does not automatically admit an extension to a motion of $\overline X$.  
In practice, we work with a weaker notion of ``branched holomorphic motion", in which collisions are allowed. 
Because of this, we have not been able to prove that weak
stability implies structural $J^*$-stability.

An important special case is when $f$ is uniformly hyperbolic on $J^*$. 
It then follows from the classical theory of hyperbolic dynamical systems that $f$ is locally structurally stable on $J^*$. In addition it is known 
that  $J^*$ moves holomorphically (see Jonsson \cite{jonsson}), 
and that this holomorphic motion extends to a holomorphic motion of $J^+\cup J^-$ 
(see Buzzard-Verma \cite{buzzard verma}). 

\medskip

The main point of this paper is to design a mechanism
creating homoclinic tangencies from bifurcations of periodic points 
(for moderately dissipative polynomial automorphisms of $\cd$).
Notice that conversely, the creation 
of sinks from (generic) homoclinic tangencies is classical and goes back to
  Newhouse \cite{newhouse2} (see Gavosto \cite{gavosto} for a proof in our context).  
The theory of weak $J^*$-stability gives a fresh insight into this phenomenon as well.

\medskip
Let us now formulate  a more precise version of Theorem
\ref{theo:dense}:

 \begin{theoprime}\label{theo:tangencies}
 Let $(f_\la)_{\la\in \La}$ be a holomorphic family of moderately
 dissipative polynomial automorphisms of $\cd$
of dynamical degree $d\geq 2$.
Then parameters with homoclinic tangencies are dense in the
bifurcation locus.  
\end{theoprime}
 
To understand the  strategy of the  proof of this theorem, let us first  
review the one-dimensional result that parameters with preperiodic
 critical points are dense in the bifurcation locus. The classical proof of this fact, 
based on Montel's theorem \cite{levin}, 
does not seem to have  
an analogue in our context. 

Let us outline an argument that admits a generalization to dimension two.
If $\lo$ belongs to the bifurcation 
locus, then 
some periodic point changes type near $\lo$. In particular
 there exists $\la_1$ close to $\lo$ such that at $\la_1$, there is a periodic 
point $p$ whose multiplier crosses the unit circle at a rational parameter. The theory of parabolic implosion \cite{lavaurs, shishikura} 
describes how  the dynamics  in the basin $\mathcal{B}$ of the parabolic point can ``implode" for some parameters close to $\la_1$. 
In particular, under generic assumptions, and replacing $f$ by some iterate if needed,  for well chosen sequences $\la_n\cv\la_1$, 
$f^{n}_{\la_n}$ converges locally uniformly in $\mathcal{B}$ to some limiting holomorphic function $g:\mathcal{B}\cv\cc$, 
which we refer to as a {\em 
transit map}. 
Fix a repelling periodic point $q$, which necessarily persists as $q(\la)$ in the neighborhood of $\la_1$.
By a classical theorem of Fatou, there exists a critical point $c$ in  $\mathcal{B}$.  
The point is that it is actually possible to adjust the 
sequence $\la_n$ so that $g(c) = q$. From this  
we infer that for large $n$,  there exists $\la'_n$ close to $\la_n$, such that $f^n_{\la'_n}(c(\la'_n)) = q(\la'_n)$, 
which is precisely  the result that we seek.

\medskip

To prove Theorem \ref{theo:tangencies}, in the second part of the paper
 we 
design a two-dimensional generalization of this argument. In the dissipative regime, 
if some periodic point $p(\la)$ bifurcates at $\lo$, then one multiplier  of $p(\la)$   crosses the unit circle while 
the other  stays smaller than 1 (recall that the product of the multipliers is equal to the Jacobian). 
If furthermore $p(\lo)$ has a root of unity as multiplier,  it is said to be {\em semi-parabolic}, and we 
say that $p(\la)$ undergoes a {\em semi-parabolic bifurcation}. 
Then the proof is divided into two main steps:
\begin{itemize}
\itm  Step 1: prove the existence of ``critical points'' in the basins of semi-parabolic periodic points.
\itm Step 2: use ``semi-parabolic implosion" to make these critical points leave the basin under small perturbations of $\lo$, eventually creating tangencies.
\end{itemize}  

\medskip

The critical points in Step 1 are defined as follows. Let $f$ be a polynomial automorphism with a semi-parabolic periodic point $p$, 
which we may assume is fixed. Then $p$ admits a basin of attraction $\mathcal{B}$, which is endowed with a holomorphic 
{\em strong stable foliation}, whose leaves are characterized by the
property that points in the same leaf approach one another 
exponentially fast under iteration. Then by definition
 a critical point is  a point of tangency between the strong stable foliation in $\mathcal{B}$ and 
the unstable manifold of some saddle periodic point $q$. 

We obtain the following result.

 \begin{theo}\label{theo:critical}
 Let $f$ be a moderately dissipative  polynomial automorphism of $\cd$.
Assume that  $f$ possesses a semi-parabolic periodic point  
 with basin of attraction $\mathcal{B}$. Then for any saddle periodic
  point $q$, every component of $W^u(q)\cap \mathcal{B}$ contains a 
 critical point. 
 \end{theo}

Notice that this is precisely the place where  the assumption on the
Jacobian is required. 
Curiously, the proof relies on the classical theory of entire functions of finite
order in one complex variable. 
The same idea was then used by H. Peters and the second author \cite{lyubich peters} to obtain a nearly 
complete classification of periodic Fatou components for moderately dissipative polynomial automorphisms of $\cd$. 

\begin{rmknonb}
The classical theory of entire functions was first
applied to (one-dimensional) polynomial dynamics by Eremenko and
Levin \cite{EL}.  
\end{rmknonb}
 
The second step 
 relies on the construction of transit mappings in the context of semi-parabolic bifurcations. Semi-parabolic points are roughly 
classified according to the multiplicity of $f-\mathrm{id}$ at the periodic point under consideration. The theory of semi-parabolic 
implosion was recently developed by Bedford, Smillie and Ueda
\cite{bsu} who obtained a 
satisfactory  picture in the 
 multiplicity two case.  
In particular,  it generalizes a theorem of Lavaurs \cite{lavaurs},  thus obtaining a  precise 
description of  the   transit behaviour in this setting. 
However, these results depend on certain  explicit changes of
variables that do not readily extend  to the general case. 

In our situation we  
 have to deal with semi-parabolic points of   arbitrary multiplicity, 
so we need  to develop 
a more general  method. 
It was inspired by a chapter of the celebrated Orsay Notes by Douady and
Hubbard,   
cheerfully 
entitled  ``{\em un tour de valse}" \cite{tv}
(written by Douady and  Sentenac).

 To be specific, if $\lo$ is a parameter at which a semi-parabolic bifurcation occurs, replacing $f$ by some iterate if needed, 
 there exists a sequence of parameters $\la_n\cv\lo$     such that 
$f^{n}_{\la_n}$ converges in $\mathcal{B}$
to some holomorphic map $g:\mathcal{B}\cv\cd$. 
Notice that due to dissipation, $g$ has 1-dimensional image. 
An important phenomenon   
here is that  $g(\mathcal{B})$ 
needn't be contained in $\mathcal{B}$: indeed $f_{\la_n}$ shifts $\mathcal{B}$ slightly, 
 which is then amplified by iteration. 
 In this sense the limiting dynamics of $f_{\la_n}$ is richer than that of $f_0$.  
Though these transit mappings $g$ are not as explicit  as in  \cite{bsu}, 
they can still  be well controlled (see Theorem \ref{thm:tv}).  If now
 $q=q(\lo)$ is any   saddle point, and $c$ is a critical point in $W^u(q)$,   we can adjust the sequence $\la_n$ so that $g(c)\in W^s(q)$. It is then easy to find parameters  $\la'_n$ close to $\la_n$ for which $W^u(q(\la'_n))$ and $W^s(q(\la'_n))$ are tangent, thereby concluding the proof. 

\medskip

The plan of the paper is the following.   The first section  is devoted to some preliminaries on polynomial automorphisms of $\cd$. 
The  notion of branched holomorphic motion is explained  in detail in \S \ref{sec:branched}. 
In \S \ref{sec:stable}, we define the notion of weak $J^*$-stability, which is the direct analogue of the one-dimensional notion of 
$J$-stability and study the properties of weakly $J^*$-stable families. 
In \S \ref{sec:extension} 
we show that   a weakly $J^*$-stable family is also weakly stable on $J^+$, $J^-$, and $K$. In particular this justifies the use of the 
more general  ``weakly stable" terminology. 
We also prove that if a
dissipative family of polynomial automorphisms has persistently connected Julia set, then it is weakly stable 
(Theorem \ref{thm:connected}). This generalizes a well known result in dimension 1. 
The proof of Theorem \ref{theo:tangencies} occupies 
\S \ref{sec:semi parabolic} to \ref{sec:proof}. In \S \ref{sec:semi parabolic}, we recall some basics on semi-parabolic 
dynamics. The existence of critical points in semi-parabolic basins (Theorem \ref{theo:critical})
 is discussed in \S\ref{sec:critical}, which also includes some  
preparatory material on entire functions of finite order. A slight adaptation gives the existence of critical points in attracting basins. Details are given in Appendix \ref{sec:attracting}.    Semi-parabolic implosion  and transit mappings are studied in 
  \S \ref{sec:tv}, and finally in \S \ref{sec:proof} we assemble these results to prove
  Theorem \ref{theo:tangencies}.

\medskip

Throughout the paper 
we use the following notation: if $u$ and $v$ are two  real valued functions, we write
 $u \asymp v$ (resp $u\lesssim v$)
  if there exists a constant $C>0$ such that $\unsur{C}u\leq v\leq Cu$ (resp. $u\leq Cv$).  
  The disk  in $\cc$  of radius $r$ centered at 0 is denoted by
  $\dd_r$. Moreover, $\dd$ stands for $\dd_1$.
 Throughout the paper, $\La$   stands for a connected complex
manifold, which serves as a dynamical parameter space.

\begin{rmknonb}
The results of this paper were first announced  at the
Balzan-Palis Symposium on Dynamical Systems (IMPA, June 2012) 
 and at the Workshop on  Holomorphic Dynamical Systems  (Banff, July 2012). 
\end{rmknonb}
 
\medskip
\noindent {\it Acknowledgement.} 
We thank Alex Eremenko for many useful comments on the
Denjoy-Carleman-Ahlfors and Wiman Theorems,
and Eric Bedford for interesting discussions.  
We are also grateful to Serge Cantat and the anonymous referee for carefully reading the
manuscript and making many suggestions that improved the exposition.  
This work was  partly supported by the NSF,  the Balzan-Palis
Fellowship,  and the   ANR project ANR-13-BS01-0002.

  \section{Preliminaries}\label{sec:prel} 

In this section we recall some basics on the dynamics of polynomial automorphisms of $\cd$, 
and establish some preparatory results. 

\subsection{Basics}\label{basics sec}

Let  $f$ be a polynomial automorphism of $\cd$ 
with non-trivial dynamics. 
Non-trivial dynamics here means for instance   that $f$ has positive 
topological entropy, which then equals $\log d$, where 
$$
   d=\lim_{n\cv\infty} (\deg (f^n))^{1/n}
$$ 
is the {\em dynamical degree} of $f$. According to Friedland and Milnor \cite{fm} this happens if and only if $f$ is conjugate to a
composition of Hénon mappings 
$$
   (z,w)\mapsto (p(z) -b\, w, z).
$$ 

Let us recall the following basic dynamical objects and facts. 
The reader can consult \cite{hov,bs1,fornaess sibony,bls}  for details. 
  
\ssk\nin $\bullet$ 
$K^\pm$ are the  forward and backward {\it filled Julia sets}, that is, the   
 sets of points with bounded forward/backward orbits respectively.

\ssk \nin $\bullet$
 $U^\pm = \cd\setminus K^\pm$ are the forward and backward {\it basins of infinity}.

\ssk\nin  $\bullet$ 
$J^\pm = \fr K^\pm$  are the   forward and backward {\it  Julia sets}.
They can be also defined 
as the {\it sets of non-normality} for the families $\{ f^{\pm n} \}_{n\geq    0}$ respectively.
Note that in the dissipative case, $K^-$ has empty interior so  $J^- =K^-$.   
 
\ssk\nin $\bullet$
 $K:=K^+\cap K^-$ is the {\it filled Julia set} consisting of points
 whose {\it two-sided orbits} do not escape.

\ssk\nin $\bullet$
   $J := J^+\cap J^-$ is the {\it ``little''  Julia set} and 
$\bJ:=    J^+\cup
   J^-$ 
     is the {\it ``big'' } one.   In   the complement of the
   former,  at least one of the families, $\{ f^n\}_{n\geq 0}$ or $\{ f^n\}_{n\leq 0}$,  
  is normal. In the complement of the latter, the whole two-sided family 
  $\{ f^n\}_{n\in \mathbb Z}$ is.   

\ssk\nin $\bullet$  
$\Saddles$ is the set of saddle periodic points ({\it saddles}).
As usual, $W^s(p)$ and $W^u(p)$  stand for the {\em stable} and
{\it unstable} manifolds of a saddle\footnote{or for a more general periodic point whenever they exist.} $p$.
They are holomorphically  immersed complex lines $\cc\ra \cc^2$.

\ssk\nin $\bullet$
Given a saddle $p$, $H(p)$ denotes  the set of {\em homoclinic
  intersections} between  $W^u(p)$ and $W^s(p)$, while 
 $H^{\rm tr}(p)$ is the subset of {\it transverse} homoclinic
 intersection.  For any $p\in \Saddles$ 
the closure of either of these sets coincides with $J^*$
\cite[Prop. 9.8]{bls}, so according to the general dynamics  terminology,
$J^*$ is the {\em homoclinic class of $f$}.  

We do not devote a special notation for the set  of {\em heteroclinic}
intersections,
but   use the following abbreviated terminology:
{\em s/u intersection} is   a shorthand for  
``homoclinic or  heteroclinic intersection of stable and unstable manifolds of saddle 
periodic orbits''.

\ssk\nin $\bullet$
  $J^*$ is the closure of $\Saddles$. 
It is contained in $J$, and it is an open problem (posed by Hubbard) whether $J=J^*$.

\ssk\nin $\bullet$
  $S^-= J^-\sm K^+$ and $S^+=J^+\sm K^-$. 

\ssk\nin  $\bullet$
 $G^\pm$ are  the forward and backward  {\it Green functions}.
Their dynamical meaning is that of     {\it escape rate}  functions:
$$
    G^\pm (z) = \lim_{n\to +\infty} \frac 1{d^n} \log^+ \| Df^{\pm n} (z) \|.
$$
Moreover, they  have the following properties:

\ssk --  $G^\pm$ are non-negative and  vanish on $K^\pm$ respectively;

\ssk -- $G^\pm$ are 
  pluri-subharmonic on the whole $\cc ^2$, and pluri-harmonic on  $U^\pm $ respectively; 

\ssk --
they satisfy the functional equation $G^\pm (f^{\pm 1}  z) = d\, G^\pm (z)$.

\ssk\nin  $\bullet$
   $\varphi^\pm$ are the forward and backward {\it B\"ottcher
     functions}.  
They are well defined and holomorphic in appropriate sectors in $U^\pm$  near
    infinity and satisfy $\log  |\varphi^\pm| = G^\pm$. Moreover, they
    satisfy the {\it B\"ottcher functional equations} 
$$
   \varphi^\pm    (f^\pm z) = (\varphi^\pm (z))^d.
$$ 
Note that by means of this equation,  $\varphi^\pm$  extend anaytically to
the whole basins $U^\pm$ as 
multi-valued functions with a single-valued  absolute value $>1$ (namely $\exp(G^\pm)$ ). 

\ssk\nin  $\bullet$
Though the   B\"ottcher functions do not coherently extend 
to the whole basins $U^\pm$,
   their level sets do (by means of the  dynamics), defining holomorphic {\it
     B\"ottcher $\cc$-foliations}%
\footnote{meaning that their  leaves are conformally equivalent to $\cc$}
 $\FF^\pm$ in   $U^\pm$. 

\ssk\nin  $\bullet$
Stable and unstable {\it Green currents}   $T^\pm  : = dd^c G^\pm $.
They are supported on the forward and backward Julia sets $J^\pm$ respectively
and satisfy the dynamical functional equations $f^* T^\pm = d^{\pm 1}$. 
Moreover, {\it the  unstable manifold $W^u(p)$ of any saddle $p$ is
equidistributed with respect to} $T^-$, while the stable manifolds
$W^s(p)$ are equidistributed with respect ot $T^+$ \cite{bs1,fornaess sibony}.
It follows that any $W^u(p)$ is dense in $J^-$,
while any $W^s(p)$ is dense in $J^+$.  

\ssk\nin $\bullet$ 
The {\it measure of maximal entropy} $\mu= T^+ \wedge T^-$.
By \cite{bls2}, {\it saddles are equidistributed with respect to} $\mu$. 
Moreover,   $\operatorname{supp}\,  \mu = J^*$.

\subsection{Families of compositions of H\'enon maps}
 We will be interested in holomorphic families $(f_\lambda)_{\la\in
   \La}$ of polynomial automorphisms,   parameterized by some complex
 manifold $\Lambda$.  
We put a subscript $\la$ to denote the  parameter dependence of the 
corresponding objects, e.g.,  $J_\la$, $\mu_\la$, etc.

\medskip

The following proposition, which might be known to some experts
 (see e.g. \cite{furter}, and also \cite[Thm 1.6]{xie}  for the birational case),
asserts that as far as we are interested in properties of $f_\la$ which are typical with respect to $\la$, it is not a restriction to assume 
that the $f_\la$ are  products of  Hénon mappings. 

\begin{prop}\label{prop:henon}
Let $(f_\la)_{\la\in \La}$ be a holomorphic family of polynomial automorphisms in $\cd$, parameterized by a connected complex 
manifold. There exists a Zariski open set $\La'\subset \La$ and an integer $d\geq 1$ 
such that for $\la\in \La'$, $f_\la$ has dynamical degree $d$. 

Furthermore, if $d\geq 2$, locally in $\La'$ we can write 
$$f_\la = \varphi_\la^{-1} \circ h_\la^1 \circ \cdots \circ h_\la^m \circ \varphi_\la$$ where $(\varphi_\la)$ is a  
polynomial automorphism  and $(h_\la^i)_{i=1, \ldots , m}$  are Hénon mappings of degree $d_i$, with $\sum d_i = d$,  
all depending holomorphically on $\la$. 
\end{prop}

To prove the proposition we need to recall some ideas from \cite{fm}. 
Fix coordinates $(z,w)$ on $\cd$. We denote by $E$ the group  of automorphisms
preserving the family of lines $\set{w= C}$. 
Such automorphisms are of the form 
$(z,w)\mapsto(\alpha z+ p(w), \beta  w +\gamma)$ and
 will be referred to as {\em elementary}. 
(More generally, an automorphisms is elementary if it can be put in this form in some system of coordinates $(z,w)$.)
The group of affine automorphisms  will be denoted by $A$. It turns out that 
 the group $\mathrm{Aut}(\cd)$ of polynomial automorphisms of $\cd$ is the free product of $A$ and 
$E$, amalgamated along their intersection $S:=A\cap E$, that is, every $f\in \mathrm{Aut}(\cd)\setminus  S$, 
 can be written as a composition 
$f=g_k\circ \cdots \circ g_1$, where $g_i$ belongs to $A\setminus S$ or $E\setminus S$. 
This decomposition is unique, up to simultaneously replacing $g_i$ by $g_{i}\circ s$ and $g_{i-1}$ by $s^{-1}\circ g_{i-1}$, for some 
$s\in S$.  
The degree of such a composition is to equal $\prod \deg(g_i)$ (of course only elementary automorphisms contribute to the degree). 
 One has that $\deg(f^n) = (\deg f)^n$ if and only if $f$ is cyclically reduced, that is the extreme factors  $g_1$ and $g_k$ 
belong to different subgroups $A$ and $E$. In general, write
$$f= a_m\circ e_m\circ a_{m-1}\circ e_{m-1} \circ \cdots \circ e_1\circ a_1, \text{ with } a_i\in A\setminus S \text{ and } e_i\in E\setminus S, $$
with possibly $a_m$ or $a_1$ equal to the identity.  
We define the multidegree of $f$ as $(d_m, \ldots, d_1)$ where $d_i  = \deg(e_i)$.  

\begin{proof}
It is clear that there exists a Zariski open set $\La_0\subset \La$ where the degree is constant, say, equal to $d'$. If $d'=1$ there is nothing to prove so assume $d'\geq 2$. 
A  theorem due to Furter asserts that in a connected holomorphic family of polynomial automorphisms, the degree is constant if and 
only if the multidegree is constant \cite[Cor. 3]{furter}. Hence there exists an integer $m$ such that 
 for every  $\la\in \La_0$   we can write 
$$f_{\la}= a_{m, \la} \circ e_{m,\la} \circ a_{m-1, \la}\circ e_{m-1, \la} \circ \cdots \circ e_{1,\la} \circ a_{1, \la}.$$ 

We claim that the factors $a_{i, \la}$ and $e_{i,\la}$ may be chosen to depend holomorphically on $\la$. 
This is not obvious since they are not unique.
 We can deal with the extreme factors $a_m$ and $a_1$ 
as in \cite[Lemma 2.4]{fm}, by observing that the coset space $A/S$ is isomorphic to $\pu$ and 
that there is a well defined      mapping 
$f_\la\mapsto (a_{1, \la}^{-1} S, a_{m, \la} S)\in \pu\times \pu$. In a   more explicit fashion, this mapping may be expressed as 
 $f_\la\mapsto  (I (f_\la), I (f_\la^{-1}))$, 
where $I(f)$ is the indeterminacy set of $f$ viewed as a rational 
mapping on $\pd$, and $\pu$ is identified to the line at infinity.  Since $f_\la$ depends holomorphically on $\la$, so do $a_{1, \la}^{-1} S$ and $a_{m, \la} S$, hence absorbing some of the  $S$ factors in 
$e_{m,\la}$ and $e_{1,\la}$ if necessary, we infer that $a_{1,\la}$ and $a_{m,\la}$ depend holomorphically in $\la$. 
Thus we are left to proving that if $f_\la$ is of the form 
$f_{\la}=  e_{m,\la} \circ a_{m-1, \la}\circ \cdots \circ e_{1,\la}$, then the factors may be chosen to depend holomorphically on $\la$. 
By \cite[Lemma 2.10]{fm}, $f_\la$ admits a {\em unique} decomposition of the form 
$$f_\la = (\hat{s}_{m,\la} \circ \hat e _{m,\la})\circ t \circ \hat e_{m-1, \la }\circ \cdots \circ t \circ \hat e_{1, \la},$$ where 
$\hat{s}_{m,\la}$ is affine with diagonal linear part, $\hat{e}_i$ is of the form 
$(z,w)\mapsto (z+p_i(w), w)$, with $p_i(0)=0$ and $t(z,w) = (w,z)$. By uniqueness, the factors of this decomposition depend holomorphically on $\la$ (see \cite[p.909]{furter} for details) and our claim is proven.  

\medskip

From this point it is clear that the set of parameters such that $f_\la$ is not cyclically reduced is Zariski closed in $\La_0$. Indeed, conjugating $f_\la$ by $a_{1,\la}$ we obtain an expression of the form 
  $$a_{1,\la}\circ a_{m,\la}\circ e_{m, \la} \circ \cdots \circ e_{1,\la},$$ which is not cyclically reduced if and only if 
  $a_{1,\la}\circ a_{m,\la}\in S$, which is an analytic condition. If so, we absorb $a_{1,\la}\circ a_{m,\la}$ into $e_{m,\la}$ and infer that the resulting word is not cyclically reduced iff $e_{1,\la}\circ e_{m,\la} \in S$, and so on. Iterating this process we obtain a Zariski open set $\La'$ such that if $\la\in \La'$, $f_\la$ is cyclically reduced, and the first part of the proposition is proved. 
  
  To establish the second assertion, in $\La'$ 
  we conjugate $f_\la$ as above to make it cyclically reduced and of the form 
  $$(t\circ e_k)\circ \cdots \circ (t\circ e_1).$$ Then we 
  argue as in \cite[Theorem 2.6]{fm}
that a mapping of the form $t\circ e_i$ is affinely conjugate to a Hénon mapping $(z,w)\mapsto (\delta_iz + p_i(w), z)$, 
which is unique up to finitely many choices if $p_i$ is chosen to be monic and centered. 
\end{proof}



 \part{Holomorphic motions and stability}

\section{Branched holomorphic motions}  
 \label{sec:branched}

Recall the notation  $\La$ for the parameter domain, which is a connected complex
manifold. 
It will often be  pointed by a  {\it base point} $\la_0\in \La$. In this case,
if we have a family of objects parametrized by $\La$,     the base
objects will often be simply  labeled with $0$, e.g.,  $f_0\equiv f_{\la_0}$,
$J_0\equiv J_{\la_0}$, etc. 

Recall that a {\em  holomorphic motion} of a set $A$ in $\cc^d$ over 
$\Lambda$ 
 is a family of mappings $h_\lambda:A\cv\cc^d$ such that                                                
\begin{itemize}
\itm for fixed $a\in A$, $\lambda\mapsto h_\lambda(a)$ is holomorphic;
\itm for fixed $\lambda\in \Lambda$, $a\mapsto h_\lambda(a)$ is injective;
\end{itemize}
Holomorphic motions are often {\em pointed} by 
assuming that $h_\lo$ is the identity mapping.  

The {\it total space} of a 
holomorphic motion of $A$ over $\La$ is a family of
 disjoint holomorphic  graphs over the first coordinate in
 $\La\times\cd$, which we 
endow  with the  topology of uniform convergence on compact subsets of $\La$.
Let us relax this notion as follows:

\begin{defi}
A branched holomorphic motion (abbreviated as ``\BHM'' in the following)   
 over $\La$ is a family   of holomorphic graphs 
over the first coordinate in $\La\times\cd$. 
\end{defi}

\begin{rmk}
This definition bears some similarity with the notion of ``analytic multifunction", which was studied by S\l odkowski, and others. 
In particular it appears in \cite{slodkowski} under the name of ``locally trivial analytic multifunction". 
\end{rmk}

With 
$\mathcal{G}$ being such a family, we let 
 $\mathcal{G}_\la =  \set{ \gamma(\la), \ \gamma\in \mathcal{G}}$ be
 the {\it sections} of the total space,  and we say that they 
 $\mathcal{G}_\la$  ``move under the branched motion 
$\mathcal{G}$''.  We let $\Pi_\la: \GG\ra \GG_\la$ be the natural
projection, which is obviously continuous.

\msk
As in the one-dimensional setting, we will use  extension 
properties of (branched) holomorphic motions.
The classical $\lambda$-lemma  asserts that  a holomorphic motion of 
$A\subset \cc$ extends to $\overline A$ and is automatically
continuous. 
These virtues come from {\it Montel normality} of the family $\GG$ of disjoint
graphs $\La \ra \cc$  and from the {\it Hurwitz Theorem} that ensures that
disjointness is inherited by the closure $\overline \GG$.  
Of course neither of  these statements is true in higher dimension, 
which motivates  our use of branched motions as well as the following
definition: 

\begin{defi}
  A branched holomorphic motion $\GG$ in $\cc^2$ over $\La$   is called normal if
  $\GG$ is a normal family of graphs $\gamma: \La\ra \cc^2$. 
\end{defi}

Recall that {\it normality} means that from any sequence of graphs   
$\gamma_n$ we can extract a subsequence $\gamma_{n_k}$  which is either locally
bounded (and hence locally equicontinuous) or else  $\gamma_{n_k}\to \infty$ locally uniformly.
In particular, this is the case if the whole family $\GG$ is locally
uniformly bounded, or more generally, if the sections $\GG_\la$ belong
to a Kobayashi hyperbolic domain $U_\la \subset \cc^2$ that depends 
lower semi-continuously on  $\la$
(i.e.,  any compact subset $Q\subset U_{\la}$ is contained in
$U_{\la'}$ for all $\la'$ sufficiently close to $\la$).

With these definitions in hand, the following lemma is obvious.

\begin{lem}\label{lem:extension}
If $\mathcal{G}$ is a normal
branched holomorphic motion   over $\La$, then so is
$\overline{\mathcal{G}}$. 
\end{lem}
 
Recall that the {\it Hausdorff topology} on the space of subsets of
$\cc^2$ is defined by the following  basis of neighborhoods: 
$  {\mathcal U}_{r, \eps} (A)  $ consists of subsets $X\subset \cc^2$ such
that the set $X\cap \dd_r^2$  is contained in the $\eps$-neighborhood of
$A\cap \dd_r^2 $,  and the other way around.  

\begin{lem}\label{Hausdorff cont}
If $\GG$ is a normal \BHM then the sections
$\GG_\la$ depends continuously on $\la$ in the Hausdorff topology. 
\end{lem}

\begin{proof}
  This  easily  follows from the local equicontinuity of the truncated families 
\begin{equation}\label{GG(r)}
  \GG_\la (r, \delta) : = \{ \gamma\in \GG:\ \gamma(\la)\in \dd_r^2  \text{ for } \la\in \dd_{1-\delta}\} .
\end{equation}
\end{proof}

Let us say that a BHM  $\GG$ is {\it unbranched} at some $\la\in
\La$ if the natural projection $\GG\ra \GG_\la$ is injective.
It is {\it unbranched along} $\gamma_0 \in \GG$ if $\gamma_0$ does not cross
any other graph $\gamma\in \GG$. 

\begin{lem}\label{continuity}
  Let $\GG$ be  a normal \BHM in $\cc^2$  over $\La$ .
If  $\overline\GG$ is
  unbranched at some parameter $\la_0\in \La$ then 
  the mappings 
$h_\la: {\overline \GG}_{\la_0} \ra {\overline \GG}_\la$ defined by 
$$
     \gamma(\la_0)\mapsto \gamma(\la), \text{ for } \gamma\in \overline \GG, 
$$
are continuous and
  depend holomorphically on $\la\in \La$. 
\end{lem}

\begin{proof}
  The last statement is obvious from the definitions. 
To prove continuity of the $h_\la$,  let us consider the functional space
$ {\overline \GG}_0 (r, \delta) \equiv \overline \GG_{\la_0} (r, \delta)$ (defined in (\ref{GG(r)})), which is compact.  
By the  unbranching assumption, 
the natural projection 
$$
   \Pi_0 : \overline\GG_0 (r, \delta) \ra \overline\GG_0 \cap \dd_r^2
$$
 is bijective and hence is a homeomorphism. It follows that the maps
$h_\la= \Pi_\la \circ \Pi_0^{-1}$ are continuous. 
\end{proof}

\begin{cor}\label{homeos}
  Under the circumstances of Lemma \ref{continuity},
if the motion of $\overline \GG$ is unbranched, then the maps
$h_\la: {\overline\GG}_{\la_0} \ra {\overline  \GG}_\la$ are homeomorphisms.  
\end{cor}

More generally, let us say that a normal holomorphic motion $\GG$
is {\it strongly unbranched} if
for every $\gamma \in \GG$, $\overline\GG$ is unbranched along $\gamma$ 
(notice that  the whole $\overline \GG$ is allowed to be branched).

We say that a holomorphic motion is {\it continuous} if all
the maps $h_\la: \GG_0\ra \GG_\la$, $\la \in \La$, are homeomorphisms.

\begin{lem}\label{strong unbranching}
  A normal holomorphic motion $\GG$ is continuous iff it is 
strongly unbranched.
\end{lem}

\begin{proof}
Assume $h_\la$ is discontinuous for some $\la\in \La$. Then for some $\gamma\in \GG$ there exists
a sequence $\gamma_k\in \GG$ such that $\gamma_k(\la_0) \to \gamma(\la_0)$ while
$\| \gamma_k(\la) -\gamma(\la) \| \geq \de>0$. Since $\GG$ is normal, we can pass to a limit
$\gamma_\infty\in \overline \GG$ such that $\gamma_\infty(\la_0)= \gamma (\la_0)$
while $\gamma_\infty(\la)\neq  \gamma (\la)$. Thus, $\overline \GG$ is branched at $\gamma$.
The same argument shows that discontinuity  of $h_\la^{-1}$ implies branching of $\overline \GG$ at some
$\gamma\in \GG$. 

The reverse assertion 
easily follows from Lemma \ref{continuity}. 
\end{proof}

\msk
Next, let us formulate a simple 
consequence of the classical one-dimensional $\la$-lemma:

\begin{lem}\label{entire curves}
    Let $\psi_\la: \cc \ra \cc^2$, $\la\in \dd $, be a holomorphic
    family of injectively immersed entire curves.
 Let $h_\la: A_0\ra \cc^2$, $\la\in \dd $, be a holomorphic motion in $\cc^2$
such that $A_\la:= h_\la (X_0) \subset \psi_\la(\cc)$.  Then it extends
to a holomorphic motion of $\psi_0(\cc) $ with values in $ \psi_\la(\cc)$. 
Moreover, locally in $\la$ (independently of the particular motion over
$\La$), there is a canonical extension which depends only on the
images $\psi_\la(\cc)$ but not on the  particular choice of the parametrizations
$\psi_\la$.  
\end{lem}

\begin{proof}
  Apply the S\l odkowski $\la$-lemma \cite{slodkowski} to the holomorphic  motion in $\cc$: 
$$
       \psi_\la^{-1} \circ h_\la \circ \psi_0: \  \psi_0^{-1}( A_0)\ra
       \psi_\la^{-1} (A_\la), \quad \la\in \dd.
$$
Moreover, locally in $\la$, there is the  {\it canonical} ``harmonic'' extension due to Bers
and Royden \cite{bers royden}  which is equivariant under complex affine
changes of variable, so it is independent of the particular
choice of the $\psi_\la$. 
\end{proof}

We will refer to the above canonical extension as the {\it Bers-Royden motion}.
It implies the following {\it foliated $\la$-lemma}
(first considered in \cite{buzzard verma}). 

Let us say that a family of holomorphic $\cc$-foliations $\FF_\la$ depends
holomorphically on $\la\in \La$ if the local defining functions
$\phi_\la$ for the $\FF_\la$ can be selected holomorphic in $\la$. 
Given a set $A$ and a $\cc$-foliation $\FF$ we define the  
 {\it leafwise closure}  $\cl_\FF A$ as $\bigcup_L (\cl_L  (A\cap L)$, where the union is
taken over all the leaves $L$ of $\FF$ and the closure $\cl_L$ is taken in the
intrinsic topology of the leaf.

\begin{cor}\label{foliated lambda-lemma}
  Let $h_\la: A_0\ra \cc^2$ be a  holomorphic motion in
  $\cc^2$, and let $\FF_\la$ be a holomorphic family of $\cc$-foliations
supported on open sets $U_\la\subset \cc^2$ 
containing  $ A_\la$. Then $h_\la$ extends to a
holomorphic motion of the leafwise   
closure  $\cl_{\FF_0}  A_0$. Moreover, locally in $\la$,
it further extends to the motion of  the whole   leaves of $\FF_\la$ that 
meet $A_\la$.   
\end{cor}

\begin{proof}
   The extension to the leafwise closure is obvious by the simplest
   one-dimensional version of the $\la$-lemma. Further extension
   comes  from Lemma \ref{entire curves} (it is important that
   this extension is canonical).
\end{proof}

\section{Weak $J^*$-stability} \label{sec:stable}

\subsection{Substantial families}

{\em 
From now on, $(f_\la)_{\la\in \La}$ will stand for  a holomorphic family of   
polynomial automorphisms of $\cc^2$ of dynamical degree $d\geq 2$ over a parameter domain $\La$,
which is a connected complex manifold. }

We will  often require an additional --presumably superfluous-- assumption. We say that 
a holomorphic family of polynomial automorphisms is {\em substantial} if 
\begin{itemize}
\itm either all members of the family are dissipative
\itm or for any 
periodic point with eigenvalues $\alpha_1$, $\alpha_2$,
no relation of the form $\alpha_1^a\alpha_2^b=c$,  holds persistently in parameter 
space, where  $a$, $b$, $c$ are complex numbers and $\abs{c}=1$. 
\end{itemize}
As an example, 
any open subset of the family of all polynomial automorphisms 
of  dynamical degree $d$ is substantial \cite[Theorem 1.4]{bhi}. On the other hand, a family of 
conservative polynomial automorphisms is not.

\subsection{Stability and Newhouse phenomenon} 

A family $(f_\la)_{\la\in \La}$ induces a  {\em fibered map } 
\begin{equation}\label{fibered map}
  \widehat{f}: \La \times \cc^2 \ra \La\times \cc^2 ,\quad     \widehat{f}:(\la,z)\mapsto (\la, f_\la(z)),
\end{equation}
which in turn, induces an action on the space  of  graphs $(\la,
\gamma(\la))_{\la\in \La}$ of holomorphic functions $\gamma: \La \ra \cc^2$.  
A branched holomorphic motion $\mathcal{G}$  over $\La$
is called {\em equivariant} if 
 $\widehat {f} (\mathcal{G}) = \GG$. 
 

\begin{defi}\label{def:stable}        
A holomorphic family $(f_\la)_{\la\in \La}$ 
 of polynomial automorphisms of $\cd$ 
is called weakly $J^*$-stable  if the sets  $J_\la^*$ move under an 
equivariant%
\footnote{Later on we will see that equivariance is automatically satisfied.}  
 \BHM. 
A map $f_{\la_0}$ and the corresponding parameter $\la_0\in \La$  are
called weakly $J^*$-stable if the family $(f_\la)$ is weakly $J^*$-stable over a
neighborhood $\La_0\subset \La$ of $\la_0$, otherwise we say that a bifurcation occurs at $\lo$.    
\end{defi} 

If in this definition we require that the motion in question is
{\it unbranched} then we obtain the usual notion of 
{\it  $J^*$-stability}. 
Given any dynamical set $X_f$ (e.g., $K_f$ or $J^\pm_f$), 
we can define {\em (weak)   $X$-stability} in the same way. 

\begin{rmknonb}
 Note that we do not assume that the \BHM in question is normal. It
 turns out that for all dynamical sets considered in this paper (e.g.,
 $X= \widehat J$), the \BHM can be selected to be normal. (Of course,
in case of $X= J^*$ it is automatically so.) 
\end{rmknonb}


\msk
The following theorem is very much in the spirit of one dimensional dynamics 
\cite{mss, lyubich}. It shows that   weak $J^*$-stability is a reasonable notion of stability for polynomial automorphisms.

\begin{thm}\label{thm:equiv}
Let $(f_\la)_{\la\in \La}$ be a substantial family of polynomial automorphisms of $\cd$ of dynamical degree $d\geq 2$. The 
following are equivalent:
\begin{enumerate}
\item[\rm { (i) } ]  
The family  $(f_\la)$ is weakly $J^*$-stable.
\item[\rm { (ii) } ] Every  periodic point  stays of constant type (saddle, attracting, repelling, indifferent) throughout the family.
\item[\rm { (iii) } ] $J_\la^*$ moves continuously in the Hausdorff topology.
\end{enumerate}
If furthermore $(f_\la)$ is dissipative,  the following two conditions are equivalent, and imply the previous ones:
\begin{enumerate}
\item[\rm { (iv) } ] The number of attracting cycles is (finite and) locally constant.
\item[\rm { (v) } ] The period of attracting cycles  is locally uniformly bounded.
\end{enumerate}
\end{thm}  


Most of the proof of this  theorem is contained in \S\S \ref{subs:motion} and \ref{subs:special} below 
(see in particular Proposition \ref{prop:extending correspondence}). The proof will be completed in \S \ref{subs:proof}.

\begin{cor}\label{attracting par}
Let $(f_\la)_{\la\in \La}$ be a 
family of dissipative polynomial automorphisms of $\cd$.
Then any bifurcation parameter $\la_0\in \La$ can be approximated
by a parameter $\la\in \La$ such that $f_\la$ has an attracting cycle. 
\end{cor}

\begin{proof}  
 By item (ii) of the theorem, $\la_0$  can be a approximated by a
 parameter $\mu_0$ such that $f_{\mu_0}$  has an non-persistently  indifferent periodic
 point $p_0$ with multipliers $|\alpha_1| < |\alpha_2|=1$.
Such a point can be perturbed to an attracting one.
\end{proof} 

The following consequence follows exactly as in dimension 1:

\begin{cor}\label{cor:opendense}
Let $(f_\la)_{\la\in \La}$ be a  
family of dissipative polynomial automorphisms of $\cd$.
If the number of attracting cycles is locally uniformly bounded on $\La$, then the 
locus of weak $J^*$-stability is open and dense. 
\end{cor}

\begin{proof}  
  By the above corollary, any bifurcation parameter $\la_0\in \La$ can
  be perturbed to a parameter $\la_1$ such that $f_{\la_1}$ with an
  attracting cycle. If $\la_1$ is a also a bifurcation parameter then
  for the same reason, it can be further perturbed to a parameter $\la_2$
  with two attracting cycles, and so on. Since the number of attacting
  cycles is locally uniformly bounded, this process must terminate, hence
  producing a stable  parameter $\la_n$ approximating $\la_0$.  
\end{proof}

In particular we have the following nice corollary in the spirit of the Palis conjectures. 
We say that $f_\la$ is a {\em Newhouse automorphism}  if it possesses  infinitely many sinks. 

\begin{cor}\label{cor:palis}
Let $(f_\la)_{\la\in \La}$ be a  
  family of   dissipative polynomial automorphisms of $\cd$. Then 
$$\set{\text{locally weakly }J^* \text{-stable parameters}} \cup \set{\text{Newhouse parameters}}$$ is dense in $\La$.
\end{cor}

\begin{proof}
Let $B\subset\La$ be the open set where the number of attracting cycles is locally 
uniformly bounded. By Corollary \ref{cor:opendense}, weak $J^*$-stability is dense 
in $B$. Now in $B^c$, the set 
$$U_m=\set{\la, f_\la \text{ possesses at least } m\text{ attracting cycles}}$$ is 
relatively open and dense. We conclude by Baire's Theorem.
\end{proof} 

According to the work of Buzzard \cite{buzzard}, it is known that 
the Newhouse region, i.e. the closure of the set of Newhouse parameters, has non-empty
interior in the space of polynomial automorphisms of sufficiently high degree. 
On the other hand,  it is an open question  whether weak $J^*$ 
stability implies that there are only finitely many sinks 
(i.e. whether conditions  (i)-(v) in Theorem \ref{thm:equiv} are equivalent).
(Of course, the Palis Conjecture would imply that this is the case.)  

\medskip

It is also worthwhile to state the following result which will follow from the proof of   Theorem  \ref{thm:equiv} 
(see Proposition \ref{prop:extending correspondence} below).

\begin{cor}\label{cor:tgcies}
Let  $(f_\la)_{\la\in \La}$ be a weakly stable 
 substantial family of polynomial automorphisms of $\cd$ of dynamical degree $d\geq 2$. Then:   
 \begin{itemize}
\itm  the \BHM  of the set $J^*$ 
 is unbranched over the set of periodic,  homoclinic and heteroclinic points;
\itm homoclinic and heteroclinic tangencies are persistent.
\end{itemize}
\end{cor}

An important question that is left open after this analysis  is whether branching  can actually occur. 
Indeed, while the $\la$-lemma clearly fails for general two-dimensional holomorphic motions, 
we do not know any instance of branching in the  dynamical context, or even  any mechanism that may lead to it.   
Thus, it is tempting to believe that weak $J^*$-stability actually implies $J^*$-stability. 
(Again, the Palis Conjecture would imply that this must be true
generically.)  

\subsection{Normality of motions in $S_\la^\pm$}

 
 According to Proposition \ref{prop:henon}, $(f_\la)$ is conjugate to a holomorphic family of composition of Hénon mappings. From 
 this it easily follows that the sets $K_\la$ are locally uniformly bounded in $\cd$. In particular, if $\mathcal{G}$ is a \BHM 
 such that $\mathcal{G}_\la\subset K_\la$ for all $\la$, then it is normal.
The next lemma  shows that it is also true if  $\mathcal{G}_\la\subset S^\pm_\la$.  This will be used in \S \ref{subs:extension}.

  \begin{lem}\label{lem:normal}
Let $(f_\la)_{\la\in \La}$ be a family of polynomial automorphisms  of dynamical degree $d\geq 2$.
 Then any family of holomorphic mappings $\gamma:\La\cv \cd$ such that for every $\la\in\La$, 
$\gamma(\la) \in S^+_\la \cup  S^-_\la$ 
is normal. 
\end{lem}

We will make use of the following lemma, which is known as the {\it Zalcman Renormalization Principle} 
 (see \cite{zalcman} and Berteloot \cite{berteloot} 
 for the version   that we state here).
 
 \begin{lem}\label{lem:zalcman}
 Let $M$ be a compact complex manifold and $(g_n)_{n\geq 1}$ be a sequence of holomorphic mappings from the unit disk to $M$. 
 If $(g_n)$ is not a normal family at $z_0\in \dd$ then there exists a sequence $(z_n)$ converging to $z_0$ and a sequence of 
 scaling factors 
 $r_n >0 $ converging to 0 such that (after possible extraction) the sequence of holomorphic mappings 
 $\zeta \mapsto g_n(z_n + r_n \zeta)$ converges uniformly on compact sets
   to a non constant entire map $g:\cc\cv M$. 
   \end{lem}

\begin{proof}[Proof of Lemma \ref{lem:normal}] 
It is no loss of generality to assume that $\La$ is the unit disk in 
$\cc$.  Moreover, since the sets 
$$
     \widehat S^\pm:=  \{ (\la, z)\in \La\times \cc^2:\,  z\in  S^\pm_\la \}
$$ 
are relatively open in   their union, each graph $\gamma\in \GG$ is
fully contained  in one of them.  Then we can assume without loss of
generality that all the graphs  $\gamma\in \GG$ are contained in one
of these sets.
For definiteness, let it be $\widehat S^+ $. 
 
Assume by contradiction that the   family of mappings given  in the statement is not normal. 
Then by the Zalcman  Lemma \ref{lem:zalcman}
  there exist  a sequence of holomorphic disks  $\gamma_n: \Lambda \rightarrow \cc^2$ 
such that  $ \gamma_n(\la) \in S^+_\la $, 
a sequence of  parameters $\la_n \cv\la_\infty\in \La  $, 
and a sequence 
$r_n\cv 0$ such that rescaled holomorphic disks $\gamma_n (\la_n+ r_nz) $ converge 
to a non-constant entire curve 
$\zeta: \cc \cv \overline { S^+_{\la_\infty} }$. 
Since     the Zalcman Lemma requires the target manifold to be compact, 
   the closure here is taken in $\mathbb{CP}^2$. Observe that     that 
in $\mathbb{CP}^2$ we have $\overline{J^+} = J^+\cup \set{I^+}$, where $I^+=I^+_\la$ is a single point at infinity.
 
Consider now a sequence of positive harmonic functions $H_n$ defined by 
$$H_n( z)=   G^-_{\la_n+ r_nz}(\gamma_n (\la_n+ r_nz)).$$ 
A first possibility is that the $H_n$ diverge uniformly to $+\infty$. 
Then $\zeta$ would take its values in $I^+$, which is absurd. 
Hence   the $H_n$ are locally uniformly bounded 
and converge to $G^-_{\la_1}(\zeta(z))$. But  a non-negative positive harmonic function on $\cc$ must be  constant. 
On the other hand,  $J^+\cap \set{G^- = c}$ is compact, 
so  again we  arrive at a contradiction. 
\end{proof}

Though normality is what we need, let us also make a slightly stronger statement: 

\begin{lem}\label{Kob hyp}
 Let $f$ be a product of Hénon mappings. 
For any $R>0$,
  the domain 
$$ 
 \Omega= U^+ \cap \big( \{\max(|x|,|y|) <R\} \cup \{ |y|<|x| \} \big) 
$$ 
is Kobayashi hyperbolic (and similarly for $U^-$).
\end{lem}

\begin{proof}
Let us consider a domain
$$
    Q = \{ z = (x,y) \in U^+ :\  |y| < R\,  | \varphi^+ (z) | \},
$$
where $\varphi^+$ is the forward B\"ottcher function. 
It is well defined since $|\varphi^+| $ is such, and it contains $\Omega$, by increasing $R$ slightly if needed.
Let $\widetilde U^+$  be the covering of $U^+$ that makes the
B\"ottcher function $\varphi^+$ well defined.  Then
$$
          \widetilde Q = \{ | Y | < R\, |\Phi |^+ \},
$$
where  capitals mean the lifts to $\widetilde U$.
Then $Z \mapsto ( \Phi^+(Z),  Y / \Phi^+ (Z))$
maps $\widetilde Q$ onto the bidisk $(\cc \setminus \dd) \times \dd_R$, with discrete fibers. 
Since the latter is hyperbolic, so are $\widetilde Q$, $ Q$, and $\Omega$ (see \cite[Prop. 1.3.12 and 3.2.9]{kobayashi}).
\end{proof}

\subsection{Motion of saddle and heteroclinic points} \label{subs:motion}


Let us start with a  result that shows that any branched holomorphic
motion $\GG_\la\subset K_\la$ is strongly
unbranched (and hence continuous)  on  hyperbolic sets. 
In particular, it is strongly unbranched at saddles and heteroclinic points. 

\begin{lem}\label{lem:extension/expansion}
Let $\mathcal{G}$ be a 
\BHM  over $\La$, 
such that for every $\la\in \La$, $\mathcal{G}_\la\subset K_\la$ 
or $\mathcal{G}_\la\subset  S^-_\la\cup   S^+_\la$.
 Assume that  $(\gamma_k)$ is a sequence of graphs in $\G$ such that for some $\la_0\in\La$, $\gamma_k(\la_0)
\rightarrow p(\la_0)$  as $k\cv\infty$, where $p(\la_0)$ belongs to some uniformly hyperbolic invariant compact set $E_\lo$.

 Then there exists a unique holomorphic continuation $(p(\lambda))_{\la\in \La}$ 
of $p(\la_0)$  such that $\gamma_k(\la)\cv p(\la)$ as $k\cv\infty$
uniformly on compact subsets of $\La$.
Furthermore $p(\lambda)$  
coincides with the  natural continuation of $p$ near $\lo$ as a point of the 
hyperbolic set $E_\la$ that  dynamically corresponds to  $E_\lo$.  
In particular,  if $(\tilde\gamma_k)$ is any other sequence with
$\tilde\gamma_k(\la_0){\cv} p(\la_0)$,  then $\tilde\gamma_k(\la)\cv
p(\la)$  on the whole $\La$.

This holds in particular when $p(\lo)$ is a saddle periodic point or a
transverse s/u intersection. 
%
%
%
%
%
\end{lem}
%

\begin{proof} 
Let $N$ be a neighborhood of $\lo\in \La$ where $E_\la$ persists as a hyperbolic set. Then the point $p$ admits a natural local continuation $(p(\la))_{\la\in N}$. We claim that in $N$, $\gamma_k(\la)\cv p(\la)$ when 
$k\cv\infty$.  Then the other conclusions of the lemma follow.

Indeed
 consider any cluster value of the sequence of holomorphic maps $(\gamma_k(\lambda))_
{\lambda\in \Lambda}$ 
(recall that from  Lemma \ref{lem:normal} and the remarks preceding it that this is a normal family).
By our claim it has to coincide with $p(\lambda)$ in $N$. This in turn 
allows to define a holomorphic continuation $p(\lambda)$ of $p$ throughout $\La$.


\medskip

It remains to prove our claim that $\gamma_k(\la)\cv p(\la)$ in some neighborhood of $\la_0$. 
Let us first deal with the case where $\mathcal{G}_\la\subset K_\la$. 
The observation is that 
for $\la\in N$, the dynamics is locally expansive near $p(\la)$, that is: there exists $\delta>0$, which 
can be chosen to be uniform in $N$ (reducing $N$ if needed), 
such that if $q(\la)$ is such that $d(f_\la^n(q(\la)), f_\la^n(p(\la)))\leq \delta$  for all $n\in\zz$, 
then $p(\la) = q(\la)$. Now let $q$ be any cluster value of the sequence of graphs $\gamma_k(\la)$, and consider the family $
( {f}^n_\la(q(\la)))_{n\in \zz}$. This is a bounded, hence normal, family of graphs (since they are contained in $\bigcup K_\la$), and by 
assumption, ${f}_\lo^n(q)(\la_0) = f_{\la_0}^n(q(\la_0))=f_{\la_0}^n(p(\la_0))$. Therefore by equicontinuity, for $\la$ close to $
\la_0$, $f_\la^n(q(\la))$ remains close to $f_\la^n(p(\la))$ and we are done.

\medskip

Assume now that for all $\la\in \La$, $\mathcal{G}_\la\subset J^+_\la\setminus K_\la$. 
By Lemma \ref{lem:normal}, we can extract a subsequence, still denoted by $\gamma_k$, such that $\gamma_k$ converges to some $\gamma:\La\cv\cd$ with $\gamma(\lo) = p(\lo)$. We claim that for $\la\in \La$, $\gamma(\la)$ is  in $K_\la$. Indeed 
 $\la\mapsto G^-_\la(\gamma_k(\la))$ is a sequence of positive harmonic functions, so its limit  $\la\mapsto
 G^-_\la(\gamma(\la))$ is    harmonic and non-negative, and we conclude by 
 observing that $G^-_{\lo}(\gamma(\lo)) = 0$, whence $\la\mapsto G^-_\la(\gamma(\la))$ vanishes  identically, so that $\gamma(\la)\in K_\la$. 
 Then by applying the reasoning of the previous paragraph we deduce that $\gamma(\la) =p(\la),$ which was the desired result. 
\end{proof}


We now show that in substantial families, 
 saddles do not change their nature under holomorphic motions.

\begin{lem}\label{lem:substantial}
Let $(f_\lambda)_{\la\in \Lambda}$ be a substantial family of polynomial automorphisms of $\cd$ of dynamical degree 
$d\geq 2$.  Let $\PP_0$ be a  set of periodic points for $f_0\equiv f_{\la_0}$ 
which admits a continuation as a branched holomorphic motion $\PP_\la\subset K_\la$ over $\La$.
If $q(\la_0)\in  \PP_0$ is a  non-isolated  
saddle periodic point, 
then its unique  continuation $q(\la)\in  \PP_\la$   remains a  saddle  for all $\la\in \La$.
\end{lem}

\begin{proof} 
From Lemma \ref{lem:extension/expansion} we know that $q(\lo)$ admits a unique continuation $q(\la)$ to 
$\La$ which locally 
coincides with its continuation as a saddle point.
By analytic continuation of the identity $f_\la^N(q(\la)) = q(\la)$, $q(\la)$ 
is periodic throughout the family, and we need to show that it remains of saddle type.

To illustrate the idea,  assume  first  that $f_\la$ is dissipative
for any $\la\in \La$.
In this case, if a saddle  bifurcates, it must become a sink for an open set of parameters.
On the other hand,  as $q(\la_0)$ is  non-isolated in $\PP_0$,
there is a sequence of other saddles $p_n (\la_0)\in \PP_0$  converging to $ q(\la_0)$.  
By Lemma \ref{lem:extension/expansion}, $ p_n(\lambda) \to q(\lambda)$
on the whole space $\La$. Moreover, 
outside countably many exceptional parameters in $\Lambda$, the points
$p_n(\la)$  remain different from $q(\la)$.  It follows that there is a parameter
$\la\in \La$ for which $q(\lambda)$ is a sink
that can be approximated by other  periodic points, which is contradictory.

Let us now address the general case\footnote{The argument is similar to that of \cite
[Theorem 3]{bls2} but   the possibility of persistent non-linearizability, e.g. 
persistent resonance between the eigenvalues, was 
overlooked there. This is the 
reason for the additional assumption that $(f_\la)$ is substantial.}.
We start with a 
saddle periodic point $q$ of period $N$, that we can follow holomorphically
as $(q(\lambda))_{\la\in \La}$. We want to show that it cannot change type, i.e. that
neither of the   eigenvalues  of the differential  $Df^N$
at $q(\la)$ crosses the unit circle.
Since these eigenvalues    
 are not locally constant (this is forbidden by the ``substantiality'' assumption),
 at a bifurcating parameter they run through an open arc of the unit circle. Thus,
we may always assume that we the eigenvalues are  far from 1, so that we can follow them  
holomorphically as $\alpha_1(\la)$ and $\alpha_2(\la)$. Without loss of generality 
we replace $\La$ by a one-dimensional submanifold   with the property that no relation of the form 
$\alpha_1^a\alpha_2^b =c$  holds persistently in it. 

If a bifurcation occurs in the locus where $\abs{\Jac f_\la}\neq 1$ then a 
sink or source can be created, and we conclude as before. In the remaining case, 
elliptic points are created.
Recall that the possibility of linearizing a periodic point
 depends on a Diophantine condition on the eigenvalues.
To be specific, a sufficient condition for linearizability is that there exists $\nu>0$ 
such that for $j_1, j_2\geq 1$ and $k=1, 2$, $\abs{\alpha_1^{j_1}\alpha_2^{j_2}-
\alpha_k}\geq \frac{C}{(j_1+j_2)^\nu}$.

Consider a piece $C$ of the curve $ \set{\abs{\Jac f_\la} =1}$ in 
parameter space, and a point $\lambda_1\in C$ where $\abs{\alpha_1}=\abs
{\alpha_2}=1$. Recall that $\alpha_1$ and $\alpha_2$ are holomorphic and non-constant. 
There are two possibilities.
Either $\abs{\alpha_1}=\abs{\alpha_2}=1$ along $C$ or not.  In the latter case there 
is a branch of the curve $\abs{\alpha_1}=1$ having an isolated intersection with $C$, 
so we have bifurcations in the dissipative regime and we are done.
In the first case, we claim  that $q(\lambda)$ cannot be persistently non-linearizable 
along $C$.
Then, at a parameter where $q(\la)$ is linearizable, it is the center of a Siegel ball, 
and we get a
contradiction in the same way as in the dissipative case.

To prove our claim, note that for $\la\in C$ we can write
$\alpha_k(\lambda) = e^{i\theta_k(\lambda)}$, $k=1,2$, where $\theta_k$ is 
real analytic. {\it Since the family is substantial},   
$(\theta_1, \theta_2, 1)$ are linearly independent. It is then a theorem of Schmidt \cite
{schmidt} (solving a conjecture of Sprindzhuk's, see also \cite{km}) that for a.e. $
\lambda$, $(\alpha_1(\lambda), \alpha_2(\lambda))$ is Diophantine. This concludes 
the proof of Lemma \ref{lem:substantial}.
\end{proof}

Let us point out the following consequence of Lemmas \ref{lem:extension/expansion} and \ref{lem:substantial}.

\begin{cor}\label{cor:continuous} 
Let $(f_\lambda)_{\la\in \Lambda}$ be a substantial family of polynomial automorphisms 
of $\cd$ of dynamical degree  $d\geq 2$. 
Let $\PP_0 \subset \Saddles_0$  
be a  set of saddles of $f_0\equiv f_{\la_0}$ without isolated points, that admits a continuation as 
a branched holomorphic motion $\PP$ with  $\PP_\la\subset K_\la$ for  $\la\in \La$.

Then all  points in $\PP_0$ persist as saddles and 
$(\PP_\la)_{\la\in \La}$ is the corresponding  holomorphic motion.  
It  is  strongly unbranched, and hence continuous.
Moreover, it extends to a strongly unbranched   holomorphic motion of all saddles   
 that  belong  to   $\overline {\PP_0}$.
\end{cor}

\begin{proof}
If $p(\lo)\in \PP_0$, Lemma \ref{lem:extension/expansion} implies  that it admits a unique continuation 
 $p(\la) \in \PP_\la$, which is never isolated in $\PP_\la$,  and Lemma \ref{lem:substantial} says that $p(\la)$ is a saddle for all $\la$. 
So we can apply Lemma \ref{lem:extension/expansion} at all parameters, and it follows that $\mathcal P$ is strongly unbranched. The same holds for every saddle point belonging to $\overline {\PP_0}$.
\end{proof}

We will also need the following result. 

\begin{lem}\label{S contains J}
  Let $\PP_\la$ be a holomorphic motion of a set of periodic points
  of $f_\la$, $\la\in \La$, such that $\PP_0\equiv \Saddles_{\la_0}$ is the
  set of all saddles of $f_0\equiv f_{\la_0}$. Then $\overline \PP_\la\supset
  J^*_\la$ for all $\la\in \La$. 
\end{lem}


\begin{proof}
Note first that the statement  is not obvious since {\it a priori} $f_0$ may have infinitely many sinks 
that could transform into  saddles during the deformation.

However,  by \cite{bls2}, if we denote by $\PP_{n,\la_0} = \Saddles_{n,\la_0}$  the
set of saddle points 
  with  period  dividing $n$,  then
$$
    \frac{   \#\Saddles_{n, \lo} } {d^n}\cv 1\quad {\mathrm{ and}}\quad 
      \frac 1{d^{n} } \sum_{p\in \mathcal{S}_{n,\la_0} } \delta_p\cv \mu_{\la_0}.
$$  
Hence the continuation $\PP_{n, \lambda}$ of $\PP_{n,\la_0}$ is a set of periodic points
with $\# \PP_{n, \lambda}=\#\PP_{n,\la_0} \sim d^n$. Thus, by the   
Equidistribution Theorem of \cite{bls2}, applied to $f_\la$ we obtain that
$$
   \frac 1{d^{n} } \sum_{p\in \PP_{n, \la}} \delta_p\cv \mu_\la,
$$
and the conclusion follows.
\end{proof}   

\subsection{From special  motions to weak $J^*$-stability} \label{subs:special}

The $\la$-lemma allows us to promote  motion of saddles or s/u
intersections to  weak $J^*$-stability:



\begin{prop}\label{prop:extending correspondence}
Let $(f_\lambda)_{\la\in \Lambda}$ be a substantial family of polynomial automorphisms of $\cd$ of dynamical degree 
$d\geq 2$. Assume there  exists a  \BHM 
 $\G$  such that:
\begin{itemize}
\itm $\mathcal{G}_\la\subset K_\la$  for any $\la\in \La$;
\itm $\mathcal{G}_0\equiv \GG_{\la_0}$ is dense in $J^*_0\equiv J^*_{\la_0}$  for some $\la_0\in \La$.
\end{itemize}

Then:
\begin{enumerate}[(a)]

\item $\overline{\mathcal{G}_\lambda}\supset J^*_\lambda$ for every
  $\lambda \in \Lambda$; 

\item no saddle point bifurcates in the family, and 
the motion of saddles of $f_0$ is an equivariant strongly
unbranched (and hence continuous) holomorphic motion; 


\item the family $(f_\lambda)_{\la\in \Lambda}$ is weakly
  $J^*$-stable;  

\item the motion of transverse  s/u  intersections is an equivariant strongly
  unbranched motion; 
\item homoclinic and   heteroclinic tangencies are persistent.
\end{enumerate}
Conversely, each of the conditions  (b), (c), or (d) implies the
others. 
\end{prop}

\begin{proof}  
Let us consider the set  $\Saddles_0$ of saddles of $f_0$. By Lemma \ref{lem:extension/expansion},
points of $\Saddles_0$ can be followed holomorphically  along $\La$, 
giving rise by Corollary \ref{cor:continuous}
 to a  strongly unbranched  holomorphic motion of saddles  $\PP=(\PP_\la)$
 over $\La$.
 Moreover, $\PP \subset \overline{\GG}$, while 
by Lemma \ref{S contains J},  
\begin{equation}\label{S contains J eq}
     \overline{\PP_\la} \supset J_\la^*\quad \mathrm{ for \ all} \ \la\in \La,
\end{equation}
 implying {\it  (a)}. 

Since  the motion
$\PP $ satisfies the assumptions of the proposition,
while  $\GG$ is not part of any further assertion, from now on  we
can assume  $\mathcal{G} = \PP$. In particular, $\GG$ is equivariant. 



Since $\GG$ is a motion of saddles, and since every saddle point
belongs to $J^*$, 
we conclude from \eqref{S contains J eq} that 
for every $\la$, $\overline{\GG_\la}=J_\la^*$.
Applying Corollary \ref{cor:continuous} once again
(with different base points), we conclude that 
 no saddle point can change type in the family, and that for any
 $\la\in \La$, 
$\GG_\la$ is the set of all saddles of $f_\la$,  i.e., $\GG_\la=\Saddles_\la$.
This proves {\it (b)}.
Since $\GG$ is an equivariant normal holomorphic motion, 
$\overline \GG $ is an equivariant \BHM  of $J^*$, implying {\it (c)}.


\msk
For {\em (d)}, let  $q$ be a transverse point of intersection of
$W^s(p_1)$ and $W^u(p_2)$ (for some parameter $\lo$).
By Lemma \ref{lem:extension/expansion} 
$q$ admits a unique continuation $q(\la)$ which locally coincides 
with its natural continuation as  a s/u intersection. 
By {\em (c)}, the saddle points $p_1$ and $p_2$ persist in the
family. Since $f^n_\la(q(\la))$  
is a 
normal family and $f^n_\la(q(\la))\cv p_1(\la)$
 for $\lambda$ close to   $\lo$, this convergence  holds throughout $\Lambda$, and similarly for $f^{-n}_\la(q(\la))$.
 In particular $q(\la)\in W^s(p_1(\la))\cap  W^u(p_2(\la))$ for all parameters so it 
remains an s/u intersection.  

Let us now show that this s/u intersection remains transverse or equivalently, that there are no collisions. This 
will establish {\em (e)} and at the same time that the motion of transverse s/u  intersections is strongly unbranched by 
Lemma \ref{lem:extension/expansion}, thus completing the proof of {\em (d)}. 
(Notice that tangencies are not a priori incompatible with the fact that 
intersections are moving holomorphically, due to the possibility of degenerate tangencies).

Consider  a pair $q(\lo)$, $q'(\lo)$ of distinct transverse intersections of   $W^s(p_1(\lo))$ and $W^u(p_2(\lo))$. We know that $q$, $q'$ (as well as $p_1$, $p_2$) can be followed holomorphically. We have to show that $q$ and $q'$ stay distinct. 
 For this we parameterize $W^u(p_2(\la))$ by some $\phi_\la:\cc\cv W^u(p_2(\la))$, depending holomorphically on $\la$
 (see the comments preceding Proposition \ref{prop:holmotion} 
 below),  
 so we may identify $W^u(p_2(\la))$ with  $\cc$. Fix another saddle point $p_3(\la)$. 
 Since $W^s(p_3(\lo))$ intersects transversally $W^u(p_1(\lo))$, by the   Lambda (or inclination) lemma of hyperbolic dynamics
 (see \cite[p. 155]{palis takens})
  we get that $q(\lo)$ is the limit, inside $\cc\simeq W^u(p_2(\lo))$ of a sequence of 
transverse intersection points $q_n(\lo)$ of $W^s(p_3(\lo))\cap W^u(p_2(\lo))$. These intersection points can be followed globally in $\La$. 

We claim that $q_n(\la)$ converges locally uniformly to $q(\la)$ in $\La$  (again here we work in $\cc$). 
Indeed notice first that by Montel's theorem $q_n$ is a normal family, since locally we can follow any finite set of transverse  
intersections of $W^s(p_1(\la))\cap W^u(p_2(\la))$, and $q_n(\la)$ stays disjoint from them. Then we argue that $q_n(\lo)$ converges to 
$q(\lo)$ while $q_n(\la)$ is disjoint from $q(\la)$, so by Hurwitz' Theorem $q_n(\la)$ converges to $q(\la)$. 

To conclude the argument, assume that there exists $\la_1$ such that $q(\la_1)  = q'(\la_1)$, and let $N$ be any neighborhood of $\la_1$. 
Now if for every 
$\la\in N$ and $n\geq 0$, $q'(\la)\neq q_n(\la)$, we have a  contradiction with
 Hurwitz' Theorem. Thus there exists $\la_2\in N$ and an integer $n$
 such that
 $q'(\la_2) =  q_n(\la_2)$ which is impossible because these points  belong to different stable manifolds. 
Hence, item {\em (e)} is established.

Conversely, if one of the conditions 
 {\em (b)}, {\em (c)} or {\em (d)} holds, then the assumption of the proposition is satisfied, and we infer that the other conclusions hold. This completes the proof. 
\end{proof}

Let us 
point out the following consequence of  Proposition \ref{prop:extending correspondence}: 


\begin{cor}\label{cor:dense saddles}
Let $(f_\lambda)_{\la\in \Lambda}$ be a substantial family of polynomial automorphisms of $\cd$ of dynamical degree 
$d\geq 2$.  If there exists a persistent 
set of   saddle points, which is dense in $J^*$ for some parameter, then the family is weakly $J^*$-stable.
\end{cor}

\begin{proof}
  Apply the implication $(b)\Rightarrow (d)$ of 
  Proposition \ref{prop:extending correspondence} 
to the given set of saddles. 
\end{proof}

\begin{rmk}\label{rmk:equivariant}
Notice that  in Proposition
 \ref{prop:extending correspondence}  we  do not assume  any equivariance for  $\G$.
This shows that the equivariance assumption   is   superfluous in Definition  \ref{def:stable}.
\end{rmk}

\begin{rmk}\label{rmk:slip}
It follows from Proposition
 \ref{prop:extending correspondence} and from the 
 density of transverse homoclinic intersections in $J^*$ at every parameter that if all homoclinic intersections can be followed in 
 some $\om\subset\La$, then $(f_\la)$ is weakly $J^*$-stable there. This a priori does  {\em not} imply   that weak $J^*$-stability 
 follows from the absence of homoclinic tangencies. Indeed, if there are no tangencies in $\om$, every homoclinic intersection can be 
 followed locally. However,  intersections may disappear by ``slipping off to infinity" inside stable and unstable manifolds. The  
 methods that  we develop in Part \ref{part:tangencies} of the paper should be seen as a way of circumventing this problem. 
 \end{rmk}

\subsection{Further consequences} 


\begin{cor}\label{cor:semiconj}
Under the assumptions of Proposition \ref{prop:extending correspondence}, if furthermore $(f_\la)$ is $J^*$-stable in a 
neighborhood of $\lo$ (for instance if $f_\lo$ is uniformly hyperbolic on $J_\lo^*$), 
then for $\la\in \La$ there is a semiconjugacy $J^*_\lo\cv J^*_\la$.
\end{cor}

\begin{proof} 
By Lemma \ref{lem:extension/expansion} we can follow holomorphically all points of $J^*_\lo$. 
The motion is continuous near $\lo$, so by analytic continuation  we easily 
deduce that
it is continuous throughout $\Lambda$. Likewise, the compatibility between the 
motion and the dynamics holds near  $\lo$ so it holds everywhere, hence it 
defines  a global semiconjugacy.
\end{proof}

\begin{cor}\label{cor:continuous to holomorphic}
If a substantial family has the property that its members   are topologically conjugate 
on $J^*$, and the conjugating map depends continuously on $\la$, then it is 
$J^*$-stable.
\end{cor}

\begin{proof}  
Fix a parameter $\la_0\in\Lambda$.
For every $\la$ there is a conjugacy $h_\la:J^*_0\cv J^*_\la$
depending continuously on $\la$.
If $p$ is a saddle point for $f_0$, then $h_\la(p)=p(\la)$ is a continuously 
moving
periodic point which is  a limit of periodic points for all $\la$. By 
 Lemma \ref{lem:substantial}, $p(\la)$ is a saddle for all parameters. In particular the 
assumptions of 
Corollary \ref{cor:dense saddles} 
hold, and the family $(f_\la)$ is weakly 
$J^*$-stable, thus all saddle points move holomorphically. To conclude, let $q\in J^*$ be 
any point, and let a sequence of saddle points $p_n\cv q$. Then for all $\la$, $h_\la
(p_n)\cv h_\la(q)$, so $\la\mapsto h_\la(q)$ is holomorphic. Finally, it is obvious that 
the motion is injective since the $h_\la$  are homeomorphisms.
\end{proof}

\section{Holomorphic motions in unstable manifolds  \\ and persistent  connectivity
of the Julia set}\label{sec:extension}

In this section we  capitalize on the idea that in a weakly
$J^*$-stable family of polynomial automorphisms, we can apply the
one-dimensional theory of holomorphic motions inside stable and unstable manifolds.  
In \S \ref{subs:connex} we show that connectivity 
properties of Julia sets are preserved in a weakly $J^*$-stable family. Conversely, if in some dissipative family $(f_\la)$, the Julia set 
is persistently connected, then the family is weakly $J^*$-stable
(Theorem \ref{thm:connected}). 
In \S \ref{subs:proof}, we use these techniques to complete the proof of Theorem \ref{thm:equiv}. 
Finally, in \S \ref{subs:extension} 
we prove, using the Bers-Royden $\lambda$-lemma, that in 
a weakly $J^*$-stable family, the equivariant \BHM  of the little Julia set  $J^*$ actually extends to 
the big Julia set $\bJ= J^+\cup J^-$.
Note that in the hyperbolic case, this method was first used by Buzzard and Verma \cite{buzzard verma} to show that 
$\bJ $ moves under an actual  holomorphic motion. 

\subsection{Motion of $\fr (W^u(p)\cap K^+)$}    
Let  $f$ be a polynomial automorphism of $\cd$ of dynamical degree  
$d\geq 2$,  and  let $p$ be a saddle periodic point. Then $W^u(p)$ is 
is  dense in $J^-$ (see \S \ref{basics sec}). 
In particular, there are two distinct topologies on $W^u(p)$: the one induced by the isomorphism with $\cc$, which we refer to as 
the {\it intrinsic topology}, and the topology induced from $\cd$.  
Viewed as subsets of $\cc$, the components of $W^u(p)\cap K^+$ are simply connected 
closed  subsets that may be bounded or unbounded. Notice that, with  our current state of knowledge, 
 nothing prevents   a component of $W^u(p)\cap K^+$
with non-empty interior from being  fully contained in $J$ or even in $J^*$. 

Recall the notation $H(p)$ and  $H^{\rm tr}(p)$  from \S \ref{basics sec}  for  
the sets of homoclinic  intersections associated with a saddle $p$. 
Throughout this section we use the notation $\Int_\intr X$, $\cl_\intr X$, and $\fr_\intr X$ for the
intrinsic  interior, closure, and boundary of a subset $X\subset
W^u(p)$.  

\begin{lem}\label{lem:boundary}
Relative to  the intrinsic topology in $W^u(p)\simeq \cc$ we have that 
$$\fr_\intr (W^u(p)\cap K^+)  = \cl_\intr{H(p)} = \cl_\intr {H^{\rm tr}(p)}.$$
\end{lem}

\begin{proof} 
This is very similar to \cite[\S 9]{bls}  (see also  the proof of \cite[Cor. 1.9]{connex}). We freely use the formalism of laminar 
currents and Pesin boxes, the reader is  referred to \cite{bls} for details. 

Let $x \in  \fr_\intr  (W^u(p)\cap K^+)$  and let us show that  $x\in\cl_\intr {H^{\rm tr}(p)}$. Since the 
dynamical Green function $G^+$ admits a non-trivial minimum at $x$, if  $\Delta\subset W^u(p)$ is a disk
containing $x$, then $G^+$ is not harmonic in $\Delta$,
i.e. $T^+\wedge [\Delta]>0$. Let $\psi$ be a cut-off function in $\Delta$, with $\psi = 1$ near $x$. 
By \cite[Thm 1.6]{bs3}, $d^{-n} (f^n)_*(\psi\,  [\Delta])\cv c T^-$ as $n\cv\infty$, with $c=\int \psi[\Delta]\wedge T^+>0$. 
We now argue exactly as  in \cite[Lemma 9.1]{bls}. 
Let $P$ be a Pesin box of positive $\mu$-measure, and $S^+$ be the uniformly laminar current made of the local stable manifolds 
$W^s_{\rm loc} (z)$, $z\in P$,  
with transverse measure given by the unstable conditionals of $\mu$. 
Then $0<S^+\leq T^+$ so $S^+$ has continuous potential \cite[Lemma 8.2]{bls}. 
It follows  that $d^{-n} (f^n)_*(\psi[\Delta])\wedge S^+ \cv cT^-\wedge S^+>0$ and we 
conclude that for large $n$, $f^n(\Delta)$ admits transverse  intersection points with  $W^s_{\rm loc}(z_0)$, for {some}  $z_0\in P$
(the transversality comes from \cite[Lemma 6.4]{bls}).

We claim that iterating a bit further, the iterates of $\Delta$
intersect $W^s_{\rm loc}(z)$ transversely,  
 for {\em every} $z\in P$.     
Indeed, for $y$ sufficiently close to $z_0$, $f^n(\Delta)$ intersects
$W^s_{\rm loc}(y)$ transversely. 
By Poincaré recurrence and the Pesin Stable Manifold Theorem, for
typical $y$ like this,
there exists an infinite sequence $(n_j)_{j\geq 1}$, such that 
$f^{n_j}(y)\in P$ and
$f^{n+n_j}(\Delta)$ contains a disk close to $W^u_{\rm
  loc}(f^{n_j}(y))$, and the claim follows. 

We now  apply exactly the same argument to a neighborhood  of $p$ in $W^s(p)$.  
In this way we obtain disks in $W^s(p)$, arbitrary close to 
$W^s_{\rm loc} (z)$ for some $z\in P$, from which we conclude that $f^{n+n_j}(\Delta)$, hence $\Delta$, intersects $W^s(p)$ 
transversely, and we are done. 

\medskip

As was mentioned in \S \ref{basics sec},  
$\overline{H(p)} = \overline{H^{\rm tr}(p)}$ in $\cd$  \cite[Prop. 9.8]{bls}.
In fact,  the proof works in the intrinsic topology of $W^u(p)$ as well. 
For convenience,  let us  recall the argument: 
since $W^u(p)$ admits non-trivial transverse intersections with
$W^s(p)$, it follows from the Hyperbolic $\lambda$-lemma that every disk $ \Delta \subset W^u(p)$ 
is the limit in $\cd$
 of an infinite sequence of disjoint disks $\Delta_n\subset W^u(p)$. 
 Hence  the result follows  from the instability of non-transverse intersections 
 \cite[Lemma~6.4]{bls}. 
 
To conclude the proof, let us 
show that $H^{\rm tr}(p)\subset \fr_\intr  (W^u(p)\cap K^+)$. 
Observe first that $p\in \fr_\intr (W^u(p)\cap K^+)$ (which is well known). 
Otherwise $p$ would lie in  the interior of $W^u(p)\cap K^+$, 
hence the family of forward iterates $f^n$ would be normal in some intrinsic neighborhood $V \subset W^u(p)$ of $p$,
contradicting the growth of the derivatives $Df^n (p)$ in the direction of $W^u(p)$. 

Let now $x\in H^{\rm tr}(p)\subset W^u(p)\cap K^+$ 
and let $\Delta$ be a small disk around $x$, 
which is thus transverse to the stable manifold of $p$. 
By the Hyperbolic $\lambda$-lemma, 
$f^n(\Delta)$ contains graphs arbitrary close in $\cd $ to  
a neighborhood of $p$ in $W^u(p)$, 
so $f^n(\Delta)$ must intersect $U^+$, and the conclusion follows. 
\end{proof}

Let now $(f_\la)_{\la\in\La}$ be a 
holomorphic family of polynomial automorphisms      
with a holomorphically moving saddle point $(p(\la))$.
Then there exists  a holomorphic family  of
parametrizations $\psi^u_\la:\cc\cv W^u(p(\la))$  
with $\psi^u_\la (0)  = p(\la)$. Indeed,
since the eigenvectors of $Df_\la(p(\la))$ depend holomorphically on
$\la$, we can normalize the family so that $p(\la)=0$ and the
eigenbasis is equal to the standard basis at $0$, 
with the vertical direction unstable. Then normalize $\psi^u_\la$ so that
$ (\pi_2 \circ \psi_\la ^u)'  (0) =1$. The Graph Transform construction of
the local unstable manifold provides an explicit formula for the
normalized parametrization:
$$
      \psi_\la (z) =   \lim_{n\to  +\infty} \frac 1{\mu_n}  f_\la^n (0,z),\quad
\mathrm{where}\  \mu_n=  
\left. \frac {d (\pi_2\circ f_\la^n (0,z)) } {d  z} \right|_{z=0}, 
$$
which is manifestly holomorphic in $\la$. 
 

The   following easy proposition is the starting point for most of the
results in this section. 

\begin{prop} \label{prop:holmotion}
Let $(f_\la)$ be a weakly $J^*$-stable  family of polynomial automorphisms, 
  $(p_\la)$ be   a holomorphically moving saddle point  and $(\psi^u_\la)$ 
  be a holomorphic family of parameterizations of $W^u(p_\la)$, as above. Then
  $(\psi^u_\la)^{-1}(\fr_\intr (W^u(p_\la)\cap K^+_\la))$ moves
  holomorphically in $\cc$.
\end{prop}

\begin{proof}
By Proposition 
\ref{prop:extending correspondence} {\em (e)} and {\em (f)},
homoclinic intersections move holomorphically and without collisions. 
Therefore the result follows directly from Lemma \ref{lem:boundary}
and the ordinary one-dimensional  $\lambda$-lemma.
\end{proof}

\subsection{Preservation of connectivity under branched motions}\label{subs:connex}


Let us start with a few general comments on connectivity of Julia sets of polynomial automorphisms.

Of course, if $J$ is totally disconnected, then so is $J^*$. Moreover, 
by \cite{bs3}, every component of $K$ intersects $J^*$, so in particular if $K$ is totally disconnected, then $J=J^*=K$. 
It follows from general topology that if $J$ is totally disconnected then so is $K$: indeed the boundary of a non-trivial continuum cannot be totally disconnected.
On the other hand it is unclear whether 
total disconnectedness of $J^*$  implies that $J$ is totally disconnected. 

Following \cite{bs6}, we say  that $f$ is said to be {\em unstably connected} if  $U^\pm\cap W^u(p)$ is simply connected for some
(and then any) saddle point $p$, and {\em unstably disconnected} otherwise. 
By \cite{bs6}, $J$ is disconnected iff $f$ is stably and unstably disconnected.
It is not difficult to see that $K$ 
is disconnected in this case  
 (an inclination lemma argument). Likewise, 
it follows from \cite{bs6} that $J$ is connected iff $J^*$ is connected. 
We also remark that a dissipative map is always stably disconnected  \cite[Cor. 7.4]{bs6}.

\medskip

We start by 
observing that  the connectedness  of $J$ 
  is preserved in weakly $J^*$-stable families. 
 
 \begin{prop}\label{connectivity preserved}
Let  $(f_\la)_{\la\in \La}$  be a weakly $J^*$-stable substantial family 
of polynomial automorphisms of $\cd$ of dynamical degree $d\geq 2$.  
Then stable and unstable connectivity are preserved in the family.   
In particular  if for some parameter $\la_0$, 
$J_{\la_0}$ is connected, then $J_\la$ is connected for all $\la$. 
\end{prop}

\begin{proof}
Let us show that disconnectedness of $J$ is preserved in a weakly $J^*$-stable family. 
So assume that  for some $\la$, $f_\la$ is stably and unstably disconnected, so for every 
saddle point $p$, $W^u(p)\cap K^+$ and $W^s(p)\cap K^-$ admit intrinsic compact components.
 By Proposition \ref{prop:holmotion},  if  $C_\lo$ is any compact component of  $W^u(p_\lo)\cap K_\lo^+$, 
$\fr_i C_\lo$ moves holomorphically as the parameter evolves, without colliding with the other components, 
so its  continuation $(\fr_i C)_\la$ bounds a compact component of   $W^u(p_\la)\cap K^+_\la$. 
Hence $f_\la$ is unstably disconnected at all parameters, and the same argument shows 
that it remains stably disconnected as well. We conclude that $J_\la$ is  disconnected for every $\la\in \La$.
\end{proof}

In the same way we obtain  that if     $(f_\la)_{\la\in \La}$  is weakly $J^*$-stable and if for some parameter $\la_0$, 
$J_{\la_0}$ is totally disconnected, then for all $\la$, $f_\la$ is stably and unstably totally disconnected. 
However, it is unclear    
whether  stable and unstable total disconnectedness implies that $J$ or   $J^*$ is totally disconnected.

\subsection{Traces of attracting basins in unstable   manifolds} \label{subs:traces}


Here we will establish the following useful property (perhaps,  known  to experts): 

\begin{lem}\label{lem:basin in unstable manifold}
Let $f$ be a  polynomial automorphism possessing  an attracting
periodic point $q$. 
Then for every saddle point $p$, 
there is a non-empty component of ${\operatorname{Int}}_i ( W^u(p)\cap K^+)$ contained
in the basin  ${\mathcal B} (q ) $.  
 
If in addition $f$ is unstably disconnected, then there exists such a component that  is relatively compact in the intrinsic topology. 
 \end{lem}

\begin{proof} 
Since the basin ${\mathcal B} ( q ) $ is
biholomorphic to $\cd$,   it contains entire curves ${\mathcal E}: {\mathbb C}\rightarrow \cd$. 
To any such  curve ${\mathcal E}$, corresponds a positive closed
{\it Ahlfors current}   produced  by averaging  currents of
integration over holomorphic disks ${\mathcal E} : {\mathbb D}_r
\rightarrow \cd$  
 (see e.g. \cite[\S 7.4 pp. 349-350]{nevanlinna}). 
Since $T^+$ is  a unique  positive closed current of mass
1 supported on $K^+$ \cite{fornaess sibony}, 
the Ahlfors current must be equal to $T^+$.  

We can now proceed  as in the proof of Lemma \ref{lem:boundary}:  
fix a Pesin box  and construct a local laminar current  $S^-\leq T^-$
made of local Pesin stable manifolds. 
Since $S^-$ has continuous potential, the 
entire curve ${\mathcal E}$  must intersect it;
 more precisely we get a transverse intersection with some local   unstable leaf $W^u_{\rm loc}(x)$.  
Since $W^u(p)$ contains disks arbitrary close to  $W^u_{\rm loc}(x)$,
$W^u(p)$ intersects ${\mathcal E}$ transversely as well. Thus,
 $W^u(p)\cap {\mathcal B} ( q ) \not= \emptyset$.

Finally, note that any component $C$ of ${\operatorname{int}}_i\, ( W^u(p) \cap K^+)$ is either entirely
contained in the basin ${\mathcal B} ( q ) $ or disjoint from it.
This follows easily from normality of the family of restrictions  
$ f^n :  C\rightarrow \cd$, ${n\geq 0}$. 

\medskip

Now assume that $f$ is unstably disconnected, or equivalently, that  $K^+\cap W^u(p)$ 
admits a compact component.
 Thus there exists 
 an intrinsically bounded  topological disk $\Delta \subset W^u(p)$ such that $\Delta\cap K^+\neq \emptyset$ and $\fr_i \Delta \subset U^+$. We claim that there exists a component of $\mathcal B(q) \cap W^u(p)$ that is contained in $\Delta$. 
By Lemma \ref{lem:boundary},  $H^{\rm tr}(p)\cap \Delta\neq
\emptyset$, so by the Hyperbolic $\lambda$-lemma, 
for large $n$, 
$f^n(\Delta)$  contains disks arbitrary $C^1$-close to any given   disk in $W^u(p)$. 
Now the first part of the proof shows that there is a point of transverse intersection between $\mathcal E$ and $W^u(p)$. Therefore 
$f^n(\Delta)$ intersects $\mathcal E$, hence $\mathcal{B}(q)$ for large $n$. By invariance, $\Delta$ intersects 
$\mathcal{B}(q)$ as well, and since  $\fr_i \Delta \cap K^+ = \emptyset$, we conclude that there is a component of 
$W^u(p)\cap \mathcal{B}(q)$ which  is compactly contained in $\Delta$. 
\end{proof}


In fact, the above  proof gives a more general statement: 

\begin{prop}\label{prop:basin unstable}
 Let $\mathcal D$ be a component of $\Int K^+$ containing an entire curve $\EE:
 \cc\rightarrow \mathcal D$. Then for every saddle point $p$, 
there is a non-empty component of ${\operatorname{Int}}_i ( W^u(p)\cap K^+)$ contained
in  $\mathcal D  $.  
 If in addition $f$ is unstably disconnected, then there exists such a component that  is relatively compact in the intrinsic topology. 
\end{prop}

\subsection {Persistent connectivity and moving Bedford-Smillie solenoids }\label{persistent connectivity} 
We will now  show that  in  the dissipative case,   the preservation    of connectivity
properties of the Julia set implies stability.  Our first statement is that persistent Cantor Julia sets are stable.
 

\begin{prop}\label{prop:totally disconnected}
Let $(f_\la)_{\la\in \La}$ be a family of dissipative polynomial automorphisms of $\cd$ of dynamical degree $d\geq 2$.
If $J^*_\la$ is totally disconnected for all $\la$, then $(f_\la)$ is weakly $J^*$-stable.
\end{prop}

%
%

\begin{proof}
If $J^*$ is totally disconnected, 
then for any saddle $p$,   $W^u(p)\cap J^*$ is totally disconnected as
well. Together with  Lemma \ref{lem:basin in unstable manifold}, 
this implies that $f$ does not have sinks. If this happens persistently
over $\Lambda$, 
then no saddle point can bifurcate and the family $(f_\la)_{\la\in \La}$ is weakly $J^*$-stable by
Proposition \ref{prop:extending correspondence}. 
%
\end{proof}

Our next result  asserts that persistent connectivity of $J$ also implies stability. 
Similarly to  the analogous statement for polynomials in $\cc$, 
the argument  is ultimately   based on the absence of ``escaping critical points". 

\begin{thm}\label{thm:connected}
Let $(f_\la)_{\la\in \La}$ be a   family of dissipative 
polynomial automorphisms of dynamical degree $d\geq 2$.
If for every $\la\in \La$, the Julia set $J_\la$ is connected 
 then the family  $(f_\la)$ is weakly $J^*$-stable.
\end{thm}

\begin{rmk}
 Note that  this is the only moment in our argument that requires dissipativity.
\end{rmk}

As a preparation to the proof, let us recall the notion of 
{\it Riemann surface lamination}. It  is a topological space $S$
endowed with local charts $g_i : U_i\rightarrow D_i\times T_i$,
where the $U_i$ are open,  $D_i$ are domains in $\cc$, and $T_i$ are topological  spaces 
({\it  transversals}) , such that the transit maps 
$g_i\circ g_j^{-1}$ (wherever they are defined) have form $(z,t)\mapsto (\gamma(z,t), h(t))$,
where $\gamma(z,t)$ is conformal in $z$.  
Preimages of $D_i\times \{t\}$ in $U_i$ are called {\it plaques} or
{\it local leaves} ;
they patch together to form  global {\it leaves} endowed with
a natural conforml structure. So, $S$ is decomposed into Riemann
surfaces, which is reflected in its name. 
We denote $L(z)$ the global leaf through a point $z\in S$. 

If all the leaves are dense in $S$ then $S$ is called {\it minimal}. 
It is equivalent to saying that for any transversal $T$ and any leaf
$L$, the intersection $L\cap T$ is dense in $T$.

A minimal Riemann surface lamination  $S$  with Cantor transversals  is called
a {\it solenoid}. The leaves of a solenoid $S$ can be topologically
recognized  as {\it path connected components} of $S$. 
It follows that any homeomorphism between solenoids $h: S\ra S'$  
maps  homeomorphically leaves to leaves, $h: L(z) \ra L(hz)$.
If this leafwise map is conformal then $h$ itself is called a 
{\it conformal} solenoidal homeomorphism. 
More generally, a {\it conformal solenoidal map} $h: S\ra S'$ is a
continuous map that induces, for any $z\in S$,  a conformal isomorphism
between the leaves $L(z)$ and $L (hz)$.

Let us now go back to 
Theorem \ref{thm:connected}.
By Proposition \ref{prop:henon},   
we can normalize the family $(f_\la)$ so that $f_\la$ is a product a Hénon mappings depending holomorphically on $\la$. 
This puts us in a position to apply the following  Structure Theorem
due to Bedford and Smillie \cite{bs6}.

Let   $f$ be a  composition of Hénon maps 
that is unstably connected.  
Then the set
$S^- = J^-\setminus K^+$ is a  solenoid
 whose leaves are 
conformally equivalent to  the upper half plane $\Hyp=\{ \operatorname{Im}\,  z>0\}$.
Furthermore, the B\"ottcher function $\varphi^+$ admits a holomorphic
extension to a neighborhood of $S^-$ \cite[Thm~6.3]{bs6},
and the map $\varphi^+ :  S^-\rightarrow \cc\setminus \overline \dd$
is a locally trivial fibration with Cantor fibers $F(c) = F_f(c) := \{ \varphi^+=c\} $.
Moreover, the restriction of $\varphi^+$ to each leaf of $S^-$ is the universal
covering over $\cc \setminus \overline \dd$.
 We will refer to $S^-=S_f^-$ as the  {\em Bedford-Smillie solenoid 
 of} $f$. 

In particular, given a saddle $p$, any component $L$  of
$W^u(p)\setminus K^+$ provides us with a leaf of the solenoid $S^-$.
It follows that $L$ intersects  any fiber $F_c$,
$c\in \cc\setminus \overline \dd$,  by a countable dense subset.  Note
also that by \cite[Theorem 4.11]{bs6}, 
$W^u(p)\setminus K^+$ consists of   only finitely many leaves.

The map $f$ restricts to a conformal  homeomorphism  of  $S^-$  
that maps fibers to the fibers,
$f(F(c)) \subset F(c^d )  $ (according to the  B\"ottcher equation).  

\comm{*******
A good model for $S^-$ is provided by the {\it natural esxtension}  of
the map $\tau: z\mapsto z^d$ on $\cc\sm \bar \dd$. 
Namely, let 
$$
    \Sigma = \set{u=(u_j)\in (\cc\setminus\dd)^\zz: \ \tau(u_j)= u_{j+1}}.
$$  
It is a solenoid with leaves consisting of  backward asymptotic
orbits:
$$
    L(u) = \{ v=(v_j)\in \Sigma:\ |v_j-u_j|\to 0,\ \mathrm{as} \ j\to
    -\infty \}.
$$
Each of the leaves is conformally equivalent to the upper half plane
$\Hyp$. 

Moreover, the  projection
$\pi: \Sigma\ra \cc\sm \bar\dd$, $\pi(u)= u_0$, is a locally trivial
fibration whose leafewise  restrictions  $\pi: L(z)\ra \cc\sm \bar\dd$
are holomorphic universal coverings.  The map $\tau$ naturally lifts to a conformal
homeomorphism $\hat\tau: \Sigma\ra \Sigma$ which is semi-conjugate to $\tau$
by means of  $\pi$. 

Let us  define   a   map 
\begin{equation}\label{Phi}
  \Phi\equiv \Phi_f :   S \cv \Sigma,\quad 
   \Phi_\la(x) = (\varphi^+_\la(f^n(x)))_{n\in \zz}
\end{equation} 
semiconjugating $f$ to $\hat \tau$.  
Bedford and Smillie have proved \cite{bs6}  that  $\Phi$ is a
conformal solenoidal map.
**********}

\medskip
Theorem \ref{thm:connected} will follow from the following result of independent interest. In the hyperbolic case, it was already established 
by P. Mummert \cite{mummert}.

\begin{prop}\label{prop:solenoid}
Let $(f_\la)_{\la\in \La}$ be a   family of unstably connected   
polynomial automorphisms of dynamical degree $d\geq 2$.
Then the
Bedford-Smillie solenoid  $S_\la^-$ of $f_\la$ moves  under an equivariant 
 holomorphic motion that  preserves the fibers of the Böttcher function $\varphi^+_\la$.
\end{prop}

\begin{proof} 
Fix some $\la_0\in \La$; the objects corresponding to this parameter
 will be labeled by ``$0$'', e.g., $S_0^-\equiv S_{\la_0}^-$. 
It is enough to show that the solenoid $S^-_\la$ moves holomorphically in 
 some neighborhood of $\la_0$. 

Pick a saddle point $p_0$ for  $f_0\equiv f_{\lo}$, and let 
 $p_\la$ be its holomorphic continuation to some neighborhood of $\la_0$.
 The unstable manifold $W^u(p_\la)$  is parameterized by the (normalized)
 linearizing coordinate  $\psi^u_\la : \cc\cv W^u(p_\la)$, which   
depends holomorphically on $\la$.
 
Let us consider a leaf  $L_0 \subset W^u(p_0) $ of $S^-_0$
and a  point   $z  = \psi^u_0 (t )\in L_0$. 
The map $  g_0:= \phi^+_0 \circ \psi^u_0$ is univalent in some
neighborhood of $t$, and so is its perturbation  
$$ 
 g_\la := \varphi^+_\la  \circ \psi^u_\la \quad 
  |\la-\la_0|<\de=\de(z).
$$ 
Hence the map $g_\la^{-1}$ is well
defined in some neighborhood of $c \equiv c(t) := g_0 (t)$ 
and depends holomorphically on $\la$. Let 
$$
   t_\la \equiv t_\la(c):= g_\la^{-1} (c), 
\quad  z_\la \equiv z_\la(c) :=  \psi_\la^u (t_\la ), \quad
|\la-\la_0|<\de(z).
$$
 Then for $\de(z)$ small enough, we have

\ssk\nin (i)
$z_\la\in W^u(p_\la)\sm K^+$,  
 hence $z_\la$  belongs to some leaf  $ L(z_\la )  \subset W^u(p_\la)$ of $S^-_\la $;

\ssk\nin (ii)
$\varphi^+ (z_\la) = c$, so $z_\la$ belongs to the fiber $F_c$
independently of $\la$;

\ssk\nin (iii)
 $z_\la $ depends holomorphically on $\la$.  

\ssk
Pick now a base point $z^* =\psi^u_0 (t^*) \in L_0$, and let  $c^*:=
\phi^+_0 (z^*)$, 
 $z^*_\la$ be its motion as
above, $L_\la^*\equiv L(z_\la^*)$.
Since the maps  $\varphi_\la^+:  (L_\la^* , z^*_\la) \ra (\cc\sm
\bar\dd,\,  c^*) $  are holomorphic
universal coverings, for $|\la-\la_0|< \de^*\equiv \de(z^*)$
 there exist conformal isomorphisms 
\begin{equation}\label{lifts}
h_\la: (L_0, z^* )\ra (L_\la^*, z_\la^*)\quad \mathrm{ such\ that} \
\varphi_\la^+ \circ h_\la = \varphi_0^+.
\end{equation}

Let us show that the maps $h_\la: L_0\ra \cc^2$ form a holomorphic motion. 
Note that $\de(z)$ can be selected so that it is lower semi-continuous
(since the same $\de=\de(z)$ serves as $\de(\zeta)$ for points $\zeta$ near $z$).
Hence it is bounded away from $0$ on compact subsets of $L_0$. 
Let us take a relative domain $D_0\Subset L_0$ containing $z^*$, and let 
$ \displaystyle {\de=\inf_{z\in  D_0 } \de(z) } $. Then for
$|\la-\la_0| < \de$,  the maps 
$$
 H_\la: D_0 \ra \cc^2,\quad   z\mapsto z_\la,\ z\in D_0 , 
$$
form a holomorphic motion of $D_0$. 
Since the solenoid $S_\la$ is contained in the Kobayashi hyperbolic domain
$\Om_\la$ from Lemma \ref{Kob hyp},  this motion of $D_0$ is normal.   
By Lemma \ref{continuity},  it is  continuous in $z$. 
Since $D_0$ is connected, $H_\la(D_0)$ belongs to some leaf of
$S_\la$, which must be  $L(z_\la^*) \equiv L_\la^*$. Also, by definition,
$$
\varphi_\la^+ \circ H_\la|\, D_0  =\varphi^+_0|\, D_0\quad\mathrm{ and} \quad  H_\la (z^*) = z^*_\la.
$$
  Comparing this with (\ref{lifts}), we conclude that $H_\la= h_\la|\, D_0$. 

Consequently,  there is $\de_1=\de_1 (z)>0 $ such that
$h_\la(z)$ depends holomorphically on $\la$ for
$ | \la-\la_0|< \de_1 (z)$.  Finally, replacing $\la_0$ by any other
parameter $\la$ in the $\de^*$-neighborhood of  $\la_0$,
 we conclude that $h_\la(z)$ depends holomorphically on $\la$ 
for all $\la$ in this neighborhood,


\msk
Applying Lemma \ref{Kob hyp} once again, 
we conclude that the motion $h_\la : L_0\ra \cc^2$ is normal.   
By the $\la$-lemma (Lemma \ref{lem:extension}),  $h_\la$ extends to a \BHM of 
$\overline L_0 \supset S^-_0$. To see that it gives an actual holomorphic
motion  of $S^-_0$,   consider the B\"ottcher foliation $\FF_\la^+$
 of $U_\la^+$, notice that $\cl_{\FF_\la} (L_\la)\supset S^-_\la$,   and 
 apply the foliated $\la$-lemma   (Corollary~\ref{foliated lambda-lemma}).

Let us use the same notation $h_\la: S^-_0\ra S^-_\la$ for the extended
holomorphic motion.
By continuity, it satisfies the covering property $\varphi_\la^+\circ
h_\la= \varphi_0^+$. Since $\varphi_\la^+$ is a leafwise covering over
$\cc\sm \bar\dd$,  this property determines $h_\la$ uniquely. 
Due to  the B\"ottcher
equation, this identity  is inherited by the motion $\tilde h_\la:= f_\la\circ
h_\la\circ f_0^{-1}$,
$$
     \varphi_\la^+ \circ \tilde h_\la = (\varphi_\la^+ \circ  h_\la\circ
     f_0^{-1} )^d = (\varphi_0^+\circ f_0^{-1})^d= \varphi_0^+,
$$
implying $\tilde h_\la= h_\la$, which by definition means that  
 the motion $h_\la$ is equivariant. 
\end{proof}

\begin{proof}[Proof of Theorem \ref{thm:connected}]   
Since for every $\la\in \La$, $f_\la$ is dissipative and $J_\la$ is connected, it follows from \cite{bs6} that $f_\la$ is unstably connected.
As was already mentioned in the proof of the above proposition,
 Lemma \ref{lem:extension} implies that the equivariant  holomorphic motion
of the Bedford-Smillie solenoid $S_\la$  extends to an equivariant
\BHM  $\GG$
of the closure $\bar S_\la$.  The latter contains all saddle points, which are dense in $J_\la^*$.
Moreover, by equivariance, these saddles remain being periodic under
the motion, so  they stay in $K_\la$. Now Proposition \ref{prop:extending correspondence} implies the
desired.  
\comm{***
By Proposition \ref{prop:solenoid}, each fiber $J_\la^-\cap \set{\varphi_\la^+=\varphi_\la^+(z_0)}$ 
moves holomorphically with $\la$. 
Now we observe that for every parameter $\la$,  
$f^{-n}_\la(\{\varphi^+_\la =  c\}\cap J^-_\la)$ clusters on the whole of $J^*_\la$ as $n\cv\infty$: 
indeed, for $\mu$-a.e. $x\in J^*$, $W^u(x)$ intersects $\{\varphi^+ = c\}$. 
In addition, $(f^{-n}_\la(\{\varphi_\la^+ =  \varphi_\la^+(z_0)\}\cap J_\la^-))_{n\geq 0}$
is locally uniformly bounded in $\cd$. 
We conclude that the set of cluster values of the preimages of the holomorphically moving points 
$\{\varphi^+_\la = \varphi^+_\lo(z_0)\}\cap J^-_\la$ 
form a branched holomorphic motion relating the $J^*_\la$ and we are done 
(recall that by Remark 
\ref{rmk:equivariant} we needn't   check that  this motion is equivariant). 
******}
\end{proof}


\subsection{Bers-Royden motion of unstable manifolds}

Let us consider  a weakly $J^*$-stable substantial  
family  $(f_\la)_{\la\in \La}$
of   polynomial automorphisms. Fix a holomorphically moving
saddle point $p_\la$, and consider a holomorphic family 
of parameterized unstable manifolds 
$\psi^u_\la: \cc\cv W^u(p_\la)$. By Proposition \ref{prop:holmotion},   the
intrinsic boundary  $\fr_\intr (W^u(p_\la)\cap K^+) $  
moves holomorphically.   
Then  locally in $\la$, 
this motion extends to the  Bers-Royden holomorphic motion  of the
whole unstable manifold
$W^u(p_\la)\isom \cc$, see Lemma \ref{entire curves}. Being canonical, it is 
automatically equivariant (where the dynamics is just multiplication by the unstable multiplier), 
so in this way we obtain an equivariant holomorphic motion of  $W^u(p_\la)$ in $\cd$. 


\begin{lem}\label{BR motion}
Let $(f_\la)_{\la\in \La}$ be a weakly $J^*$-stable  substantial family  of 
polynomial automorphisms, and let $(p_\la)$ be a holomorphically
moving saddle point as above. Then the Bers-Royden holomorphic motion of  
$W^u(p_\la)$ preserves the decomposition $\cd = K^+ \sqcup U^+$. 
\end{lem}

\begin{proof}   
For $\lo\in \La$, let $C_0\equiv C_\lo$ be an intrinsic connected component of 
$W^u(p_0 )\cap K^+_0$,
and let $C_\la$  be its image under the Bers-Royden   motion $h_\la$. 
We have to show that  for every $\la\in \La$, $C_\la$ is an intrinsic  connected component of 
$ W^u(p_\la)\cap K^+_\la  $.

By the one-dimensional $\la$-lemma, the maps $h_\la$ 
are intrinsic homeomorphisms, so they preserve intrinsic topological
properties of the subsets of $W^u(p_\la)) $, e.g., 
$\fr_i C_\la = h_\la(\fr_i(C_0))$,   $C_\la$ is intrinsically bounded
iff $C_0$ is, etc. 

 From Lemma \ref{lem:boundary} we know that $\fr_i C_0\subset J^*_0$,
hence  for every $\la$,
 $\fr_\intr C_\la \subset J^*_\la\subset K^+_\la$. 
 So, we need to show that for every $\la\in \La$, $\Int_\intr
 (C_\la)\subset K^+_\la$.
 
 Let  $\om_0 $ be a connected component of  $\Int_\intr C_0 $,
and let $\om_\la= h_\la(\om_0 )$.
If $\om_0$ is intrinsically  bounded in $W^u(p_0)$  then 
 the Maximum Principle applied to the non-negative subharmonic
 function $G_\la^+|\, W^u(p_\la) $ implies   that $\om_\la \subset K^+_\la$. 

Assume $\om_0 $   is intrinsically unbounded.
Then the Maximum Principle implies 
that  $\om_0 $ is simply connected, so the same holds for $\om_\la$.

 Given a parameter $\la_1$, we will label the corresponding objects
 with ``1'', e.g.,   $f_1\equiv f_{\la_1}$,  $U_1\equiv U_{\la_1}$. 
 Assume  by contradiction that $\om_1 \cap U^+_1
 \neq \emptyset$ for some  $\la_1$. 
We first claim that   $\om_1 \subset U^+_1 $.  Otherwise  $\om_1$
would  intersect 
 $\fr_\intr (W^u(p_1)\cap K_1^+) = \cl_\intr ( H(p_1 ) )$. Since the 
 holomorphic motion preserves $H(p)$, $\om_0$ 
would intersect  $\cl_\intr ( H(p_0 ) )$
 contradicting the fact that $\om_0 $ is contained in $K^+$. 

It follows that $\om_1$ is a simply connected intrinsic component of
 $W^u(p_1 )\cap U_1^+$. 
The existence of such a component implies   
 that $f_1 $ is unstably connected \cite[Theorem 0.1]{bs6},
 so $J_1 $ is connected. Since unstable  connectivity  
is preserved in weakly $J^*$-stable families (see Proposition
\ref{connectivity preserved}), 
$f_0 $ is unstably connected, too. 

 On the other hand, for families of unstably connected polynomial automorphisms 
 we have shown  in Proposition \ref{prop:solenoid} that 
every intrinsic  component of  $W^u(p_\la)\setminus K_\la^+$ can 
 be followed   by some holomorphic motion $\tilde h_\la$
 coinciding with $h_\la$ on $\fr_\intr (W^u (p_0) \cap K^+_0)$. 
But then the action of  $\tilde h_\la$ on the space of
connected components of $W^u(p_0)\sm K^+_0$ 
 must agree with that of $h_\la$, which is impossible for  
the component $\tilde h_1^{-1} (\om_1)$.
This contradiction completes the proof.  
 \end{proof}

\subsection{Proof of Theorem \ref{thm:equiv}}  \label{subs:proof}
Let us show that  $\mathrm { (i) } \Leftrightarrow \mathrm {(ii)}$. 
First, it follows from Proposition~\ref{prop:extending correspondence}  that 
$\mathrm { (i) }  $ is equivalent to the statement $$\mathrm {(ii')} \quad \text{\it  Saddle points stay of saddle type throughout the family}.$$ 
Obviously, (ii) implies $\mathrm {(ii') }$.  
The reverse is also obvious in the dissipative case, since every bifurcation of
a sink gives rise to a saddle. 
In  general, 
we have to rule out the possibility that  in a weakly $J^*$-stable substantial family,  
  a periodic point $q$ bifurcates from attracting  to repelling
  through indifferent without ever turning into  a saddle. 

Assume  by contradiction that such a scenario happens.
 Fix a parameter domain $\La'$  
over which  
$q$ can be followed holomorphically, and its eigenvalues cross the
unit circle. 
So, inside $\La'$ there is a region $\La^-$ where 
  $q(\la)$ is a sink and a region $\La^+$ where   $q(\la)$ is a source. 
Fix a (necessarily persistent) saddle point $p(\la)$.  For $\la\in \La^+$, since   $\abs{\Jac f_\la}>1$, 
 $f_\la$ is   unstably disconnected by  \cite[Cor. 7.4]{bs6}. 
By the weak $J^*$-stability, the same is true for every $\la\in \La$ (see Proposition \ref{connectivity preserved}). 
Then  Lemma \ref{lem:basin in unstable manifold} implies that 
 for $\la\in \La^-$,  there   is a non-trivial bounded component
 $\om_\la$ of $W^u(p(\la)) \cap \mathcal{B}(q(\la))$. 

By Lemma \ref{BR motion}, 
 we infer that  under the Bers-Royden motion  of $W^u(p(\la))$,
 $\om_\la$ 
persists throughout $\La$  as a bounded component  of 
 $W^u(p(\la))\cap K^+_\lambda$. 
Let us consider  the Bers-Royden orbit  $z(\la)$ of some point of $\om_\la$.
Then the family of maps $\La\ra \cc^2$, 
 $\la\mapsto f^n_\la(z(\la)) $,   $n=0,1,\dots$, is locally bounded and hence 
 normal over $\La$.
 Hence by analytic continuation the convergence $f^n_\la(z(\la)) \cv q(\la)$ persists throughout $\La$. But if $\la\in \La^+$, 
 $q(\la)$ is repelling, so we arrive at a contradiction, which finishes the proof of $\mathrm {(ii)} \Leftrightarrow \mathrm {(ii')} $.

\medskip


Condition  (i) implies  (iii) by Lemma \ref{Hausdorff cont}.  
%
Conversely, (iii) implies (ii). Indeed if a periodic point changes type, then arguing as in Lemma \ref{lem:substantial}, we see that for some $\la$, a multiplier of the cycle must cross the unit circle at a linearizable parameter. So at this parameter a Siegel ball or Siegel/attracting basin is created, and the corresponding 
periodic orbit jumps outside $J^*$, thus preventing continuity of $J^*$.

\medskip

 To conclude the proof we show the (rather obvious) chain of implications: 
$$ 
  \mathrm{ (iv)}  \Rightarrow  \mathrm{ (v)}  \Rightarrow  
   \mathrm {(i)}  + \text{ the number of non-saddle cycles is finite }  
  \Rightarrow \mathrm { (iv) } .
$$
Indeed $ \mathrm{(iv)} \Rightarrow \mathrm{ (v)} $ is clear. 
Next, if (v) holds, then all periodic points of sufficiently high prime period are (necessarily persistent) saddles, so by Corollary \ref{cor:dense saddles}, the family is weakly $J^*$-stable. 
Therefore, all periodic points are of constant type, hence (iv) holds.

 The Theorem is proved. \qed

\subsection{Motion of the big Julia set  
$\bJ= J^+\cup  J^-$}\label{subs:extension} 

Let us start with a simple observation:

\begin{lem}\label{respect to K}
   Any equivariant normal \BHM   $\GG$  preserves the sets  $K^\pm$
and hence preserves $K$.  
\end{lem}

\begin{proof}
For definiteness, let us treat  the case of $K^+$.
Let  $\gamma  = (\la, z(\la))$ be a graph of $\GG$
 such that $z(\lo)\in K^+_0$  for some $\lo\in \La$. 
Then the forward orbit  $(f^n_\la(z(\la_0)))_{n\geq 0}$ is bounded. 
By the equivariance,
all the graphs ${\widehat f}^n(\gamma) = (\la, (f^n_\la(z(\la)))$,
${n\geq 0}$,  belong to $\GG$ as well.
By normality of $\GG$, 
the ${\widehat f}^n (\gamma)$ form a normal family. 
Consequently,  this family is locally uniformly bounded in $\La$,  
implying  that  $z(\la)\in K^+_\la$ for all $\la \in \La$. 
\end{proof}

 Recall that a family $(f_\la)$ is said to be weakly $X$-stable if the sets $X_\la$ move under an equivariant \BHM. 
 We now prove the equivalence of several notions of weak stability.
 
\begin{thm}\label{thm:extension}
Let $(f_\la)_{\la\in \La}$ be a substantial 
family of polynomial automorphisms of dynamical degree $d\geq 2$. The following properties are equivalent:
\begin{enumerate}[{\rm (i)}]
\item $(f_\la)_{\la\in \La}$ is weakly $J^*$-stable. 
\item $(f_\la)_{\la\in \La}$ is weakly $J^-$-stable.
\item $(f_\la)_{\la\in \La}$ is weakly $J^+$-stable.
\item $(f_\la)_{\la\in \La}$ is weakly $K$-stable.
\item $(f_\la)_{\la\in \La}$ is weakly $S^-$-stable (resp. $S^+$-stable).
\end{enumerate}

 If $(f_\la)_{\la\in \La}$ is  weakly $J$-stable, then the properties {\rm (i)-(v)} hold. 

In items {\rm{ (i), (iv)}} and {\rm {(v)}}, the motions in question  are automatically
normal, while in items  {\rm{(ii)}}  and {\rm{ (iii) }}  they can be  selected to be so. 
Moreover, in the latter items the motions preserve respectively the unstable and
stable manifolds of  all saddles.
\end{thm}

From now on a family satisfying the equivalent conditions  (i)-(v) of 
this theorem will  be simply referred to as  {\em weakly stable.} 

\begin{rmk}\label{J-comp}
 We do not know if weak stability implies weak $J$-stability
since we cannot rule out a scenario where under a \BHM  of $K = K^+\cap J^-$, 
a point in $J$ moves out to $({\mathrm{Int}}\, K^+)\cap J^-$.  
\end{rmk}  

\begin{proof}   
We start by proving  that (i)$\implies$(ii) (of course (i)$\implies$(iii) for the same reason). 
So, assume  $(f_\la)_{\la\in \La}$ is  weakly $J^*$-stable.
Take a holomorphically moving saddle $p(\la)$, and consider the
Bers-Royden equivariant  holomorphic motion of its unstable manifold $W^u(p(\la))$.
By  Lemma \ref{BR motion}, it respects the decomposition 
$K_\la^+\sqcup U_\la^+$ inside $W^u(p(\la))$.
The motion of $K_\la^+ \cap W^u(p(\la))$ is obviously normal, while  
the motion of $ W^u(p(\la) )\setminus K_\la^+$ is normal by 
Lemma~\ref{lem:normal}.
Hence the motion of the whole unstable manifold $W^u(p(\la))$ is normal
as well. By the $\la$-lemma, it extends to an equivariant  normal
\BHM   of $\overline{W^u } (p(\la)) = J_\la^-$, as desired. 

By Lemma \ref{respect to K}, this motion preserves the
decomposition $J^-= K\sqcup S^-$ (and similarly, for $J^+$), so
(i) implies (iv) and (v) as well.

\medskip
Let us show that (iii)$\implies$(i). 
Consider an arbitrary saddle $p(\la_0)$ and its the graph 
$\gamma_0 = (\la, p(\la))$ of its motion. 
By Proposition \ref{prop:henon}, we may normalize our family so that each
$f_\la$ is a product of Hénon mappings. Then  on a given compact subset of 
$\La$, for sufficiently large $R$,  the set   $J^+\cap \dd_R^2$ is forward
invariant. 
Hence the family of graphs 
$(f^n(\gamma_0))_{n\geq 0}$  is normal%
\footnote{Recall that we do
  not assume any normality in the definition of the weak
  $J^+$-stability}. 
Let $\gamma$ be a cluster graph for this family. 
Then $\gamma(\la)\in K_\la$  for any $\la\in \La$,
while $\gamma(\lo)=p(\lo)$ is an arbitrary saddle.
Proposition \ref{prop:extending correspondence} 
implies $J^*$-stability once again.  Moreover, this argument shows 
that the  motions in question preserve the stable manifolds of  all saddles.

Of course, (ii)$\implies$(i) for the same reason.
Furthermore, it follows directly from Proposition \ref{prop:extending
  correspondence} that (iv)$\implies$(i).  Likewise,
  weak $J$-stability implies (i) as well.

Let us  show that  (v)$\implies$(iv). 
Assume $(f_\la)_{\la\in \La}$ is weakly $S^-$-stable.
By Lemma \ref{lem:normal}, the corresponding BHM is normal,
so it extends to a normal BHM $\GG$ of $\overline{S^-}$.
By Lemma \ref{lem:boundary},  $J^*\subset \overline{S^-}$.
By Lemma \ref{respect to K}, the graphs of $\GG$ that begin in $J^*$
for some $\la_0\in \La$ remain in $K$ for all $\la\in \La$.
Applying Proposition \ref{prop:extending correspondence}  once again,
we obtain the desired.

\medskip The last assertion on normality and preservation of s/u manifolds has been established along the
lines of the proof.
\end{proof}


\begin{cor}
If $(f_\la)_{\la\in \La}$ is weakly stable, there exists an equivariant 
normal  \BHM of $\bJ$   
that preserves  the 
stable and unstable manifolds of saddle periodic points. In particular $\bJ$ 
moves continuously in the Hausdorff topology. 
\end{cor}

\begin{proof}
The former assertion follows directly from the theorem.
The latter follows from Lemma~\ref{Hausdorff cont}.
\end{proof}

\begin{rmk}
It is not difficult to show that 
$(f_\la)$ is weakly stable 
if and only if $\la\mapsto J^+_\la$ is continuous  for the Hausdorff 
topology (for $J^*$ this was done in Theorem \ref{thm:equiv}). On the other hand this is false for $J^-$. Indeed   
in the dissipative setting $K^-= J^-$, and it is classical that $K^-$ moves upper semi-continuously while $J^-$ moves 
lower semi-continuously (see e.g. \cite[Prop. 4.7]{bsu}).  In particular $\la\mapsto J^-_\la$ is always continuous. 
\end{rmk}

%

According to \cite{lyubich peters}, non-wandering components of 
${\mathrm{Int}}\,( K^+)$ of a moderately dissipative polynomial
automorphism
 $f: \cc^2\rightarrow \cc^2$ can be classified
as attracting, parabolic, or rotation  basins.
(For components of ${\mathrm{Int}}_i\, (W^u(p)\cap K^+)$,  
there is one more theoretical  option:
they can be contained in the small Julia set $J$.) We cannot rule out
that some of these components change type under a branched holomorphic
motion of the Julia set
(compare Remark \ref{J-comp}).
Arguing as in the proof of Theorem \ref{thm:equiv} we get: 

\begin{prop}
  Under a branched holomorphic motion over a weakly stable domain, if for  some parameter $\lo$, 
   $z_0$ is a point  in  $J^-_0\cap K^+_0$ belonging to the basin of attraction of a sink $q_0$ then  for every $\la$, $z(\la)$ stays in the basin of $q(\la)$.  
\end{prop} 
  


\part{Semi-parabolic implosion and homoclinic tangencies}\label{part:tangencies}

 \section{Semi-parabolic dynamics}\label{sec:semi parabolic}
 
 In this paragraph we collect some basic facts about semi-parabolic dynamics: basins, petals, etc., following the work of  Ueda 
 \cite{ueda, ueda2},   Hakim \cite{hakim} (see  also \cite{bsu}). 
  Let $f$ be a  polynomial automorphism of $\cd$. A periodic point $p$ is {\em semi-parabolic} if its multipliers are 1 (or more generally a root of unity) 
 and $b$ with $\abs{b}<1$. Notice that this forces $f$ to be dissipative. Replacing $f$ by $f^q$ for some $q\geq 1 $  we can assume that $p$ is fixed. 
 Then, if we denote by $k+1$ the multiplicity at 0 of  $f  - \mathrm{id}$ 
 (which is finite because $f$ has no curve of fixed points), 
  there exist  
  local coordinates $(x,y)$ in the neighborhood of $p$ such that $p=(0,0)$ and $f$ is of the form 
  \begin{equation}\label{eq:ueda hakim}
 (x,y)\mapsto (x + x^{k+1} + Cx^{2k+1} + x^{2k+2} g(x,y), b y + xh(x,y)),
 \end{equation}
 where $g$ and $h$ are holomorphic near the origin and $C$ is a complex number (see \cite[Prop.~2.3]{hakim}). 
  Notice that in these coordinates, $\set{x=0} = W^{ss}_{\rm loc}(0)$ 
  is the local (strong) stable manifold of $0$, and $f\rest{\set{x=0}}$ is linear. 
  For $r>0$, the flower-shaped  open set $\set{x, \ \abs{x^k+ r^k}<r^k}$ admits $k$ connected components, which will be denoted by $P^\iota_{r,j}$,  $0\leq j\leq k-1$. Then for small $\eta>0$, 
  the domains $\mathcal{B}_{r,j,\eta}:= P^\iota_{r,j}\times \dd_\eta$ are attracted to the the origin under iteration. Finally, let
  $\mathcal{B}_j = \bigcup_{n\geq0} f^{-n}(\mathcal{B}_{r,j,\eta})$. The open sets $\mathcal{B}_j$ are biholomorphic to $\cd$ and are the components of the basin of attraction of $p$ in $\cd$. 
  
To be more specific, in   $\mathcal{B}_{r,j,\eta}$ we change coordinates by letting $(z,w) = ((kx^k)^{-1}, y)$, 
so that in the new coordinates, $f$ assumes a form
$$(z,w)\longmapsto \left(z-1+ \frac{c}{z} + O\lrpar{\unsur{\abs{z}^{1+1/k}}}, bw +   O\lrpar{\unsur{\abs{z}^{ 1/k}}}\right) ,$$
  where $c$ is a complex number depending on $C$. Notice that 
   in the new coordinates, 
   $\mathcal{B}_{r,j,\eta}$ corresponds to a region of the form $\set{\Re(z)<-M} \times \dd_\eta$, 
   with $M = (2kr^k)^{-1}$. 
  Therefore if we set $$w^\iota(x,y)=w^\iota(x) = \frac{1}{kx^k}+ c \log \unsur{kx^k}$$ we infer that the limit $$\varphi^\iota(x,y)  = \lim_{n\cv\infty} \lrpar{w^\iota(f^n(x,y)) +n}$$ 
  exists and satisfies the functional equation 
  $\varphi^\iota\circ f  = \varphi^\iota-1$ (beware that this normalization 
  differs from   the  references mentioned above). 
 In addition, $\varphi^\iota - {w^\iota}$ is a bounded holomorphic function in $\mathcal{B}_{r,j,\eta}$. In the paper, the letter 
  $\iota$  will stand for ``incoming" and $o$ for ``outgoing", following a 
convenient notation from \cite{bsu}. 
 
 It easily follows that in the original coordinates, if $(x,y)\in \mathcal{B}$ then $f^n(x,y) = (x_n, y_n)$ with 
 ${x_n}\sim   (kn)^{-1/k}$ and  ${y_n} = O(n^{-1/k})$, see \cite[Prop. 3.1]{hakim} (beware that $y_n$  needn't be exponentially small). 
 
Fix a component $\mathcal{B} = \mathcal{B}_j$ of the basin of attraction.  By the iteration
 we can extend $\varphi^\iota$ to $\mathcal{B}$. It turns out that $\varphi^\iota: \mathcal{B} \cv \cc$ is a fibration 
 \cite[Thm 1.3]{hakim}, and that there 
 exists a function $\phi_2: \mathcal{B} \cv \cc$ such that $\Phi = (\varphi^\iota, \phi_2):   \mathcal{B} \cv \cd$ is a biholomorphism. 
 
 The following result is similar to \cite[Thm 1.2]{bsu}, and its proof will be left to the reader. 
 
\begin{prop}
If $p_1$ and  $p_2$ are points in $\mathcal{B} $ such that $\varphi^\iota(p_1) = \varphi^\iota(p_2)$ then 
$$\lim_{n\cv +\infty} \unsur{n}\log \dist(f^n(p_1), f^n(p_2) )  = \log \abs{b}<0.$$
 On the other hand, if  $\varphi^\iota(p_1) \neq \varphi^\iota(p_2)$
 then   $ \dist(f^n(p_1), f^n(p_2))$ decreases  like  $n^{-(1+1/k)}$. 
 \end{prop}
 
From now on we will refer to the foliation $\set{\varphi^\iota = C^{ st}} $ as the {\em strong stable foliation} in $\mathcal{B}$, and it will 
be denoted by $\mathcal{F}^{ss}$. Its structure near the origin is easy to describe.   Indeed since $\varphi^\iota - {w^\iota}$ is 
bounded near the origin, it follows from Rouché's theorem that  the leaf of $\mathcal{F}^{ss}$ 
through $(x_0,0)$ in 
 $\mathcal{B}_{r,j,\eta}$ is   graph over the vertical direction, whose distance   to the line $\set{x=x_0}$ 
 tends to 0 as $x_0\cv 0$. In particular 
 $\mathcal{F}^{ss}$ extends continuously 
 to  $\mathcal{B}_{r,j,\eta}\cup \set{x=0} =  \mathcal{B}_{r,j,\eta}\cup W^{ss}_{\rm loc}(0)$
  by adding $W^{ss}_{\rm loc}(0)$ as a leaf. 
  

   \medskip
   
 On the ``outgoing" side, there is also a notion of ``repelling petal'', which is defined as follows.  With coordinates as in 
 \eqref{eq:ueda hakim},  denote by $P^o_{r,j}$, $0\leq j\leq k-1$,  the connected components of 
  $\set{x, \ \abs{x^k- r^k}<r^k}$. Then for small $r, \eta>0$ the set $\Sigma_{j, {\rm loc}} $ defined by 
  $$\Sigma_{j, {\rm loc}}  = \set{ q \in  P^o_{r,j}\times \dd_\eta: \ \forall n\geq 0, f^{-n}(q)\in P^o_{r,j}\times \dd_\eta \text{ and } f^{-n}(q)\underset{n\cv\infty}{\longrightarrow} 0}.$$ 
  Then $\Sigma_{j, {\rm loc}}$ is a graph $\set{y=\psi (x)}$ over the first coordinate, 
  which   extends continuously to the origin
  by putting $\psi(0) = 0$ (this extension cannot be made holomorphic, see \cite[Prop. 1.3]{bsu}).  This is stated in 
  \cite[Thm 11.1]{ueda2} only for $k=1$, however the proof relies on a more general result \cite[Lemma 11.2]{ueda2} 
  which allows to treat the  general case as well. 
  As usual we extend the petal globally to $\cd$ by letting 
$$
  \Sigma_j = \bigcup_{n\geq 0} f^n\left(\Sigma_{j, {\rm loc}} \right).
$$ 
  Then $\Sigma_j$  is biholomorphic to $\cc$ 
(this is stated  only for $k=1$ in \cite[Thm 11.6]{ueda2} but the
adaptation to the general case  is straightforward).  
The union  $\bigcup_j \Sigma_j$ is the set of points converging to the
semi-parabolic point  $p$ under backward iteration, 
referred to as the {\em asymptotic curve} in \cite{ueda2} and \cite{bsu}.  
We will also call it the {\it repelling petal} (or simply  {\it unstable manifold}) of $p$.

 \section{Critical points in basins}\label{sec:critical}
 
  In this section we prove the existence of critical points in semi-parabolic basins for 
  sufficiently dissipative maps. Let $f$ be a   
  polynomial automorphism with a semi-parabolic basin $\mathcal{B}$. 
  Recall that by a critical point, we mean a point of tangency between 
  the strong stable foliation in $\mathcal{B}$ and the unstable manifold of some saddle periodic point. 
 The argument is based on  a refined version of some classical
 properties of  entire functions of finite order: see \S\ref{subs:entire}.  
 The proof of Theorem \ref{theo:critical} comes in \S\ref{subs:critical}.
  These results will be generalized  to attracting basins in Appendix \ref{sec:attracting}.

 \subsection{Entire functions of finite order}  \label{subs:entire}
Let $f:\cc\cv\cc$ be an entire function. The {\em order} of $f$ is  defined as
$$\rho(f) = \limsup_{r\cv\infty} \frac{\log^+\log^+ {M(r,f)}}{\log r}, \text{ where } M(r,f) = \max\set{\abs{f(z)}, \ \abs{z} =r}.$$   
The class of entire functions of finite order is well-known to display a number of remarkable  properties, some of which we recall now. 
We say that $a\in \cc$ is an {\em asymptotic value} of $f$ if there exists a continuous   path $\gamma:[0, \infty)\cv \cc$ 
tending to infinity such that   $f(\gamma(t))\cv a$ as $t\cv\infty$.  
The famous Denjoy-Carleman-Ahlfors Theorem asserts that if $f$ is an entire function of order $\rho<\infty$, then 
it admits at most  $2\rho$ distinct asymptotic values 
(see e.g. \cite[Chap. 5]{goldberg ostrovskii} or \cite{langley}). 
Another essentially equivalent formulation is that for every $R>0$, the open set 
$\set{z, \abs{f(z)}>R}$ admits at most $\max(2\rho, 1)$ connected components. 

When $\rho<\frac12$, we see that $f$ has no asymptotic values. One can actually be more precise in this case. Indeed, 
Wiman's theorem \cite[Chap. 5, Thm 1.3]{goldberg ostrovskii}
asserts that there exists a sequence of circles $\set{\abs{z}=r_n}$  with radii $r_n\cv\infty$ such that 
$\min \set{\abs{f(z)}, \ \abs{z} =r_n} \cv \infty$.

To prove the existence of critical points in semi-parabolic basins we will require a slight generalization of the 
Denjoy-Carleman-Ahlfors theorem on asymptotic values. We say that $a$ is an {\em $\e$-approximate asymptotic value} of $f$ if 
there exists a continuous 
path $\gamma:[0, \infty)\cv \cc$ 
tending to infinity such that  $\limsup_{t\cv\infty} \abs{f(\gamma(t)) - a}<\e$. The statement is as follows:

\begin{thm}\label{thm:fuzzy DCA}
 Let $f$ be an entire function of finite order. Assume that $f$ admits $n$ distinct $\e$-approximate asymptotic values 
 $(a_i)_{i=1, \ldots, n}$, with $\e< \min_{i\neq j}\frac{\abs{a_i-a_j}}{5}$.

Then the order of $f$ is at least $n/2$. 
\end{thm}

In order to prove the theorem let us first recall the classical Phragmen-Lindelöf Principle. 

\begin{prop}
Let $D$ be an unbounded domain in $\cc$. 
 Let $f$ be a bounded holomorphic function on $D$, such that $\limsup_{  \fr D \ni z\cv\infty } \abs{f}\leq \delta$. Then $\limsup_{D\ni  z\cv\infty } \abs{f}\leq \delta$. 
\end{prop}

The following result is a version of a  classical Lindelöf Theorem 
(the argument below is adadpted from   \cite[Thm 12.2.2]{langley}). 

\begin{thm}
 Let $D$ be a simply connected unbounded domain in $\cc$, whose boundary  consists of
two simple curves $\gamma_1$, $\gamma_2$ both tending to infinity, and disjoint apart from their common starting point.
Let $f$ be holomorphic on $D$ and continuous on $\fr D$, and assume that when $z$ goes to infinity along 
$\gamma_i$, $f$ has the property that  $\limsup_{t\cv\infty} \abs{f(\gamma_i(t)) - a_i}<\e$, with  $\e<  \frac{\abs{a_1-a_2}}{5}$. Then $f$ is unbounded on $D$. 
\end{thm}

\begin{proof}
 Assume by contradiction that $f$ is bounded, and let $g(z) = (f(z)-a_1)(f(z)-a_2)$. 
 Then $g$ is bounded on $D$, and $\limsup\abs{g(z)}\leq \delta$, 
  as $z\cv\infty$ along $\fr D$, for some  $\delta<\frac65 \abs{a_1-a_2} \e$. It follows that 
$\limsup_{ D\ni z\cv\infty } \abs{g}\leq \delta$.  

Now for every $R>0$ there exists a curve $\Gamma$ in $D$ joining 
$\gamma_1$ and $\gamma_2$ and  staying at distance at least $R$ from the origin. If $R$ is large enough, we then have that 
$\abs{g}< \frac65 \abs{a_1-a_2} \e$ along $\Gamma$. Furthermore at  $ \Gamma\cap \gamma_1$ (resp. 
$ \Gamma\cap \gamma_2$), $f$ is $\e$-close to $a_1$ (resp. $a_2$). So there exists $z_0\in \Gamma$ such that 
$\abs{f(z_0)-a_1} = \abs{f(z_0)-a_2} \geq \frac{\abs{a_1-a_2}}{2}$. We infer that 
$$\abs{g(z_0)} \geq  \frac{\abs{a_1-a_2}^2}{4} \geq \frac54 \abs{a_1-a_2} \e,$$
a contradiction.
\end{proof}

\begin{proof}[Proof of Theorem \ref{thm:fuzzy DCA}] (compare \cite[Cor. 14.2.3]{langley})  
By assumption  there are $n$ curves $\gamma_i$ going to infinity along which $f$ $\e$-approximately
 converges to $a_i$. We may assume that all these curves  are simple, start from 0, 
  and intersect only at 0. We reassign the indices so that the curves are arranged in clockwise order. By the previous theorem $f$ must be unbounded in the domain comprised between $\gamma_i$ and $\gamma_{i+1}$ (here we put $\gamma_{n+1} = \gamma_1$). Therefore the order of $f$ is at least $n/2$ by the ordinary Denjoy-Carleman-Ahlfors Theorem.
\end{proof}

\begin{rmk}\label{rmk:eremenko}
  Alex Eremenko has pointed out to us the following version of the 
Denjoy-Carleman-Ahlfors  Theorem. Let $f$ be an entire function bounded in the left-half
 plane and outside horizontal  strips. Let 
$$
  M(x) = 
  \max_{{\mathrm {Re}}\, z  = x} \abs{f(z)}
\text{ and } 
  \rho=  \rho(f) = \limsup_{x\to +\infty}\frac{\log M(x)} {x}. 
$$
Then $f$ admits at most $2\rho$ asymptotic values. 
[This follows from the subharmonic version of the DCA Theorem applied to the function
$u(z)= \log^+ (f(\log z)/ M)  $, where $M$ is the supremum of $|f|$ outside the half-strip
$ \Pi=  \{ {\mathrm {Re}} \, f > 1, \; |{\mathrm {Im}}|\, f <\pi \} ,$
and $\log z$ is the principal value of the logarithm in $\cc\setminus {\mathbb R}_-$.]

This Theorem admits an $\e$-approximate version similar to 
Theorem~\ref{thm:fuzzy DCA}. 
\end{rmk}

 \subsection{Semi-parabolic basins}\label{subs:critical}
In this paragraph we prove Theorem \ref{theo:critical}.
%
We first recall that stable and unstable manifolds of saddle points, as well as    strong stable 
manifolds of semi-parabolic points are entire curves, whose parameterizations are defined dynamically. More precisely, if 
$p$ is a fixed  point with  an expanding  eigenvalue  $\kappa^u$, 
then the associated stable manifold is parameterized by an entire function 
$\psi^u:\cc\cv\cd$ satisfying  $\psi^u(0) = p$ and  $f\circ \psi^u(t) = \psi^u(\kappa^u t)$ for every $t\in \cc$, 
and similarly for a contracting  eigenvalue. 
Let us now make an easy but important observation 
(see \cite{EL} in the one-dimensional setting,  
\cite{bs6,jin} in our setting, and \cite{cantat croissance} for automorphisms of compact projective surfaces).

\begin{lem}\label{thm:jin}
 Let $q$ be a fixed  point of a polynomial automorphism of dynamical degree $d\geq 2$
with  an expanding  eigenvalue  $\kappa^u$, and  $\psi^u : \cc\cv\cd$ is a parameterization of the associated 
unstable manifold as above. Then the coordinates of $\psi^u$ are   entire functions of finite order 
$$\rho = \frac{\log d}{\log \abs{\kappa^u}}.$$
\end{lem}

Notice that any other parameterization has the same order, since two
parameterizations differ from an affine map of $\cc$. 
So we may speak of the order of the unstable manifold $W^u(q)$.

 
In the semi-parabolic case the result specializes as follows:

\begin{cor}
If $p$ is a semi-parabolic  
periodic point for a polynomial automorphism of dynamical degree $d\geq 2$, then the order of $W^{ss}(p)$ is 
$$ \frac{\log d}{\log \abs{\Jac f}^{-1}}.$$
\end{cor}

\begin{proof} Apply  Lemma \ref{thm:jin} to $f^{-k}$, where $k$ is the period of $p$. \end{proof}

 Our use of this corollary will be the following.

\begin{cor}\label{cor:point}
Let $f$ be a polynomial automorphism of $\cd$ of dynamical degree
$d\geq 2$, with a semi-parabolic 
periodic point $p$. 
 Assume that the   Jacobian of $f$   satisfies 
 $\abs{\Jac f}<\unsur{d^2}$. Then   the connected component of $p$ in 
 $W^{ss}(p)\cap J^-$ is $\set{p}$. 
 \end{cor}

\begin{proof}
The assumption on the Jacobian together with the previous corollary imply that the order of $W^{ss}(p)$ is smaller than $\frac12$. 
Then by Wiman's theorem there exists a sequence of circles $\set{\abs{t} = r_n}$ such that the second
 coordinate (say) of $\psi^s$ satisfies $\min_{\set{\abs{t} = r_n}} \abs{\psi^s(t)} \cv\infty$. Since $K$ is bounded, these circles must be 
 eventually disjoint from $K$. So  we infer that the connected component of $p$ in 
 $(\psi^s)^{-1}(J^-)$ is bounded in $\cc$.  
 Since in addition the component is invariant under multiplication by $\kappa^s$, we are done. 
 \end{proof}

Theorem \ref{theo:critical} now clearly follows from the previous corollary, together with the following result.

 \begin{prop}\label{prop:critical}
 Let $f$ be a polynomial automorphism of $\cd$ of dynamical degree $d\geq 2$,  possessing a semi-parabolic periodic point  $p$
 with basin of attraction $\mathcal{B}$. Assume that the connected component of $p$ in $W^{ss}(p)\cap J^-$ is reduced to $\set{p}$. 
 Then for every saddle periodic
  point $q$, every component of $W^u(q)\cap \mathcal{B}$ contains a 
 critical point. 
 \end{prop}
 
 Notice that in the situation of the proposition,    Proposition \ref{prop:basin unstable} guarantees that  
$W^u(q)\cap \mathcal{B}$ is never empty.

\begin{rmk}
Proposition \ref{prop:critical} remains true in the  case where  $q$ is a semi-parabolic point rather than
a saddle (in particular, when $q=p$). 
 The proof is the same except it makes use of the version of the Denjoy-Carleman-Ahlfors Theorem stated in 
Remark \ref{rmk:eremenko}
\end{rmk}

 \begin{proof}
 Let $q$   be as in the statement of the proposition, and $\psi^u$ as above be a parameterization of $W^u(q)$. Translating the
 coordinates and  iterating if needed we may assume that the semi-parabolic point is fixed and equal to $0\in \cd$. 
 Also we may assume that $q$ is fixed.
 Let $\om\subset \cc$ be a component of $(\psi^u)^{-1} (\mathcal{B})$ --which must be non-empty by Proposition \ref{prop:basin unstable}. By the maximum principle, $\om$ is biholomorphic to a disk. Recall that the incoming Fatou function $\varphi^\iota$ was defined 
 in Section \ref{sec:semi parabolic}. 
 Observe first that $\varphi^\iota\circ\psi^u:\om\cv\cc$ cannot be constant for otherwise 
 $W^u(q)$ would coincide with a strong stable leaf, which cannot happen since it would then be contained in the compact set $K$. We argue by contradiction, so assume that $\om$ contains no critical point. 
 Then $\varphi^\iota\circ\psi^u:\om\cv\cc$ is a locally univalent map. Since 
  $\varphi^\iota\circ\psi^u$ cannot be a covering, it 
  must possess asymptotic values, that is, there exists a path $\gamma:[0, \infty)\cv \om$, tending to infinity in $\om$ (that is, converging to $\fr\om$ or to infinity in $\cc$) such that 
  $\varphi^\iota\circ\psi^u(\gamma(t))$ has a well defined limit in $\cc$ as $t\cv\infty$ (see \cite[p. 284]{nevanlinna}
   or \cite[Lemma 1.2]{goldberg keen} for a modern presentation). 
   Notice that by definition, $\psi^u(\gamma)$, as well as all its iterates and cluster values,  are contained in $K^+\cap J^-$. 
 
Recall that there exists a function $\phi_2:\mathcal{B}\cv\cc$ such that $\Phi:=(\varphi^\iota, \phi_2):\mathcal{B}\cv\cd$ is a
 biholomorphism. At this point the proof splits into two   cases. 

\medskip

\noindent{\bf Case 1:}  $\phi_2\circ\psi^u(\gamma )$ is unbounded. 

\smallskip

Observe that this case must occur when $\Phi\circ \psi^u :\om\cv\cd$ is proper, which happens for instance when $\om$ is relatively compact in $\cc$. Consider a domain $\mathcal{B}_{r,j,\eta}$ as in \S \ref{sec:semi parabolic}, 
corresponding to the basin $\mathcal{B}$, and  
in which the strong stable foliation is made of vertical graphs, clustering at $W^{ss}_{\rm loc}(0) = \set{x=0}$. 
For sufficiently large $n$, the connected component of 
$f^n(\psi^u(\gamma))\cap \mathcal{B}_{r,j,\eta}$ containing $f^n(\psi^u(\gamma(0)))$  is a   path 
 which by our unboundedness assumption goes up to the horizontal  boundary $P^\iota_{r,j}\times \fr \dd_\eta$ 
 of $\mathcal{B}_{r,j,\eta}$. When $n$ is large, this connected component is contained in a small neighborhood of 
 $W^{ss}_{\rm loc}(0)$. In addition  $f^n(\psi^u(\gamma(0)))$ converges to 0 when $n\cv\infty$. Thus, taking 
 a cluster value of this sequence of paths for the Hausdorff topology, we obtain a closed connected subset of 
 $W^{ss}_{\rm loc}(0)\cap J^-$, containing 0, and touching the boundary, hence not reduced to a point. This contradicts 
 Corollary \ref{cor:point}, and finishes the proof in this case. 
 
 \medskip

 \noindent{\bf Case 2:}  $\phi_2\circ\psi^u(\gamma )$ is bounded. 

\smallskip

A first observation is that under this assumption, the path $\gamma$ must go to infinity in $\cc$. Indeed otherwise let $(t_n)$ be a 
sequence such that $\gamma(t_n)$ converges to $\zeta\in \fr\om\subset \cc$. Then $\psi^u(\gamma(t_n))$ converges to 
$\psi^u(\zeta) \notin \mathcal{B}$, contradicting the fact that $\Phi(\psi^u(\gamma(t_n)))$ stays bounded in $\cd$. In particular $
\gamma$ is an asymptotic path for the entire mapping $\psi^u$

\medskip

For the sake of explanation,  assume first  that $\phi_2\circ\psi^u(\gamma(t ))$   admits a limit  as $t\cv\infty$. Thus  
$\Phi\circ \psi^u(\gamma(t ))$ converges in $\cd$, and $\psi^u(\gamma(t ))$ converges to some limiting point 
$\omega\in \mathcal{B}$, which 
must be an asymptotic value of $\psi^u$ (that is, both coordinates are asymptotic values of the coordinate functions of $\psi^u$). 
By the invariance of  $W^u(p)$, all iterates $f^n(\omega)$, $n\in\zz$, are asymptotic values of $\psi^u$. Since $\psi^u$ has finite order, this contradicts the Denjoy-Carleman-Ahlfors theorem.   

\medskip
 
 In the general case we   use Theorem \ref{thm:fuzzy DCA} instead. Let $K$ be the cluster set of $\psi^u(\gamma)$, which is a compact subset of $\mathcal{B}$, contained in a leaf $\set{\varphi^\iota = C^{st}}$ of the strong stable foliation. 
Let us study the shape of $f^n(K)$. When $n$ is   large enough, 
 $f^n(K)$ is contained in a small neighborhood of the origin, inside a local strong stable leaf. Perform a {\em linear} change of coordinates so that in the new coordinates, $f$ expresses as $f(x,y) = (x, by)+ h.o.t.$.  These coordinates are tangent (at 0), but not equal, to the adapted local coordinates of \S \ref{sec:semi parabolic}.
 In a bidisk near 0, the strong stable foliation is made of vertical graphs and 
  $W^{ss}_{\rm loc}(0)$ is a vertical graph tangent to the $y$-axis. Let $(x_0, y_0)\in K$ and 
  $(x_n, y_n) = f^n(x_0, y_0)$. Then we infer that $x_n\sim (kn)^{-1/k}$ and $f^n(K)$ is a subset of the leaf 
  $\mathcal{F}^{ss}(x_n,y_n)$ of size exponentially small with $n$. Denoting by $\mathrm{pr}_1: (x,y)\mapsto y$ 
    the first projection, we deduce  that  $\mathrm{pr}_1(f^n(K))$ is a set of exponentially small diameter about $x_n$. It follows that for 
   every integer $k$ there exists  $\e>0$ and  integers $n_1, \ldots , n_k$ 
  such that the sets  $\mathrm{pr}_1(f^{n_j}(K))$, $1\leq j\leq k$ are of 
  diameter smaller than $\e$, and $5\e$-apart from each other.  
  Notice that the sets $\mathrm{pr}_1(f^{n_j}(K))$ are 
 $\e$-approximate asymptotic values of $\mathrm{pr}_1\circ\psi^u$ in the sense of Theorem \ref{thm:fuzzy DCA}. 
 Now since $\mathrm{pr}_1$ is linear, $  \mathrm{pr}_1\circ\psi^u$
  is an entire function of finite order.  
  Thus   we obtain a  contradiction with Theorem  \ref{thm:fuzzy DCA}, and the proof is   complete. 
 \end{proof}
 

 \section{Semi-parabolic bifurcations and transit mappings}\label{sec:tv}
 
 In this section  we develop an analogue of  the ``tour de valse" of Douady and Sentenac \cite{tv} 
 in  the context of  semi-parabolic implosion.  When the family  can be put in  the form 
 $$ f_\e(x,y) = ( x+ (x^2+\e^2) \alpha_\e(x,y), b_\e (x)y +
 (x^2+\e^2)\beta_\e (x,y)),$$ 
  (this corresponds to the case $k=1$ in \eqref{eq:flambdaq} below), we may directly 
  appeal to the results of  Bedford, Smillie and Ueda \cite{bsu}. 
 In our setting, however, 
 we  have to deal with periodic points and bifurcations of a more general nature, and it is unclear  
  how to extend the results of \cite{bsu}.  
  
\medskip

Consider  a family $(f_\la)_{\la\in \La}$ of dissipative polynomial automorphisms of $\cd$, with a periodic point changing type, that is,
 one multiplier crosses the unit circle. It is no loss of generality to assume that $\La$ is the unit disk. 

Let us start with a few standard reductions. Recall that for polynomial automorphisms all periodic points are isolated.  Replacing $f_\la
$ by some iterate, we may assume the bifurcating periodic point is fixed. Passing to a branched cover of $\La$ if necessary, the fixed 
point moves holomorphically, so we assume  it is equal to $0\in \cd$. Since $f_\la$ is dissipative the multipliers depend
 holomorphically on $\la
$, and we denote them by $\rho_\la$ and $b_\la$, with $\rho_0 =
e^{2\pi i \frac{p}{q}}$ and $\abs{b_\la}<1$ for all $\la\in \La$.  
We further assume that $\rho_\la$
crosses $\fr\dd$ with non-zero 
  speed, i.e. $\frac{\fr \rho}{\fr\la} \big
\vert_{\la=0} \neq 0$.  

 \subsection{Good local coordinates}
 
 The first step is to find adapted local coordinates. 
 which is a parameterized version of the discussion in \S\ref{sec:semi parabolic}. This  is analogous to Proposition 1 in \cite{tv}.

\begin{prop} \label{prop:good coord}
If $(f_\la)$ is as above, then
for $\la$ sufficiently close to 0, 
there exists a local change of coordinates $(x,y) = \varphi_\la(z,w)$ in which $f_\la^q$ takes the form
\begin{equation}\label{eq:flambdaq}
f_\la^q(x,y) = (\rho_\la^q x + x^{k+1} + x^{k+2}g_\la(x,y), b_\la^q y + x h_\la(x,y))
\end{equation}
with $g_\la$ and $h_\la$ holomorphic and $h_\la(0,0) = 0$. %
%
Moreover $k = \nu q$ for some integer $\nu\geq 1$. 
\end{prop}

\begin{proof} 
Since we are working locally near 
$(0, 0)\in \La\times \cd $,
 we   freely reduce the domains of definition in $(\la,x,y)$ when
 necessary. 
We  will also feel free to use to the same symbols 
for coordinates  after successive changes of variable.

 Recall from \S \ref{sec:semi parabolic}  that there exists an integer $k$ and local coordinates in which $f_0^q$ is of  the form 
 \begin{equation}\label{eq:ueda}
 (x,y)\mapsto (x + x^{k+1} + x^{k+2}g(x,y), b_0^qy + xh(x,y)).
 \end{equation}
  Then $f^q$ admits $k$ (open) attracting petals and 
 $k$  repelling petals, which are  permuted by $f$. 
These petals approach  $0$ at certain directions permuted by the
differential $D_0 f$,
 so necessarily $k=\nu q$ for some nonzero integer $\nu$.  



From now on for notational ease we replace $f^q$ by $f$, that is we assume   $\rho_0=1$. 
Since the differential $D_0  f_\la$ 
is diagonalizable, there exists a ($\la$-dependent) linear change of coordinates
so that $f_\la(x,y)$ takes the form  $(\rho_\la x, b_\la y)+ h.o.t.$ 
There exists a local strong stable manifold tangent to the $y$-axis;
we change coordinates so that it becomes $\set{x=0}$. 
By the Schröder theorem we can linearize $f_\la\rest{\set{x=0}}$, holomorphically in $\la$. Hence in 
the new coordinates, $f_\la(0,y ) = ( 0, b_\la y)$, so that 
$$
 f_\la (x,y) = (\rho_\la x(1+ O(x)) , b_\la y + xh_\la(x,y)).$$ Of
 course $h_\la(0,0) = 0$ since the linear part  is $(\rho_\la x, b_\la
 y)$. 
From now on all changes of variables will be ``horizontal'',  i.e. of the
type $(x,y)\mapsto (x(1+O(x,y)), y)$,   
so the form of  the  second coordinate persists and  we focus on the first one.

Express $f_\la(x,y)$ as 
\begin{equation}\label{eq:aj} f_\la(x,y) = 
(\rho_\la a_1(\la, y) x + a_2(\la, y) x^2 +\cdots + a_j(\la,y)x^j + \cdots , b_\la y + O_{\la,y}(x)),
\end{equation}
where the $a_j$ are holomorphic and $a_1(\la, 0) = 1$.
To start with, for $\la = 0$, we   put $f_0$ in form \eqref{eq:ueda}, so that
 for $j\leq k$, $a_j(0, y) =0$ and $a_{k+1}(0,y) = 1$.

The first task is to arrange so that $a_1\equiv 1$. This is similar to \cite{ueda}. For this we look for a change of coordinates of the form  
$(X,Y) = (x \varphi_\la(y), y)$, with $\varphi_\la(0) =1$. Using the notation $f_\la(x,y) = (x_1, y_1) $ (and similarly in the $(X,Y)$ variables), we infer that 
$$X_1 = \rho_\la a_1(\la, Y) \frac{\varphi_\la(Y_1)}{\varphi_\la(Y)}  X + O(X^2) = 
\rho_\la a_1(\la, Y) \frac{\varphi_\la(b_\la Y)}{\varphi_\la(Y)}  X + O(X^2) . 
$$
 Therefore we see that to obtain the desired form, it is enough to choose 
 $$\varphi_\la(y)  = \prod_{n=0}^\infty a_1(\la, b_\la^ny),$$
 which is locally a convergent product since $a_1(\la, y) = 1+O_\la(y)$ and $\abs{b_\la}<1$. 
 Notice also that for $\la=0$, $\varphi_0(y) =1$ so the change of variables is the identity. In particular $f_0$ remains of the form \eqref{eq:ueda}.

 \medskip
 
We then argue by induction.  So assume that we have found coordinates such that 
  for some $j\leq k$, $a_2(\la,y)= \cdots = a_{j-1}(\la,y)=0$, and 
  $f_0$ remains under the form \eqref{eq:ueda} .
  Put  $$(X,Y) = \left(x+\frac{a_j(\la,y)}{\rho_\la - \rho^j_\la} x^j, y\right).$$   
 Notice that  since $a_j(0) = 0$ and $\rho_\la-1$ has a simple root at the origin, the change of coordinates  
 is  also well defined at $\la=0$. 
 Now, for $\la\neq 0$, since 
the term $a_j(\la,y)$ is non-resonant,   
 a classical explicit  computation (see e.g.  \cite[Thm 6.10.5.]{beardon iteration}) shows that   it disappears
in the new coordinates. 
Hence by continuity the same holds for $\la=0$.  

Moreover, for $\la=0$ the change of coordinates is of the  form 
 $(X,Y) = (x+A_j(y)x^j, y)$, so
$(x,y) = (X-A_j(Y)X^j + h.o.t., Y)$. 
 In the new coordinates we obtain 
\begin{align*}
X_1 &= x_1+A_j(y_1)x_1^j  = x+x^{k+1} + O(x^{k+2}) + A_j(y)(x+x^{k+1} + O(x^{k+2}))^j\\
&= x+A_j(y)x^j + x^{k+1} + O(x^{k+2}) \\
&= X + X^{k+1} + O(X^{k+2}),
\end{align*} 
 so $f_0$ remains of  form \eqref{eq:ueda} (observe that $j\leq k$ is used here).  
 
 Hence by induction we arrive at a situation where  the first coordinate of 
$  f_\la(x,y)$  is of  the form  $\rho_\la x + a_k(\la, y)x^k + O(x^{k+1})$, with   
$a_k(0,y)=1$, and the desired form follows by putting $(X,Y) = (a_k(\la, y)^{\frac{1}{k-1}} x,y)$. 
\end{proof}

\begin{rmk}
 Observe that the normal form \eqref{eq:ueda hakim} is more precise than
 the one that we obtain here for $f_0$. 
 Indeed, as opposed to the case $
\la=0$,   we cannot in general kill the terms $x^{k+2}, \cdots , x^{2k}$ in the first coordinate of 
\eqref{eq:flambdaq}. 

In fact,  the vanishing of these terms for $\la=0$ is incompatible with keeping $(f_\la)$ 
in form \eqref{eq:flambdaq}. Indeed, 
the change of variables required to kill these terms at $\la=0$ is of the  form
 $(x,y)\mapsto (x+\alpha_\la (y) x^j, y+h.o.t.)$, where $\alpha_0(0)\not=0$ and  $j\leq k$ 
(compare \cite[Thm 6.5.7]{beardon iteration}).  If $\rho_\la \neq 1$ for some $\la$, 
it brings back a non-zero term $x^j$ in the first coordinate of $f_\la$.

On the other hand, 
if by chance the terms $x^j$, $j=k+2 \leq k \leq  2k$, vanish, we can reduce  ourselves to \cite{bsu}  by letting 
$(x',y') = (k\rho_\la^{k-1} x^k,y)$ and then $(x'', y'') =(x' + (1-\rho^k)/2, y')$. 
It is the presence of these extra non-vanishing terms that prevents us from 
using  \cite{bsu} directly.
\end{rmk}

\subsection{Transit mappings in the one-dimensional case: un tour de valse}

To fix the ideas, let us establish the   statement we need  in the
one-dimensional case first. 
This is a  
refined version  of \cite{tv}.
Consider a holomorphic family $(f_\la)_{\la\in \La}$ of 
mappings defined in some neighborhood of the origin in $\cc$, of the form 
\begin{equation}\label{eq:flambda dim1}
f_\la(x)  = \rho_\la x + x^{k+1} + x^{k+2}g_\la(x),
\end{equation} with $g$ holomorphic.  As before $\La$ is the unit disk. We assume that $\rho_0=1$ and 
$\frac{\fr\rho}{\fr\la}(0)\neq 0$ 
(this amounts to replacing $f_\la$ by its    $q^{\rm th}$ iterate in \eqref{eq:flambdaq}). 

Recall that for $\la=0$ the repelling and attracting directions are respectively defined by the property 
 that
$(1+x^k)\in \re^{+/-}$. We fix two consecutive such directions with respective angles $0$ and 
$\frac{\pi}{k}$ , and non-overlapping sectors about them 
by putting $$S^\iota = \set{\arg x\in \lrpar{-\frac{5\pi}{4k}, -\frac{3\pi}{4k}} }\text{ and }
S^o = \set{\arg x\in \lrpar{-\frac{\pi}{4k}, \frac{\pi}{4k}}}.$$

The result is as follows.

\begin{thm}\label{thm:tv dim1} 
Let $f_\la$ be as in \eqref{eq:flambda dim1} and $S^{\iota/o}$ be as above. 
There exists a neighborhood $V$ of the origin in $\cc$ with the following property: if $Q^\iota$ and $Q^o$ are open topological disks  with $Q^\iota\Subset S^\iota\cap V$ and $Q^o\Subset S^o\cap V$, then for every  neighborhood $W$ of $0$ in $\La$, there exists an integer $N$ and a radius 
$r$  such that if $n\geq N$ there exists a holomorphic map $\la_n:Q^\iota\times Q^o \cv W$ such that for every $(z^\iota, z^o)\in Q^\iota\times Q^o$, $f^n_{\la_n(z^\iota, z^o)}$ is a well defined univalent  function on $B(z^\iota, r)$, with $f^n_{\la_n}(z^\iota)=z^o$ and $ \abs{(f^n_{\la_n})'-1}
\leq \frac15$.
\end{thm}

This statement being quite technical, a few words of explanation are in order.  
 What this theorem says is that if a parabolic bifurcation of 
the form \eqref{eq:flambda dim1} occurs, then by carefully selecting the parameters $\la_n$,   taking high iterates $f^n_{\la_n}$
we can map any point $z^\iota$ from an attracting sector to any point $z^o$ in a consecutive  repelling sector, with uniform control on the 
derivative $(f^n_{\la_n})'$ in the neighborhood of $z^\iota$. 

\medskip

To prove the theorem we begin with some background and intermediate results.
We work in the new coordinate  $z= \frac{\rho_\la ^{k+1}}{kx^k}$. Notice that for $\la=0$,
 the change of variables maps 
the sector $\set{\arg x \in \lrpar{-\frac{3\pi}{2k}, \frac{\pi}{2k}}}$ onto $\cc\setminus i\re^-$, hence for small enough $\la$, $S^\iota$ and $S^o$ are contained in 
$\lrpar{\frac{\rho_\la ^{k+1}}{kx^k}}^{-1} (\cc\setminus i\re^-)$. 
In the new coordinates, $S^\iota$ and $S^o$ are respectively perturbations of the sectors 
$\set{\arg z\in \lrpar{\frac{3\pi}{4}, \frac{5\pi}{4}}}$ and 
$\set{\arg z\in \lrpar{-\frac{\pi}{4}, \frac{\pi}{4}}}$.

Using the  the classical notation
 $x_1 = f_\la(x)$  (and similarly for $z$), we infer that 
 \begin{align*}
 z_1 &= \frac{\rho_\la ^{k+1}}{kx_1^k} = \frac{\rho_\la ^{k+1}}{k(f_\la(x))^k} 
= \frac{\rho_\la ^{k+1}}{k \rho_\la^k x^k \lrpar{1+ \frac{x^k}{\rho_\la} + O(x^{k+1})}^k} \\
&= \frac{\rho_\la ^{k+1}}{k \rho_\la^k x^k}\lrpar{1- \frac{ k x^k}{\rho_\la} + O(x^{k+1})}
= \frac{\rho_\la}{kx^k} -1 + O (x )\\
&= \rho_\la^{-k} z - 1 + \eta_\la(z), \text{ with } \eta_\la (z)  = O\lrpar{\unsur{\abs{z}^{1/k}}} \text{ as } z\cv\infty \text{, uniformly in }\la\in \La.
\end{align*}

The exponent $1/k$ will play a special role in the estimates to come, so for notational ease, 
from now on we put $\gamma = 1/k$. 
We also change coordinates in the parameter space by putting $u =\rho_\la^{-k}-1$, 
so that $u$ now ranges in some neighborhood $W$ of the origin, 
 and our mapping   writes  as 
$$f_u(z) = (1+u)z-1+ \eta_u( z), \text{ with }  \eta_u( z)  = O\lrpar{\unsur{\abs{z}^{\gamma}}}.$$ 

In these coordinates, $f_u$ is defined in an open set $\om_R$ of the form 
\begin{equation}\label{eq:omegaR}
\om_R = \set{z, \ \abs{z}>R, \arg(z) \in \lrpar{\frac{-3\pi}{8},  \frac{11\pi}{8}}},
\end{equation} for some $R=R_0$. 
Its complement is shaded on Figure \ref{fig:tourdevalse} and will be
referred to as the  ``forbidden region". 
 We also pick two bounded open topological disks $Q^\iota$ and $Q^o$ such that    
$Q^\iota\Subset S^\iota\cap \Omega_{R}$ and $Q^o\Subset S^o\cap\Omega_{R}$, where $R\geq R_0$ is  to be fixed later (this corresponds to the choice of the neighborhood $V$ in the statement of the theorem).

We fix a constant $M$ such that for every parameter $u\in W$  and  every $z\in \Omega_{R_0}$, 
\begin{equation}\label{eq:eta u}
\abs{\eta_u(z)}\leq 
 \frac{M}{\abs{z}^{1/k}} =  
 \frac{M}{\abs{z}^{\gamma}} \text{ and } 
\abs{\eta'_u(z)}\leq \frac{M}{\abs{z}^{1+\gamma}}.
\end{equation} 

We will let $u$ vary in a small subset of $W$, of the form 
$W_n = B\lrpar{-\frac{2\pi i}{n}, \frac{1}{n^{1+ \gamma/2}}}$. 
Notice that for $u\in W_n$ we have that $$\abs{1+u} = 1+ \frac{2\pi^2}{n^2} + o\lrpar{\unsur{n^2}}
\text{ and } \arg(1+u) = -\frac{2\pi}{n} + O\lrpar{\unsur{n^{1+ \gamma/2}}} .$$ 
In the following we always consider  $n$ so large that $n^{ \gamma/2}>100$, $1-\frac{30}{n^2}\leq 
\abs{1+u} \leq 1+\frac{30}{n^2}$ and $\abs{\arg(1+u)+\frac{2\pi}{n} }\leq \unsur{100n}$.

\medskip

To understand the argument better, it is instructive to think of $f_u$ as a perturbation of the affine map $
\ell_u:z\mapsto (1+u)z - 1$.   
When $u\in W_n$ and $n$ is large, $\ell_u$ is approximately a rotation by 
angle $-\frac{2\pi}{n}$  centered at $\frac{1}{u}$. Notice also  that $\frac{1}{u}$ is close to 
$\frac{in}{2\pi}$ (see Figure  \ref{fig:tourdevalse}). 
 
To fix the ideas,  let us first analyze the linear case,
dealing with $\ell_u$ instead of $f_u$.

\begin{prop}\label{prop:linear}
With notation as above, there exists an integer $N$, and a radius $r$  
 such that if $n\geq N$
then  for every $(z^\iota, z^o)\in Q^\iota\times Q^o$, there exists a parameter $ u=u(z^\iota, z^o)\in W_n$, depending holomorphically on $(z^\iota, z^o)$ and such that
\begin{itemize}
\itm $\ell_u^n(z^i) =z^o$;
\itm for every $z\in B(z^\iota, r)$ the iterates $\ell_u^j(z)$, $j = 1,\ldots, n$ do not enter the forbidden region;
\itm $  \abs{(\ell_u^n)'-1}\leq  \frac15$ on $B(z^\iota, r)$. 
\end{itemize} 
\end{prop}

\begin{figure} [h]
\centering 
\def\svgwidth{7cm}
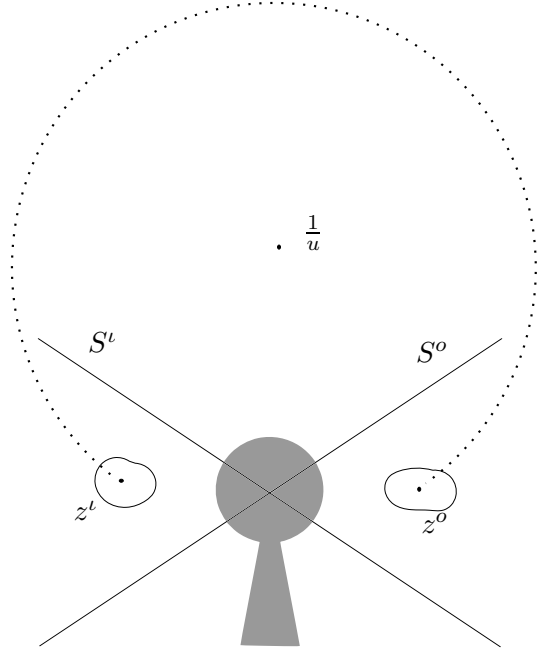
 \caption{Schematic view of the 
orbit connecting $z^\iota$ to $z^o$.
It  shadows the arc of a
circle passing through the ``gate'' 
between the fixed points $1/u$ and $\infty$.
The forbidden region is shaded.}
 \label{fig:tourdevalse}
\end{figure}

\begin{proof}
Let $l = \lfloor n/2 \rfloor$ and $m=n-l$. 
 As said above, $\ell_u(z)  = (1+u) \lrpar{z-\unsur{u}} + 
\unsur{u}$  has its fixpoint at $\unsur{u}$. Write $u = \frac{-2\pi i}{n} + \frac{v}{n^{1+ \gamma/2}}$, 
with $v\in \dd$. For $j\leq n$ we have that 
\begin{align}
(1+u)^j \notag &= \exp(j\log(1+u))  = \exp \lrpar{j\log\lrpar{1 - \frac{2\pi i}{n} + \frac{v}{n^{1+ \gamma/2}}}}\\
& = \exp \lrpar{j \lrpar{  - \frac{2\pi i}{n} + \frac{v}{n^{1+ \gamma/2}}  + O\lrpar{\unsur{n^2} }}} \notag\\
&= \exp\lrpar{    - \frac{2j\pi i}{n}  +  \frac{jv}{n^{ 1+  \gamma/2}}    + O\lrpar{\frac{j}{n^2}}},\notag
  \end{align}
in particular for $j=n$
\begin{equation}\label{eq:DL}
(1+u)^n = 1+   \frac{v}{n^{  \gamma/2}}    + O\lrpar{\unsur{n}},
\end{equation} where the $O(\cdot)$ is uniform with respect to  
$v\in \dd$.

Simple geometric considerations (see \cite{tv}) 
then show that for 
$j\leq \lceil n/2\rceil$ $\ell_u^j( z^\iota)$ (resp. $\ell^{-j}_u(z^o)$) do not enter the forbidden area. 

Let us prove that there exists $u(z^\iota, z^o)$, depending holomorphically on 
$(z^\iota, z^o)\in Q^\iota\times Q^o$ and such that $\ell^{l}_u(z^\iota) = \ell_u^{-m}(z^o)$. Then for such a parameter, 
by connecting the two pieces of orbits $1, \ldots , l$ and $l+1, \ldots , n$, we infer
  that the iterates $\ell^j_u(z^\iota)$ do not enter the forbidden area for $1\leq j\leq n$ and since $\ell_u$ is affine, 
the control of the derivative follows from \eqref{eq:DL}. 

To prove this,    consider the expression 
$$\frac{\ell_u^{l} (z^\iota)- \unsur{u}}{\ell_u^{-m} (z^o)- \unsur{u}} =
(1+u)^n\frac{z^\iota- \unsur{u}}{z^o- \unsur{u}}.$$ A simple computation shows that 
$$\frac{z^\iota- \unsur{u}}{z^o- \unsur{u}} = 1 + (z^\iota-z^o)\frac{2\pi i}{n} + 
O\lrpar{\unsur{n^{1+ \gamma/2}}}.$$ Therefore by \eqref{eq:DL}, we infer that 
\begin{equation}\label{eq:lun}
\frac{\ell_u^{l} (z^\iota)- \unsur{u}}{\ell_u^{-m} (z^o)- \unsur{u}} =
1+   \frac{v}{n^{  \gamma/2}}    + O\lrpar{\unsur{n}},
\end{equation} 
where  the $O(\cdot)$  is uniform with respect to $v\in \dd$, $z^\iota\in Q^\iota$ and $z^o\in Q^o$.
Thus  when $n$ is large enough the quantity in \eqref{eq:lun} 
winds once around 1 as $v$ turns once around $\fr \dd$, and the result follows from the Argument Principle.
\end{proof}

We now turn to $f_u$. 
Let us start with a technical lemma. 

\begin{lem}\label{lem:fu} Fix $R\geq (10^5 kM)^k$.
With notation as above, there exists an integer $N=N(R)$ 
depending only on $R$ such that if $n\geq N$, $u\in W_n$ and 
if $z^\iota\in Q^\iota$ 
then for every  $1\leq j\leq \lceil \frac{n}{2}\rceil$:   
\begin{enumerate}[{\rm (i)}]
\item $f_u^j(z^\iota)$ stays outside the forbidden area; 
\item $\abs{f_u^j(z^\iota) -\unsur{u}} \geq \frac{n}{10}$;
\item writing $f^j_u(z^\iota)=z_j = x_j+iy_j$ we have that either $x_j\leq x_0 - \frac{j}{10}$ 
or $y_j\geq \frac{n}{10}$. In particular $\abs{z_j}\geq \min\lrpar{\frac{\abs{z_0}}{2}+\frac{j}{10},\frac{n}{10}}$.
\end{enumerate}
The same results holds 
for $f^{-j}_u(z^o)$, when $z^o\in Q^o$ (in that case the last condition needs to be replaced by 
``either $x_j\geq x_0 + \frac{j}{10}$  or $y_j\geq \frac{n}{10}$")
 \end{lem}
%
%

\begin{proof}
We first deal with the  assertions  (i) and  (ii).
We argue by induction so assume the result holds for $j\leq k-1$, for some 
$k\leq \lceil \frac{n}{2}\rceil$. Let us write 
$$\frac{f_u^j(z^\iota)-\frac1u}{f_u^{j-1}(z^\iota)-\frac1u} = (1+u) +
 \frac{ \eta_u(f_u^{j-1}(z^\iota))}{f_u^{j-1}(z^\iota)-\frac1u},$$
so that
\begin{equation}
     \label{eq:prod fu}
     \frac{f_u^{k}(z^\iota)-\frac1u}{ z^\iota-\frac1u} = (1+u)^{k} \prod_{j=0}^{k} \lrpar{1+  
    \frac{ \eta_u(f_u^{j}(z^\iota))}{(1+u)(f_u^{j}(z^\iota)-\frac1u)} }.
\end{equation} 
Considering the modulus of this expression, we see that 
\begin{align*}
    \abs{f_u^{k}(z^\iota)-\frac1u} &\geq \abs{ z^\iota-\frac1u}
\lrpar{1- \frac{30}{n^2}}^{\lceil \frac{n}{2}\rceil}  \prod_{j=0}^{k-1}
\lrpar{1- \frac{M}{(0.9)R^{\gamma} \abs{f_u^{j}(z^\iota)-\frac1u} }} \\
&\geq \abs{ z^\iota-\frac1u}
\lrpar{1- \frac{30}{n^2}}^{\lceil \frac{n}{2}\rceil} 
     \lrpar{1- \frac{10M}{(0.9)R^{\gamma} n }}^{\lceil \frac{n}{2}\rceil},
\end{align*}
where 
the first estimate follows from bound \eqref{eq:eta u} on $\eta_u$ and 
the second estimate follows from the induction hypothesis.  
Since $\lrpar{1- \frac{30}{n^2}}^{n/2}\cv 1$ as $n\cv\infty$, 
by our choice of $R$ we see that when $n\geq N(R)$, 
$$\abs{f_u^{k}(z^\iota)-\frac1u} \geq \frac{9}{10} \abs{ z^\iota-\frac1u}\geq \frac{9}{10} 
d(\frac1u, S^\iota)\geq  \frac{9}{10}  \frac{n}{2\sqrt2\pi}\geq
\frac{n}{10}, $$ 
which proves (ii).

To prove that $f^k_u(z^\iota)$ does not enter the forbidden region, we
look at the argument of $f^k_u(z^\iota)-\frac1u$. Recall that
$\abs{\arg(1+u)+\frac{2\pi}{n} }\leq \unsur{100n}$ 
so by (\ref{eq:prod fu}) 
$$
      \abs{\arg\lrpar{\frac{f_u^{k}(z^\iota)-\frac1u}{ z^\iota-\frac1u}} - \lrpar{\frac{-2k\pi}{n}} }
    \leq \frac{k}{100n} + \sum_{j=0}^{k-1} \abs{\arg\lrpar{1+  
   \frac{ \eta_u(f_u^{j}(z^\iota))}{(1+u)(f_u^{j}(z^\iota)-\frac1u)} }}.
$$ 
     With our choice of $R$, 
$\abs{\frac{ \eta_u(f_u^{j}(z^\iota))}{(1+u)(f_u^{j}(z^\iota)-\frac1u)} }\leq \frac{1}{200n}$, so since 
$\log(1+z) = z + h.o.t.$, when $n$ is large enough, we infer that 
$$ 
      \abs{\arg\lrpar{1+  
     \frac{ \eta_u(f_u^{j}(z^\iota))}{(1+u)(f_u^{j}(z^\iota)-\frac1u)}   }} \leq \frac{1}{100n}.
$$ 
Thus we obtain that 
$$
  \abs{\arg\lrpar{\frac{f_u^{k}(z^\iota)-\frac1u}{ z^\iota-\frac1u}} - \lrpar{\frac{-2k\pi}{n}} } \leq 
  \frac{k}{100n} +  \frac{k}{100n} \leq  \frac{k}{50n},
$$  
therefore arguing geometrically we see that 
 $f_u^{k}(z^\iota)$ stays outside the forbidden region. The induction
 step is complete proving  (i).
 
 \medskip
 
 To establish (iii), let us first observe that due to the the above estimate on the argument, 
when $j\leq \lceil \frac{n}2 \rceil$,   
$\displaystyle{\arg\lrpar{\frac{f_u^{j}(z^\iota)-\frac1u}{
      z^\iota-\frac1u}}}$  
is equal to ${\frac{-2j\pi}{n}} $, up to an error of at most $\unsur{50}$. 
Expressing in coordinates, we see that
$$x_{j+1} = x_j -1 + \frac{2\pi}{n} y_j + \e_j \text{ and } y_{j+1} = y_j -  \frac{2\pi}{n} x_j + \e'_j,$$
 with 
 $$\abs{\e_j} , \abs{\e'_j}\leq \max \lrpar{\unsur{n^{1+ \gamma/2}}, \frac{M}{\abs{z_j}^{\gamma}}}\leq  
\unsur{1000}$$ 
because $n^ {\gamma/2}\geq 100$ and by the previous step, $z_j\in \om_R$.  We see that, as soon as 
$y_j\leq \frac{n}{10}$, we have that $x_{j+1} \leq x_j - \frac{1}{10}$. Now when  $y_j$ reaches 
$\frac{n}{10}$, and until $j$ is as large as $ \frac{n}{4}$ (a time at which 
 $y_j$ is approximately equal to $\frac{n}{2\pi}$), 
since $\abs{z_j - \frac1u}\geq \frac{9}{10} \abs{z_0 -\frac1u}$, by expressing the distance in coordinates 
 and using the estimate on the 
 argument, we infer that 
$x_j\leq -\frac{n}{100}$ 
therefore $y_{j+1}\geq y_j$. The result follows. 

The argument  for $f^{-j}_u(z^o)$, 
$1\leq j\leq \lceil \frac{n}{2}\rceil$ is similar, and is left to the reader.
 \end{proof}

%
%
%
\begin{proof}[Proof of Theorem \ref{thm:tv dim1}]
We argue as in 
Proposition \ref{prop:linear}. As before let $l = \lfloor n/2 \rfloor$ and $m=n-l$.
Using  \eqref{eq:prod fu} with $k=l$, we obtain
 $$
     \frac{f_u^{l}(z^\iota)-\frac1u}{ z^\iota-\frac1u} = (1+u)^{l} \prod_{j=1}^{l} \lrpar{1+  
    \frac{
      \eta_u(f_u^{j-1}(z^\iota))}{(1+u)(f_u^{j-1}(z^\iota)-\frac1u)}}.
$$
  Hence, using  Lemma \ref{lem:fu} 
together with the  inequality 
$\abs{\prod(1+x_j)-1}\leq \exp {\sum \abs{x_j}}-1$,   
we infer that 
$$
  \abs{\frac{f_u^{l}(z^\iota)-\frac1u}{ (1+u)^{l}\lrpar{z^\iota-\frac1u}} -1}
   = \abs{\frac{f_u^{l}(z^\iota)-\frac1u}{\ell_u^{l}(z^\iota)-\frac1u} -1} 
    \leq \exp \lrpar{\sum_{j=0}^{l} \frac{10 M}{(0.9) n \min(\frac{\abs{z^\iota}}{2}+ \frac{j}{10} , \frac{n}{10}) ^{\gamma}}} -1   \\
$$
\begin{equation}\label{X}
\leq \exp \lrpar{\sum_{j=0}^{\lceil\frac{n}{2}\rceil} \frac{100 M}{9 n  \min(\frac{j}{10}+1, \frac{n}{10})^{\gamma}}} -1
\leq \exp\lrpar{\frac{1000M}{n^{\gamma}}}-1 = 
 O\lrpar{\unsur{n^{\gamma}}},
 \end{equation}
 where in the last inequality we use an elementary estimate
 $$\sum_{j=0}^{\lceil\frac{n}{2}\rceil} \frac{1}{ \min(\frac{j}{10}+1, \frac{n}{10})^{\gamma}} \leq 50 n^{1-\gamma}.$$
Doing the same with $f^{-m}_u(z^o)$ we get  that 
$$ \frac{f_u^{l}(z^\iota)-\frac1u}{f_u^{-m}(z^o)-\frac1u }  
\lrpar{\frac{\ell_u^{l}(z^\iota)-\frac1u}{\ell_u^{-m}(z^o)-\frac1u }}^{-1} = 1 + O\lrpar{\unsur{n^{\gamma}}}.$$
Thus, from \eqref{eq:lun} we deduce that
$$\frac{f_u^{l}(z^\iota)-\frac1u}{f_u^{-m}(z^o)-\frac1u }   = 1+   \frac{v}{n^{  \gamma/2}}    + 
O\lrpar{\unsur{n^{\gamma}}}, $$ where the $O(\cdot)$ is uniform with
respect to  $(v,z^\iota, z^o)\in \dd\times Q^\iota\times Q^o$. 
Therefore  we conclude that if $n$ is large enough,  when $u$ winds
once around $\fr W_n$ 
(i.e. $v$ winds once around $\fr\dd$),  the curve  
$u\mapsto \displaystyle{\frac{f_u^{l}(z^\iota)- 1/u}{f_u^{-m}(z^o)- 1/u }}$ winds once around 1, so by the Argument Principle, there exists a unique 
 $u = u(z^\iota, z^o)\in W_n$ (thus, depending holomorphically on $(z^\iota, z^o)$), such that 
 $ f_u^{l}(z^\iota) = f_u^{-m}(z^o) $. Given such a $u$, we see that the iterates 
  $ f_u^{j}(z^\iota) $ $1\leq j\leq n$ stay outside the forbidden region, and 
  $f_u^{n }(z^\iota) = z^o$. 

\medskip

Let us now estimate the derivative   $(f_u^{l})'(z)$, for $z\in Q^\iota$ (this is place where we  
 need to be precise on the value of $R$).  
 Recall that for $1\leq j\leq l$, 
$f_u^{j}(z)$ is well-defined, and write  
$$(f_u^{l})'(z)  = \prod_{j=1}^{l} (f_u)'(f_u^{j-1}(z)) \text{, where } (f_u)'(z)  = 1+u + \eta_u'(z), \ 
\abs{\eta_u'(z)} \leq \frac{M}{\abs{z}^{1+\unsur{k}}}.$$ So we get that 
$$(f_u^{l})'(z)  = (1+u)^{l} \prod_{j=1}^{l}\lrpar{1 + \frac{\eta_u'(f_u^{j-1}(z))}{1+u}} .$$
 Our choice of $u$ and $l$ implies that $0.99 \leq \abs{ (1+u)^{l}}\leq 1.01$ for large $n$, while 
 \begin{equation}\label{eq:exp}
 \abs{\prod_{j=1}^{l}\lrpar{1 + \frac{\eta_u'(f_u^{j-1}(z))}{1+u}}-1} \leq 
 \exp\lrpar{ \sum_{j=1}^{\lceil\frac{n}2\rceil} \frac{10M}{9(\min(\frac{R}{2}+\frac{j}{10},\frac{n}{10}))^{1+ \gamma} }}-1 .
 \end{equation}
For large $n$ (depending only on $R$) we have that 
 \begin{equation}\label{eq:exp2}
  \sum_{j=1}^{\lceil\frac{n}2\rceil} \unsur{(\min(\frac{R}{2}+\frac{j}{10},\frac{n}{10}))^{1+ \gamma} } \leq 
 \sum_{j=\lfloor \frac{R}{2} \rfloor}^{\infty} \frac{10^{1+\gamma}}{j^{1+\gamma}} + 
 \sum_{j=\lfloor \frac{n}{10}\rfloor}^{\lceil\frac{n}2\rceil}\lrpar{ \frac{n}{10}}^{-(1+\gamma)}  \leq 200k R^{-\gamma}.
 \end{equation}
 Since $R\geq (10^5kM)^{k}$, by \eqref{eq:exp2} 
 we infer that the right hand side of \eqref{eq:exp} 
 is smaller than $\frac{1}{400}$, and finally we conclude that 
  when  $z\in Q^\iota$ and $n$ is large enough, 
 $  \abs{(f_u^{l})'(z)-1 }\leq  \frac{1}{50}$.
 
 The following lemma is classical, for convenience we recall the proof below. 
 
 \begin{lem}\label{lem:classical}
 Let $f$ be a holomorphic function on 
 $\dd_r$ 
 such that $\abs{f'-1}\leq a<1$ on  $\dd_r$. 
 Then $f$ is injective on 
 $\dd_r$ and 
  $$ D(f(0), (1-a)r)\subset f(\dd_r)\subset D(f(0), (1+a)r).$$
 \end{lem}

 From this lemma we deduce that  there exists $r>0$ independent on $n$
  such that  $f_u^{l}$ is univalent on $B(z^\iota, r)$, and its image contains $B(f^{l}(z^\iota), r)$. 
 Likewise, there exists $r>0$ such that $f_u^{-m}$ is univalent on $B(z^o, r)$, with  derivative  $\frac{1}{50}$-close to 1.  
 Thus we  conclude that $f_u^{l}$ maps  
 univalently $B(z^\iota, \frac{r}{2})$ into $  B(z^o, r)$, and its derivative satisfies 
 $ \abs{(f_u^{n})'(z)-1 }\leq  \frac{1}{5}$ This completes the proof of the theorem. 
\end{proof}
 
 \begin{proof}[Proof of Lemma \ref{lem:classical}]
 Replacing $f$ by $z\mapsto r^{-1}f(rz) -f(0)$ it is no loss of generality to assume that $f(0) = 0$ and $r=1$. 
Since $z\mapsto z-f(z)$ is contracting, if $f(z) =f(z')$ we get that 
$$\abs{(z-f(z)) - (z'-f(z'))} =\abs{z-z'}\leq a \abs{z-z'},$$  whence $z=z'$. Thus $f$ is injective. That $f(\dd) \subset \dd_{1+a}$ readily follows from  the mean value inequality. Finally, to prove that for any $w\in \dd_{1-a}$, the equation $f(z) = w$ admits a solution, it is enough to apply the Contraction Mapping Principle to $g: z\mapsto z-f(z) + w$ in $\dd$.
\end{proof}

 \subsection{Transit mappings in dimension 2}  
 We return to the two-dimensional setting. 
The treatment will be based on the observation that in a      
two-dimensional thickening of the domain $\Omega_R$, the  maps $f_\la$ 
admit a {\it dominated splitting}, i.e., they have a horizontal
cone field invariant under the forward dynamics, and moreover, they
are contracting in the vertical direction.

 Let us first fix some notation.
As before the parameter space  $\La$ is the unit disk. 
 Changing coordinates and passing to an iterate if needed, by Proposition \ref{prop:good coord} 
 we may assume that 
 $(f_\la)_{\la\in \La}$ is  a holomorphic family of germs of diffeomorphisms in $(\cd, 0)$ of the form 
\begin{equation}\label{eq:flambdaqbis}
f_\la(x,y ) =   (\rho_\la x+ x^{k+1} + x^{k+2}g_\la(x,y), b_\la y + xh_\la(x,y)),
\end{equation} where $\rho_0=1$, $\frac{d\rho}{d\la}(0)\neq 0$, and $\abs{b_\la}\leq b<1$ for all $\la$. 

As in the one-dimensional case, we consider two consecutive sectors 
 $S^\iota = \set{\arg x\in \lrpar{-\frac{5\pi}{4k}, -\frac{3\pi}{4k}} }$ and 
$S^o = \set{\arg x\in \lrpar{-\frac{\pi}{4k}, \frac{\pi}{4k}}}$. For $\la=0$, 
consider  a bidisk  $V=V_1\times V_2$ around  the origin  such that: 
\begin{itemize}
\itm $(S^\iota\cap V_1) \times V_2$ is attracted to the origin under forward iteration;
\itm there exists a local repelling petal $\Sigma\subset V_1\times V_2$, which is a graph over $S^o\cap V_1$,
   defined by the property that every orbit converging to 0 under backward iteration 
in $(S^o\cap V_1) \times V_2$ belongs to $\Sigma$. 
\end{itemize}

The next theorem is the key technical mechanism which will allow us to create a homoclinic tangency from a 
critical point in a semi-parabolic basin. 
It is the  counterpart of Theorem \ref{thm:tv dim1} in the dissipative 2-dimensional setting. 
Assuming that  a semi-parabolic bifurcation of the form 
\eqref{eq:flambdaqbis} occurs, 
it  select parameters $\la_n$ such that the iterates $f^n_{\la_n}$ map a given point $p^\iota$ in some  semi-parabolic 
 basin (almost) onto a given target $p^o$ located in a repelling petal $\Sigma$,   with a good control on the geometry of 
 $f^n_{\la_n}$ near $z^\iota$. This geometric control   is expressed in terms of the pull-back action on a foliation  transverse to 
 $\Sigma$ near $p^o$.  
 
 \begin{thm}\label{thm:tv}
Let $(f_\la)_{\la\in \La}$ be as above.
There exists a bidisk $V=V_1\times V_2$ around  $0\in \cd$  such that  if $Q^\iota\Subset (S^\iota\cap V_1)\times V_2$, 
 $Q^o\Subset \Sigma$, and $\mathcal{F}$ is a germ of holomorphic foliation   
  transverse to $\Sigma$ along $Q^o$, 
   then for every  neighborhood $W$ of $0$ in $\La$, there exists an integer $N$ and a radius 
$r$  such that if $n\geq N$, there exists a holomorphic map $\la_n:Q^\iota\times Q^o\cv W$ such that 
for every $(p^\iota, p^o)\in Q^\iota\times Q^o$, for 
$\la_n= \la_n(p^\iota, p^o)$, 
there exists a  bidisk $D^2(p^\iota, r)$ around $p^\iota$, and a neighborhood $D_\Sigma(p^o,r) = B(p^o, r)\cap \Sigma$ 
of $p^o$ in $\Sigma$,   such that the following properties hold: 
 \begin{itemize}
\itm $f^n_{\la_n} (p^\iota)$ belongs to $\mathcal{F}(p^o)$, the leaf of $\mathcal{F}$ through $p^o$;
 \itm the preimage of $\mathcal{F}$ under $f^n_{\la_n}$ 
 defines a holomorphic foliation  $\mathcal{F}^{-n}$ of $D^2(p^\iota, r)$ by vertical  graphs
along which  $f^n_{\la_n}$ contracts by a factor $b^n$;
\itm the  derivative of $f^n_{\la_n}$ along any horizontal line  in  $D^2(p^\iota, r)$ satisfies 
$\displaystyle \abs{\frac{\fr  f^n_{\la_n}}{\fr z} -1}\leq \frac15.$
\end{itemize}
\end{thm}
 
To prove the theorem, 
we consider the dynamics of $f_\la$ in a domain 
of the form
$$\set{\arg(x)\in \lrpar{\frac{-3\pi}{2k}, \frac{\pi}{2k}}}\times \dd_{s_0}$$ and 
as in the previous section  
we change coordinates by putting 
$(z,w) = (\frac{\rho_\la^{k+1}}{kx^k}, y)$, and $u  = \rho_\la^{-k} -1$. In the new coordinates, $u$ ranges in some small
 neighborhood $W$ of    the origin and $f_u$ is defined in a 
domain of the form $\om_{R_0} \times \dd_{s_0}$, where $\om_R$ is as in \eqref{eq:omegaR}, 
$R_0\geq 1$, $s_0\leq 1$, 
 and its expression   becomes 
\begin{equation}
f_u(z,w)  = ((1+u) z -1 + \eta_u(z,w), b_u w + \theta_u(z,w)), \label{eq:fubis}
\end{equation} where 
$\eta_u(z,w)$ and $\theta_u(z,w)$ are of the form $\unsur{z^{1/k}}\varphi_u(\unsur{z^{1/k}},w)$, 
 with $\varphi_u$ holomorphic in the neighborhood of the origin. In the new coordinates, 
$$S^\iota = \set{z,\ \arg(z)\in \lrpar{\frac{3\pi}{4}, \frac{5\pi}{4}}}\text{ and }
S^o = \set{z, \ \arg(z)\in \lrpar{-\frac{\pi}{4}, \frac{\pi}{4}}}.$$
As above  we let 
$W= W_n = D\lrpar{-\frac{2\pi i}{n}, \frac{1}{n^{1+ \gamma/2}}}$ (recall that $\gamma=1/k$).  

We will gradually adjust the parameters $R$ and $s$.  
We fix $M$ such that for $(z,w)\in \om_{R_0}\times \dd_s$ and $u\in W$, 
$$\abs{\eta_u(z,w)},\ 
\abs{ \frac{\fr\eta_u}{\fr w}(z,w)},  \ \abs{\theta_{u}(z,w)} , \ \abs{ \frac{\fr\theta_u}{\fr w}(z,w)} \leq \frac{M}{\abs{z}^{\gamma}},
\text{ and }\abs{ \frac{\fr\eta_u}{\fr z}(z,w)} \leq \frac{M}{\abs{z}^{1+\gamma}} 
.$$
 
 Due to dissipation, there is now an asymmetry between positive and negative iterates. 
 The idea of the construction of the transition mapping  is now
to  pull back   $n/2$ times a leaf of the foliation $\mathcal{F}$ from the ``outgoing'' region $Q^o$ 
  and to push forward $n/2$ times  a point from the ``incoming" region $Q^i$, and use the Argument Principle to make the image of the point belong to the preimage of the leaf. 
 
 \medskip

 We will first prove Theorem \ref{thm:tv} 
under a seemingly stronger assumption that the foliation
$\mathcal{F}$ is composed of graphs over the second coordinate
 in $\om_R\times \dd_s$,   
with slope bounded by $1/100$. 
We start by showing that 
the  backward graph transform is well defined for  
such vertical graphs  on an appropriate subregion of 
 $\om_{R_0}\times D_{s_0}$ 
(as long as  $R_0$ is large and $s_0$ is small).
This is a  standard technique for maps with dominated splitting, which is 
e.g. 
used to construct the strong stable  foliation on forward
invariant regions  
(this  is not the case
we are dealing with here). 

%
%

 \begin{lem}\label{lem:graph}
 Let $\widetilde \om_R = \set {\zeta \in \om_R :\, 
 D(\frac{1+\zeta   }{1+ u}, 1)\subset  \om_R}$. 
 There exists $R_0$ and $s_0$ such that if $R\geq R_0$ and $s\leq s_0$, then if $\Gamma$ is a vertical graph of slope $\leq 1/100$ in $\widetilde{\om}_R \times \dd_s$ then
 $f_u^{-1}(\Gamma) \cap({\om_R} \times \dd_s)$ is a vertical graph
 in ${\om_R} \times \dd_s$
  of slope $\leq 1/100$.
 \end{lem}
 
 \begin{proof} Take $\Gamma = \set{z=\psi(w)}$ with 
 $\psi (0)\in  \widetilde\om_R$ and $\abs{\psi'}\leq 1/100$. Then $f^{-1}(\Gamma)$ admits an equation of the form $\Psi(z,w) =0$, where 
 $$\Psi(z,w) = (1+u)z-1 + \eta_u(z,w)  - \psi(b_u w + \theta_u(z,w)).$$ For $w=0$, Rouché's theorem implies that for $R\geq R_0$, there exists $z$ such that 
 \begin{equation}\label{eq:psi}
 \Psi(z, 0)=0  \text{ and }
 \abs{z- \frac{1+\psi(0)}{1+u}}\leq \frac{2M}{\abs{z}^{\gamma}}.
 \end{equation}
In addition we have
$$
    \frac{\fr\Psi}{ \fr z}  = 1+u+ O(R^{-\gamma})\text{ and }  
 \abs{ \frac{\fr\Psi}{ \fr w}}  \leq \frac{b}{100} + O(R^{-\gamma}).
$$  
Thus, the result follows from the Implicit Function Theorem. 
 \end{proof}
 
From now on the parameter $s=s_0$ will be fixed, and for notational
simplicity we   denote the second factor $\dd_{s_0}$ by $D$. 
Let $Q^\iota\Subset S^\iota\times D$. 
We will now state two different  counterparts of Lemma \ref{lem:fu}: one for push-forwards, and the other one for pullbacks. 
For $p^\iota=(z^\iota, w^\iota)\in Q^\iota$, we denote by $p^\iota_j = (z^\iota_j,w^\iota_j)$ its $j^{\rm th}$ iterate under $f_u$. 

\begin{lem}\label{lem:fu push} With notation as above,
fix $R\geq \max(M^k (1-b)^{-k} s_0^{-k}, (10^5kM)^k)$.
Then there exists an integer $N=N(R)$ 
such that if $n\geq N$, 
$p^\iota\in Q^\iota\times D$ and $u\in W_n$   
 then for every  $1\leq j\leq \lceil \frac{n}{2}\rceil$ we have that 
\begin{enumerate}[{\rm (i)}]
\item $f_u^j(p^\iota)$ belongs to $\om_R\times D$ 
\item $\displaystyle \abs{z^\iota_j-\unsur{u}} \geq \frac{n}{10}$;
\item  $\displaystyle \abs{z^\iota_j}\geq \min\lrpar{\frac{\abs{z_0}}{2}+\frac{j}{10},\frac{n}{10}}$. 
\end{enumerate}
 \end{lem}

\begin{proof} 
It follows from  expression \eqref{eq:fubis} for $f_u$  that if
$\frac{M}{R^{\gamma}}< (1-b) s_0$ and 
$(z,w) \in\om_R\times D$,  then the second coordinate of $f_u(z,w)$ belongs to $D$.
 So we only need to focus on the first coordinate. 
By \eqref{eq:fubis} we have that 
\begin{equation}\label{eq:zj+1}
\frac{z^\iota_{j+1} - \unsur{u}}{z^\iota_j - \unsur{u}} = 1+u+ \frac{\eta_u(f_u^j(p^\iota))}{z^\iota_j - \unsur{u}}, \end{equation} with 
$ \abs{\eta_u(f_u^j(p^\iota))}
\leq \frac{M}{R^{\gamma}}$ as soon as $f_u^j(p^\iota)\in \om_R\times D$. 
Then   the proof is identical to that of Lemma \ref{lem:fu}. 
\end{proof}

We now deal with pullbacks. Given $p^o\in Q^o$, we consider a holomorphic foliation 
$\mathcal{F}$ by  vertical graphs of slope bounded by $1/100$  
 in a neighborhood of $p^o$, and 
  by $\mathcal{F}(p)$  the leaf  through 
$p$. Starting from 
$\mathcal{F}_0 = \mathcal{F}(p)$, by applying  successive graph transforms 
we   inductively define $\mathcal{F}_{-j-1}  = f_u^{-1}(\mathcal{F}_{-j})\cap (\om_R\times D)$. 
 We also let $\zeta_{-j} = \mathcal{F}_{-j} \cap \set{w=0}$. 
 
 \begin{lem}\label{lem:fu pull} Let $R$ be as in Lemma \ref{lem:fu push}.
There exists an integer $N=N(R)$ 
such that if $n\geq N$, if $p\in Q^o$  and $u\in W_n$   
 then for every  $1\leq j\leq \lceil \frac{n}{2}\rceil$ we have  
\begin{enumerate}[{\rm (i)}]
\item $\mathcal{F}_{-j}(p)$ is a well-defined vertical graph in
  $\om_R\times D$,
       with  slope bounded by  $ 1/100$; 
\item $\displaystyle \abs{\zeta_{-j}-\unsur{u}} \geq \frac{n}{10}$;
\item  $\displaystyle \abs{\zeta_{-j}}\geq \min\lrpar{\frac{\abs{z_0}}{2}+\frac{j}{10},\frac{n}{10}}$. 
\end{enumerate}
 \end{lem}

\begin{proof}
From \eqref{eq:psi} we infer that 
$$\abs{\zeta_{-j-1} - \frac{1+\zeta_{-j}}{1+u}}\leq \frac{2M}{\abs{\zeta_{-j}}^{\gamma}},$$ that is, 
\begin{equation}\label{eq:z-j}
\frac{\zeta_{-j-1} - \unsur{u}}{\zeta_{-j} - \unsur{u}} = \unsur{1+u}+ 
\frac{\e_j}{\zeta_{-j} - \unsur{u}}, \text{ with} \abs{\e_j}\leq \frac{2M}{\abs{\zeta_{-j}}^{\gamma}},
\end{equation}
so as before the result follows exactly as in the one-dimensional case (for  {(i)} we also use Lemma \ref{lem:graph}).
\end{proof}

Pick now an $R$ satisfying all the above requirements.
Let as above $p^\iota  = (z^\iota, w^\iota) \in Q^\iota$,    $p^o\in Q^o$, and let 
 $\mathcal{F}^o$ be the leaf of $\mathcal{F}$ through $p^o$.
Let $l = \lfloor n/2 \rfloor$ and $m=n-l$. By Lemma \ref{lem:fu push}, for $1\leq j\leq l$  
$ f_u^j(p^\iota)\in \om_R$. 
Then, using \eqref{eq:zj+1} exactly as in (\ref{X})  
(i.e. by taking the product from 0 to  $l-1$) 
we deduce that  
$$
   \abs{\frac{ z^\iota_l - \unsur{u}}{(1+u)^{l} (z^\iota -\unsur{u})}-1}  
             = O\lrpar{\unsur{n^{\gamma}}}.
$$

%
%

On the pullback side, recall that  $\mathcal{F}_{-j}(p^o)$ denotes 
the $j^{\rm th}$ graph transform of $\mathcal{F}_0$, and  let  $\zeta_{-j}
= \mathcal{F}_{-j}(p^o)\cap \set{w=0}$. 
Using \eqref{eq:z-j} and taking the product from $j=0$  to $j =-m+1$
we obtain: 
$$
   \abs{(1+u)^{m} \frac{\zeta_{-m} - \unsur{u}}{\zeta_0 - \unsur{u}} -1 }=
   O\lrpar{\unsur{n^{\gamma}}}.
$$ 
Thus, writing 
$u = \frac{-2\pi i}{n} + \frac{v}{n^{1+\gamma/2}}$, from the two previous displayed equations   together with 
 \eqref{eq:DL}, we obtain that 
\begin{equation}
\label{eq:pr1}
 \frac{z^\iota_l - \unsur{u}}{\zeta_{-m} - \unsur{u}}  = 
1+\frac{v}{n^{\gamma/2}}+ O\lrpar{\unsur{n^{\gamma}}}.
\end{equation}
Now express the graph $\mathcal{F}_{-m}$ as $z=\psi(w)$, with $\psi(0) = \zeta_{-m}$. Since 
$\abs{\psi(w^\iota_l) - \psi(0)}\leq 1/100$, we infer that 
$$\frac{\psi(w^\iota_l) -\unsur{u}}{\zeta_{-m} - \unsur{u}}= 1+ O\lrpar{\unsur{n}},$$
so from \eqref{eq:pr1} we finally deduce that   
$$ \frac{z^\iota_l -
  \unsur{u}}{\psi(w^\iota_l)  - \unsur{u}}  =
1+\frac{v}{n^{\gamma/2}}+ O\lrpar{\unsur{n^{\gamma}}}.$$ 
By the Argument Principle we conclude that for every $(p^\iota, p^o)\in Q^\iota\times Q^o$, there exists a unique (hence, depending holomorphically on $(p^\iota, p^o)$)
 $u = u (p^\iota, p^o) \in W_n$ such that 
 $\psi(w^\iota_l)=z^\iota_l$, that is, 
 $f_u^{l}(p^\iota)\in \mathcal{F}_{-m}(p^o)$. 
 
 For   this parameter $u$ we can pull back   $\mathcal{F}_{-m}(p^o)$ under $f^{m}_u$, 
 thus obtaining a vertical graph $\mathcal{F}_{-n}(p^o)$
  through $p^\iota$.  It is clear that the 
 derivative $df^n$ contracts exponentially  along this graph, more precisely $\norm{d(f^n\rest{\mathcal{F}_{-n}(p^o)})}\lesssim b^n$.  
Indeed,  the tangent vectors to $\mathcal{F}_{-n}(p^o)$
 remain in a   cone field close to the vertical under iteration, and the second factor gets 
contracted at rate $b$. 

\medskip

From now on the parameter $u$ is fixed. To simplify notation we drop the subscript $u$ and write $f^j_u = (f_1^j, f_2^j)$ 
We will prove at the same time 
that $f_u^n$ is defined in a fixed domain around $p^\iota$ and
estimate its derivatives. Let $r>0$ be such that all iterates $f^j$,
$1\leq j\leq n$ are well defined on $  D(z^\iota, r)\times \set{w^\iota}$, 
and for all $j\leq n$,
$\norm{f^j(p^\iota) - f^j (z,w^\iota)}$ is bounded by, say, 1. 
For the moment $r$ 
 depends on $n$.  
 
The estimate we need is contained in  the following lemma. Denote $f^j(z, w^\iota) $ by $ (z_j,w_j^\iota)$.
  
 \begin{lem}\label{lem:derivative} For $r$ as above,  
let $K = 2(1-b)^{-1} M$.   Then 
 for every $z\in D(z^\iota, r)$, and  every $1\leq j\leq n$
 $$ {\frac{\fr f_1^j}{\fr z}(z,w^\iota)} = (1+u)^j \prod_{i=1}^j(1+\delta_i), \text{ with } \abs{\delta_i}
 \leq \frac{K}{\abs{z_{i-1}}^{1+\gamma}}, \text{ and } \abs{\frac{\fr f_2^j}{\fr z}(z,w^\iota)}\leq \frac{K}{\abs{z_{j-1}}^{1+\gamma}} $$   
 \end{lem}

As a preliminary observation, notice 
 that if $R\geq (10^6k(1-b)^{-1}M)^k$, and $\delta_i$ is as in the statement of the lemma and $n$ is large enough, then for every $1\leq j\leq n$,
\begin{equation}\label{eq:1/5}
\abs{(1+u)^j \prod_{i=1}^j(1+\delta_i)-1}\leq \frac{1}{5}.
\end{equation}
Indeed, this follows from the proof of Theorem \ref{thm:tv dim1} (see \eqref{eq:exp} and \eqref{eq:exp2} there; also   
 if $j\geq l$ we need to split the product at $l$ and 
to estimate separately the two terms).

\begin{proof} 
We argue by induction on $j$. The result holds true  for $j=1$. So assume that it holds for some $j$. We compute
\begin{align*}
\frac{\fr f_1^{j+1}}{\fr z}(z,w^\iota) &= \lrpar{(1+u)   +\frac{\fr \eta_u}{\fr z} (z_j, w_j^\iota)}  \frac{\fr f_1^{j}}{\fr z}(z,w^\iota)  + 
\frac{\fr f_2^{j}}{\fr z}(z,w^\iota) \frac{\fr \eta_u}{\fr w} (z_j, w_j^\iota) \\
&=(1+u)^{j+1} \prod_{i=1}^j(1+\delta_i) \lrpar{1+ \unsur{1+u} \frac{\fr \eta_u}{\fr z} (z_j, w_j^\iota)}   
+  
\frac{\fr f_2^{j}}{\fr z}(z,w^\iota) \frac{\fr \eta_u}{\fr w} (z_j, w_j^\iota) .
\end{align*}
By the induction hypothesis, 
$$\abs{\frac{\fr f_2^{j}}{\fr z}(z,w^\iota) \frac{\fr \eta_u}{\fr w} (z_j, w_j^\iota) } \leq 
\frac{K}{\abs{z_{j-1}}^{1+\gamma}} \frac{M}{\abs{z_{j}}^{\gamma}}. $$ 
Since $\frac{z_{j+1}}{z_j}$ is close to $1+u$ and $(1+u)^{j+1} \prod_{i=1}^j(1+\delta_i)$ is close to 1, we can write 
$$\frac{\fr f_1^{j+1}}{\fr z}(z,w^\iota)  = 
(1+u)^{j+1} \prod_{i=1}^j(1+\delta_i) \lrpar{1+ \unsur{1+u} \frac{\fr \eta_u}{\fr z} (z_j, w_j^\iota) + \delta}, \text{ with } \abs{ \delta }
\leq \frac{2KM}{\abs{z_j}^{1+2\gamma}}.$$ Thus if we  put 
$\delta_{j+1} =   \unsur{1+u} \frac{\fr \eta_u}{\fr z} (z_j, w_j^\iota) + \delta$
we get that 
$$\abs{\delta_{j+1}}\leq \lrpar{\frac{M}{\abs{1+u}}+ \frac{2KM}{R^{\gamma}}} \unsur{\abs{z_j}^{1+\gamma}},$$ which, from the choice of $R$ and $K$ is not greater than $\frac{K}{\abs{z_j}^{1+\gamma}}$.

To get the bound on the derivative of $f_2^{j+1}$, we write
 $$\frac{\fr f_2^{j+1}}{\fr z}(z,w^\iota) = b_u \frac{\fr f_2^{j}}{\fr z}(z,w^\iota) +
  \frac{\fr f_1^{j}}{\fr z}(z,w^\iota) 
\frac{\fr \theta_u}{\fr z} (z_j, w_j^\iota)  + 
  \frac{\fr f_2^{j}}{\fr z}(z,w^\iota) \frac{\fr \theta_u}{\fr w} (z_j, w_j^\iota),$$
and  we get that 
\begin{align*}
\abs{\frac{\fr f_2^{j+1}}{\fr z}(z,w^\iota) }&\leq b \frac{K}{\abs{z_{j-1}}^{1+\gamma}} 
+ \frac65 \frac{M}{\abs{z_{j}}^{1+\gamma}}
+ \frac{K}{\abs{z_{j-1}}^{1+\gamma}}  \frac{M}{\abs{z_{j}}^{\gamma}}\\  
&\leq \frac{K}{\abs{z_{j} }^{1+\gamma}} 
 \lrpar{ 
b \abs{ \frac{z_{j}}{z_{j-1}}} ^{1+\gamma}
+ \frac{6M}{5K} + 
  \frac{M}{R^{\gamma}} \abs{\frac{z_{j} } {{z_{j-1}} }}^{1+\gamma}
}.
\end{align*}
To conclude, we observe that when $n$ is large enough, due to the choice of $R$ and $K$, the expression within parentheses is smaller than 1. The proof of the lemma is complete.   
\end{proof}

We are now in position to conclude the proof of Theorem \ref{thm:tv}. 
Let $r_0$ be the supremum of the radii $r>0$ such that $f^j$ is 
well-defined, and $f^j(z,w^\iota)$ stays at distance at most 1 from $f^j(z^\iota,w^\iota)$ for $1\leq j\leq n$. 
By the above lemma and \eqref{eq:1/5},  
$r_0\geq \frac23$. Then the image of $D(z^i, r_0)\times\set{w^\iota}$ under $f^n$ is a graph over some
 neighborhood of $z^o$, which 
by Lemma \ref{lem:classical} must contain $ D(z^o, \unsur{2})$. Now since the repelling petal $\Sigma$ is a graph (relative to  the first 
coordinate) over $\set{z, \ \Re(z)>R}$, we infer that for $p\in B(p^o, \frac{1}{4})\cap \Sigma$, 
$f^n( D(z^i, r_0)\times\set{w^\iota})$ intersects $\mathcal{F}(p)$ close to $p$. Therefore we can pull back 
$\mathcal{F}(p)$ under $f^n$ to get a vertical graph intersecting $D(z^i, r_0)\times\set{w^\iota}$
along which (for the same reasons as before)  the derivative of $f^n$ along  is smaller than $b^n$, and the proof is complete in the case 
where $\mathcal{F}$ is a foliation by vertical graphs in $\om_R\times D$.

\medskip

What remains to be done is to remove the simplifying assumption on $\mathcal{F}$. 
For this we   simply iterate backwards  and use the previous analysis to show that   for $k$ large enough, 
$f^{-k}(\mathcal{F})\rest{\om_R\times D}$ is made of vertical graphs of slope $\leq 1/100$. Indeed, let $\Delta$ be a germ of a holomorphic disk transverse to $\Sigma$ at $p^o =(z_0,w_0)\in \om_R\times D$. Let $f_0^{-k}(p) = (z_{-k}, w_{-k})$ which (in our coordinates) converges to infinity by staying in $\om_R\times D$. Applying the reasoning of Lemma
 \ref{lem:derivative} for $u=0$ (together with \eqref{eq:1/5})
  shows that for every $w\in D$, 
 $f^k(D(z_{-k}, \frac23)\times\set{w})$ is a horizontal graph over $D(z_0, \unsur{2})$, which is exponentially close to $\Sigma$ due to vertical contraction. Thus when $k$ is large enough, it intersects $\Delta$ exactly in one point, and the result follows.\qed


 \section{Proof of the main theorem on homoclinic tangencies}\label{sec:proof}

\subsection{Creating tangencies between horizontal and vertical moving curves}

Here we explain how to obtain a tangency between two holomorphically moving complex curves by using only
``soft complex analysis", i.e. basically the Argument Principle. We work in the unit bidisk $\bb = \dd^2$. A subvariety $V$ in $\bb$ (or current, etc.) is   {\em horizontal} 
if there exists some $\e>0$ such that $V\subset \dd\times
\dd_{1-\e}$.  
 Vertical objects are defined similarly. 

Following \cite{hov}, we define the {\em horizontal (resp. vertical)  Poincaré cone field} as the set of tangent vectors 
$v=(v_1, v_2)\in T_x\bb\simeq\cd$ such that $\abs{v_1}_{\rm Poin} > \abs{v_2}_{\rm Poin}$ (resp. 
$\abs{v_2}_{\rm Poin} > \abs{v_2}_{\rm Poin}$), where $\abs{\cdot}_{\rm Poin}$ denotes the Poincaré metric in $\dd$. 
The contraction property of the Poincaré metric implies that if $\Gamma$ is a horizontal graph in $\bb$, then for every $x\in \Gamma$, 
$T_x\Gamma$ is contained in the horizontal Poincaré cone field. 

A horizontal manifold (or subvariety) $V$ in $\bb$ has a {\em degree}, which is the degree of the branched cover $\pi_1:V\cv\dd$ (here 
of course $\pi_1$ is the first coordinate projection). If $V$ is irreducible and $d>1$ then $\pi_1\rest{V}$ must have critical points 
(indeed otherwise it would be a non-trivial covering of the unit
disk). 
In particular, it admits tangent vectors in the vertical Poincaré cone
field. 

By definition a holomorphic family of submanifolds $(V_\la)_{\la \in \La}$ of a complex manifold $M$ is the data of a codimension 1 analytic set (which   might be singular)
 $\widehat V\subset \La\times M $ such that for every $\la\in \La$, $V_\la = \widehat V\cap (  \set{\la} \times \bb)$. 

%
%
Here is the precise statement. 

\begin{prop}\label{prop:tangencies}
Let $(V_\la)_{\la\in \La}$ be a holomorphic family of horizontal submanifolds of degree $k$ in $\bb$, parameterized by a connected Stein manifold $\La$. We assume that:
\begin{enumerate}[{\rm (i)}]
\item There exists a compact subset $\La_0\Subset \La$ such that if $\la\notin \La_0$, $V_\la$ is the union of $k$ graphs.
\item There exists $\la_0\in \La$ such that $V_\lo$ is not the union of  graphs.
\end{enumerate}
Then, if $(W_\la)_{\la \in \La}$ is any holomorphic family of vertical graphs in $\bb$, there exists $\la_1\in\La$ such that 
$V_{\la_1}$ and $W_{\la_1}$ admit a point of tangency.
\end{prop}

  Using the above remarks, we see that  condition (ii) could  be replaced by ``there exists $x\in \bb$ and $\la_0\in \La$ such that 
  $T_xV_\lo$ is contained  in the vertical Poincaré cone field". Notice also that (ii) implies that $k>1$ in (i). 
  
\begin{proof}
To simplify the exposition, we assume that $\La$ is the unit disk (we
will use the result in that case only). The proof in the general case is similar. 

Notice first that if $\la$ is close to $\fr \La$, then  there are no tangencies between $V_\la$ and $W_\la$. Indeed the tangent vectors to 
$V_\la$ and $W_\la$  belong to  disjoint    cone fields. In particular, reducing $\La$ a little bit if needed, we may assume that the $V_\la$ (resp. $W_\la$) are uniformly horizontal (resp. vertical), that is,  that they are contained in 
$\dd\times \dd_{1-\e}$ (resp. $\dd_{1-\e} \times \dd$) for some fixed $\e>0$. 

\medskip

If $V\subset \bb$ is a smooth holomorphic curve, 
we let $\pp TV$ be its lift (which is also a holomorphic curve) to the
projectivized tangent bundle $\pp T \bb \simeq \bb\times \pu$.
 Notice that since $V$ is smooth, $\pp TV$ intersects every $\pu$ fiber at a single point. 
If $(V_\la)_{\la\in \La}$ is a holomorphic  family of submanifolds, we obtain in this way a holomorphic family of submanifolds 
$(\pp TV_\la)_{\la\in \La}$ in $\bb\times \pu$. In other words, there
exists a subvariety of  $\La\times \bb\times \pu  $, of dimension 2, 
which we  denote $\widehat{\pp TV}$ such that for every $\la\in \La$, 
$$
    \widehat{\pp TV}\cap (\set{\la} \times\bb\times \pu ) = 
     \pp  TV_\la.
$$

Let now $W= (W_\la)_{\la\in \La}$ be any holomorphic family of vertical graphs in $\bb$. An intersection point between 
$\widehat{\pp TV}$ and $\widehat{\pp TW}$ corresponds to a parameter $\lo$  at which $V_\lo$ and $W_\lo$ are tangent. 
We claim that 
then  $\widehat{\pp TV}\cap \widehat{\pp TW} $ has dimension 0 (if non-empty). 
 In particular, the varieties $\widehat{\pp TV}$ and $\widehat{\pp TW}$ intersect properly in  $\bb\times \pu \times \La$. 
 Observe first that this 
 intersection is compactly supported in $\La \times\bb\times \pu  $, indeed:
\begin{itemize}
\itm  as observed above, there are no tangencies between $V_\la$ and $W_\la$ when $\la$ is close to $\fr\La$;
\itm the intersection points  between  $V_\la$ and $W_\la$ are contained in $\dd_{1-\e}^2$ for some $\e>0$. 
\end{itemize}
By the Maximum Principle, 
the projection of $\widehat{\pp TV}\cap \widehat{\pp TW}$ to
$\La\times \bb$ is a finite set. Hence
any  component of 
$\widehat{\pp TV}\cap \widehat{\pp TW}$ of positive dimension is contained in a $\pu$ fiber, 
which is impossible by definition of the lifts  $\widehat{\pp TV}$ and $\widehat{\pp TW}$. This proves our claim. 

\medskip

By assumption,  there exists $\lo$ such that $V_\lo$ admits a vertical tangent vector, hence a tangency with some vertical line $L$. Let 
  $\widehat{\pp TL}\subset \La\times \bb\times \pu$ be the surface corresponding to the trivial family where 
$L$ is fixed.  Then $\widehat{\pp TV}\cap \widehat{\pp TL}$ is non-empty, therefore  it is a finite set.  
 
We can now deform $L$ to $W$ through some 
 holomorphic family $(W_{\la, s})$ of vertical graphs
 with $W_{\la, 0} = L$ and $W_{\la, 1} = W_\la$, and $s$ ranges in
 some neighborhood $  \dd_{1+\e}$ 
 of the closed 
 unit disk. For this we can simply take a linear interpolation. In this way  we obtain a holomorphic deformation  
$(\widehat{\pp T W_s})_{s\in \overline\dd_\e}$ from $ \widehat{\pp TL}$ to $ \widehat{\pp TW}$, parameterized by a neighborhood of the unit disk.

To conclude that $\widehat{\pp TV}\cap \widehat{\pp TW}\neq \emptyset$, we argue that the set of parameters 
$s\in  \dd_{1+\e}$ such that 
$\widehat{\pp TV}\cap \widehat{\pp TW_s}\neq \emptyset$ is open   in $ \dd_{1+\e}$
 by the continuity of the intersection index of properly intersecting analytic sets of complementary dimensions
  (see \cite[Prop. 2 p.141]{chirka}) and closed 
because intersection points stay compactly contained in $\bb$. This completes the proof.
\end{proof}

\subsection{From critical points to tangencies}
In this section 
 we settle the second main step of Theorem \ref{theo:tangencies}. 
If $p_\la$ is a holomorphically varying  
periodic point for a holomorphic family 
$(f_\la)_{\la\in \La}$ of dissipative polynomial automorphisms of
$\cd$, we say that $p_\la$ {\it undergoes a non-degenerate 
semi-parabolic bifurcation at $\la_0$} 
 if one of the multipliers $\rho_\la$ of $p_\la$ satisfies  $\rho_\lo=1$ and 
 $\frac{\fr \rho}{\fr \la}\rest{\la =\lo}$ is a submersion $\La\cv\cc$. 

\begin{prop}\label{prop:critical to tangency}
Let $(f_\la)_{\la\in \La}$ be a holomorphic family of dissipative
polynomial automorphisms of $\cd$ 
of dynamical degree $d$ with positive entropy, parameterized by the unit disk. Assume that: 
\begin{itemize}
\itm there exists a holomorphically varying periodic point $p_\la$ which   
admits a non-degenerate semi-parabolic bifurcation at $\la_0$;
\itm   for $\la=\lo$, there is a    
critical point in one of the basins of attraction of $p_\lo$. 
\end{itemize}
Then $\la_0$ can be approximated by  parameters possessing homoclinic tangencies. 
\end{prop}   


\begin{proof}
Without loss of generality we may assume that $\La$ is the unit disk, $\lo = 0$, and $p_\la$ is fixed. Normalize the situation so that $f_\la$ is locally of the form \eqref{eq:flambdaqbis}. Conjugating by a rotation, we may assume that the critical point lies in the basin 
$\mathcal{B}$ corresponding to the attracting direction $\set{(x, 0), \ \arg(x) = -\frac{\pi}{k}}$. Let $q_0$ be a saddle point such that 
$W^u(q_0)$ admits a point of tangency  with the strong stable
foliation in $\mathcal{B}$. (Then, by the hyperbolic $\la$-lemma
 (see e.g. \cite[Thm 2, p. 155]{palis takens})
and the fact that any pair of 
saddle points have transverse heteroclinic intersections, any saddle
point would do.)  Let  $t$ be such a point of tangency. 

The global 
repelling petal $\Sigma$ in the direction of $\set{(x, 0), \ \arg(x) = 0}$ is an immersed curve biholomorphic to $\cc$  (\cite{ueda}, see \S\ref{sec:semi parabolic}).
 Hence, using the 
theory of laminar currents,  exactly as  in \cite{bls}, it admits transversal intersections with $W^s(q_0)$. We  fix a transverse intersection point $m\in \Sigma\cap W^s(q_0)$.

Fix  a neighborhood $W$ of 0 in $\La$. Iterating $t$ forward and $m$
backward a few times, we may assume that both points are 
close to 0.
 Theorem \ref{thm:tv} thus provides us with an integer $N$ and a radius $r$, so that for $n\geq N$, 
a  transition mapping $f^n_{\la_n}$ is  defined from 
$D^2(t,r)$ to $D^2(m,r)$   
 such that $f^n_{\la_n}(t)=m$. Let $\Gamma^u_\la$ be the  
component of $W^u_\la(q_\la)\cap D^2(t,r)$ containing $t$. 
Reducing  $W$ if necessary,  $\Gamma^u_\la$ can be followed 
holomorphically as $\la\in W$. 

To make the situation visually clearer, we consider adapted coordinates close to $t$ and $m$. These  changes of coordinates  have 
bounded derivatives. Abusing   slightly, we declare that  in the new coordinates, the neighborhoods remain of size $r$.  
Near $t$ we choose $(z^\iota, w^\iota)$ so that $t=(0,0)$,  the strong stable foliation $\mathcal{F}^{ss}$
becomes the vertical foliation $\set{z^\iota = C^{st}}$ and  for $\la\in W$, $\Gamma^u_\la$ is a horizontal manifold in $D^2(t, r)$ of some degree $d\geq 2$, which is transverse to  $\mathcal{F}^{ss}$ outside $t$.  
Near $m$ we choose local coordinates $(z^o,w^o)$ such  that  $m=(0,0)$ 
  $\Sigma^0 = \set{w^o=0}$, and the component of $W^s(q_\la)$ containing $m$ is $\set{z^o=0}$. Denote by $\mathcal{F}$ the vertical foliation in the target bidisk $D^2(m,r)$. 

 For $\abs{z^o}\leq r$,  let us     consider
 the parameter $\la_n= \la_n(t,z^o)$  
 given by Theorem \ref{thm:tv} such that 
the first coordinate of $f^n_{\la_n } (t)$ is $z^o$. 
For every such parameter, by Theorem \ref{thm:tv}
$f^n_{\la_n }$ 
realizes a crossed mapping  of degree 1 \cite{hov}  from $D^2(t,r)$ to $D(z^o, \frac{r}{2})\times \dd_r$.  
 So when $\abs{z^o}\leq \frac{r}{4}$ we get  by restriction a crossed mapping from $D^2(t,r)$ to  $\dd_{\frac{r}{4}} \times \dd_r$. 
In particular, we infer that for $\abs{z^o}\leq \frac{r}{4}$, 
$ f^n_{\la_n }  (\Gamma^u_{\la_n  } )$
  is a horizontal submanifold of degree $d$ in $\dd_{\frac{r}{4}} \times \dd_r$.

\medskip

 To conclude the argument, let us show that when $z^o$ ranges in
 $\dd_{\frac{r}{4}}$ and $n$ is large,
 the family of curves $ f^n_{\la_n }(\Gamma^u_{\la_n })$ satisfies the assumptions of Proposition \ref{prop:tangencies} in 
 $\dd_{\frac{r}{8}} \times \dd_r$. 
 
 The first observation is that the   preimage $\mathcal{F}^{-n}_{\la_n}$ of $\mathcal{F}$ under  
$f^n_{\la_n}$ converges to the strong stable foliation associated to $f_0$ in $D^2(t,r)$, uniformly in $z^o$. 
Indeed, we know that the 
 leaves of $\mathcal{F}_{-n}$ are graphs with bounded geometry over some fixed direction, 
and the  $f^n_{\la_n}$ contract exponentially 
 along these graphs, with uniform bounds. So any cluster limit  of $\mathcal{F}^{-n}$ must be $\mathcal{F}^{ss}(f_0)$,  which proves our claim. 
 
 Now,  when $\abs{z^o} = \frac{r}{4}$, 
for every $z$ such that $\abs{z}<\frac{3r}{16}$, when $n$ is large, for $\la_n=\la_n(t,z^o)$
$f^{-n}_{\la_n}(\set{z}\times \dd_r )$ is close to a leaf of $\mathcal{F}^{ss}$
which intersects $\dd_r\times\set{0}$ transversely  at a distance $\geq \frac{r}{8}$ from $t$. 
Therefore, 
$f^{-n}_{\la_n}\big(\dd_{\frac{3r}{16}} \times \dd_r \big)\cap \Gamma_{\la_n}^u$ is the union of $d$ graphs, 
over a disk of radius greater than $\frac{3r}{16} \cdot \frac45 = \frac{3r}{20}$ by Lemma \ref{lem:classical}. 
Pushing by $f^n_{\la_n}$ and applying the lemma again, we conclude that 
 $f^n_{\la_n}(\Gamma_{\la_n}^u)$ is a union of $d$ graphs over $\dd_{\frac{3r}{25}}$.
  
 On the contrary, when ${z^o}=0$, for $\la_n = \la_n(t,0)$,  $\mathrm{pr_1}(f^n_{\la_n }) = 0$. Since $\mathcal{F}^{-n}$ converges to 
 $\mathcal{F}^{\rm  ss}$, when $n$ is large $ \Gamma_{\la_n }^u$ has a tangency with $\mathcal{F}^{-n}$ close to $t$, hence 
 $f^n_{\la_n }(\Gamma_{\la_n }^u)$ has a vertical tangency close to $m$.  
 
We see that  the assumptions of Proposition \ref{prop:tangencies} are satisfied so there exists a parameter $\la_n=\la_n(t,z^o)$ such that $f^n_{\la_n}(\Gamma_{\la_n }^u)$ has a tangency with $W^s(q_{\la_n})$ close to $m$, and we are done.
\end{proof}

\begin{proof}[Proof of Theorem \ref{theo:tangencies}]
Let $(f_\la)_{\la\in \La}$ be a holomorphic family of polynomial automorphisms of $\cd$ of dynamical 
degree $d\geq 2$ with a bifurcation at $\la_0$. By Proposition \ref{prop:henon} we may assume that the 
 $f_\la$ are products of Hénon mappings.  Then close 
to $\la_0$ there is a periodic point with a multiplier $\rho$ crossing the unit circle. Without loss of generality we may 
assume that $\dim(\La) =1$.
Hence there exists $\la_1$ close to $\la_0$ such that at $\la_1$, the multiplier  is a root of unity, different from 1, 
so that this periodic point $p_\la$ can be followed holomorphically close to $\la_1$. In addition we may assume that 
$\frac{\fr\rho}{\fr\la}\rest{\la=\la_1}\neq 0 $. Replacing $f_\la$ by $f_\la^k$ for some $k$, we may assume that 
$p_\la$ is fixed and $\rho_{\la_1} = 1$ (we keep the same notation for the new multiplier, which is the $k^{\rm th}$ power of the previous one) . Notice that for the new multiplier we still have that   $\frac{\fr\rho}{\fr\la}\rest{\la=\la_1}\neq 0$.

Since the condition that $\abs{\Jac f_\la} < \deg(f_\la)^{-2}$  is preserved under iteration, Theorem \ref{theo:critical} asserts that there is a critical point in every  component of the attracting basin of $p_{\la_1}$. Thus  the result follows from 
 Proposition \ref{prop:critical to tangency}.   
 \end{proof}

\appendix
 \section{Attracting basins}  \label{sec:attracting}
 The methods of \S\ref{subs:critical} also give 
 the existence of ``critical points" in attracting basins, under
 certain minor   hypotheses (that are needed to even  define the
 critical points).   Though these results are not used   in the paper,
they are interesting on their own right. 
 
 Let $f$ be a polynomial automorphism of dynamical degree $d\geq 2$ 
 with an attracting point $p$. As usual, we may assume that $p$ is fixed, 
 and we denote by $\mathcal{B}$ its basin of attraction. It is classical that there is a local 
 holomorphic change of coordinates which puts $f$ in a simple normal form 
 (this result goes apparently  back 
 to Lattès \cite{lattes}).
 Let $\kappa_1$ and $\kappa_2$ be the eigenvalues of $Df_p$,  
 ordered so that $0< \abs{\kappa_2}\leq \abs{\kappa_1} <1$. 
 We say that $(\kappa_1, \kappa_2)$ is resonant if there exists an integer $i\geq 1$ such that $\kappa_2 = \kappa_1^i$ (notice that   $i=1$ is allowed).  
Then there exists a local change of coordinates near $p$    
such that in the new coordinates $(z_1, z_2)$, $f$ expresses as 
 $$f(z_1, z_2 ) =  \begin{cases}
 (\kappa_1 z_1, \kappa_2  z_2) \text{ if }(\kappa_1, \kappa_2) \text{ is not resonant,}\\
 (\kappa_1 z_1, \kappa_2  z_2+ \alpha z_1^i) \text{ otherwise, where }i\text{ is as above, and }\alpha\in \set{0,1}.
 \end{cases} $$ 
In any case, we see that the vertical foliation $\set{z_1 = C}$ 
 is invariant under $f$.  If $\abs{\kappa_2}<\abs{\kappa_1}$   
this is the ``strong stable foliation", characterized by the property that points in the same leaf approach 
each other at the fastest possible rate $\kappa_2^n$. As before, it
 will be denoted by $\mathcal{F}^{ss}$.
Using the dynamics, the  coordinates $(z_1, z_2)$ extend to the basin and define a biholomorphism $\mathcal{B}\simeq \cd$.
 In the non-resonant (i.e. linearizable) 
 case, the foliation  $\set{z_2 = C}$ 
 is invariant as well. 
  We then simply refer to $\set{z_1=C}$ 
 and $\set{z_2=C}$ as the invariant coordinate foliations in $\mathcal{B}$.  
  
We give two statements on the existence of critical points. The first one     parallels Theorem~\ref{theo:critical}.
 
 \begin{thm}\label{thm:critical hyperbolic1}
 Let $f$ be a polynomial automorphism of $\cd$ of dynamical degree
 $d\geq 2$, possessing an attracting point $p$,  whose 
eigenvalues satisfy  $0<\abs{\kappa_2}<\abs{\kappa_1}<1$, with basin of
 attraction $\mathcal{B}$. Assume that 
$\abs{\Jac f}<{d^{-4}}$, or more generally that the connected component of  $p$ in $W^{ss}(p)\cap J^-$ is $\set{p}$. 
Then for every saddle periodic
  point $q$, every component of $W^u(q)\cap \mathcal{B}$ contains a 
 critical point, that is, a point of tangency with the strong stable foliation in $\mathcal{B}$.
 \end{thm}
  
  The second statement concerns the hyperbolic case. 
  
 \begin{thm}\label{thm:critical hyperbolic2}
 Let $f$ be a polynomial automorphism of $\cd$ of dynamical degree $d\geq 2$, possessing an attracting point $p$ with 
 basin $\mathcal{B}$. Assume that $f$ is uniformly hyperbolic on $J$, and fix any saddle periodic point $q$.
 
 If the eigenvalues of $p$ satisfy      $\abs{\kappa_2}<\abs{\kappa_1}$ (resp. are  non-resonant), then every component of 
    $W^u(q)\cap \mathcal{B}$ admits a tangency  with the strong stable foliation of $\mathcal{B}$   
  (resp. with   both invariant coordinate foliations).
 \end{thm} 

%
%
%

Here is an interesting geometric consequence. Recall that if $f$ is dissipative and 
hyperbolic, $J^+$ is (uniquely) laminated by stable manifolds. Let us denote by $W^s(J)$ this lamination.
It is natural to wonder whether 
the strong stable foliation in $\mathcal{B}$ matches continuously with the lamination of $J^+$ (recall that $\fr \mathcal{B} =J^+$). 
The existence of critical points implies that this is never the case 
(compare \cite[Cor. A.2]{bs7}).  

\begin{cor}
Let $f$ be as in the previous theorem,  in particular $f$ is hyperbolic on $J$.  
Then if $p$ is an attracting point with eigenvalues $\abs{\kappa_2}<\abs{\kappa_1}$ and basin $\mathcal{B}$, then
  for every $x\in J$,  $W^s(J)\cup \mathcal{F}^{ss}(\mathcal{B})$ does not define a lamination near $x$.  
If $p$ is linearizable, the same holds for 
both invariant coordinate foliations. 
\end{cor}
 
 \begin{proof}
 Let us deal with the case where $\abs{\kappa_1}<\abs{\kappa_2}$.
  It is enough to assume that $x$ is a saddle periodic point. Hyperbolicity implies that 
  $W^u(J)$ and $W^s(J)$ are transverse near $x$, so if $W^s(J)\cup \mathcal{F}^{ss}(\mathcal{B})$  is a lamination near $x$, 
 $\mathcal{F}^{ss}(\mathcal{B})$ must be transverse to $W^u(J)$ near $x$.  
 On the other hand, there exist critical points on $W^u(x)$ arbitrary close to $x$ (obtained from the previous ones by iterating backwards). This contradiction finishes the proof. \end{proof}

\begin{proof}[Proof of Theorems   \ref{thm:critical hyperbolic1} and \ref{thm:critical hyperbolic2}]
This is very similar to Proposition \ref{prop:critical} so the proof is merely sketched. Let us first deal with the case where 
$\abs{\kappa_2}<\abs{\kappa_1}$, with $f$ hyperbolic or not. 
  Let $\pi_1:\mathcal{B}\cv\cc$ be the projection along the strong stable foliation. 
In the coordinates $(z_1, z_2)$, it simply expresses as $(z_1, z_2)\mapsto z_1$. 
Assume by contradiction that there is no critical point in $\om$. 
Then $\pi_1\circ \psi^u: \om\setminus (\psi^u)^{-1} (W^{ss}(p)) \cv \cc^*$ is a locally univalent map. 
Since it cannot be a covering it must possess an asymptotic value, hence there is a 
diverging path $\gamma $ in $\om$ such that the limit  
$\lim_{t\cv\infty} \pi_1\circ \psi^u(\gamma(t)) = \omega$ exists in $\cc^*$. 
Let $\pi_2:\mathcal{B}\cv\cc$ be the second coordinate projection. As before, we split the argument according to the bounded or unbounded character of $\pi_2\circ \psi^u(\gamma)$. 

\medskip

If $\pi_2\circ \psi^u(\gamma)$ is unbounded, we iterate forward and take cluster values to create an un\-bounded 
component $C$ of $J^-\cap W^{ss}(p)$ containing $p$.  Now if 
$\abs{\Jac f}< d^{-4}$, then $\abs{\kappa_2}<d^{-2}$, so by Corollary \ref{cor:point}, the component of $p$ in 
$W^{ss}(p)\cap J^-$ is a point, and we get a contradiction.

 
If $f$ is hyperbolic 
we argue as follows:  in $\mathcal{B}\setminus\set{p}$, $J^-$ is laminated by unstable manifolds. In 
particular by \cite[Lemma 6.4]{bls} the set of tangencies between
$W^{ss}(p)$   
and the unstable lamination is discrete. 
Pick $c\in C\setminus\set{p}$   such that $W^{ss}(p)$ and the unstable lamination are transverse near $c$. There exist coordinates $(x,y)$ close to $c$ in which $W^{ss}(p)$ is $\set{x=0}$, $c=(0,0)$, 
and the leaves of the unstable lamination close to $c$ are horizontal in the 
unit bidisk 
$\bb$.  By construction, there is a sequence of integers  $n_j$ such that 
$f^{n_j}(\psi^u(\gamma))$  
 has a component $C_j$, vertically  contained in $\bb$,  
 touching the boundary, and 
 passing close to $c$. On the other hand $C_j$ must be contained in a leaf of the unstable foliation so we get a contradiction. 
 
 
 \medskip
 
If $\pi_2\circ \psi^u(\gamma)$ is bounded, then as before the path $\gamma$ must be unbounded in $W^u(q)$. Let $E$ 
be the cluster set of $\psi^u(\gamma)$, which is a compact subset of the  strong stable leaf 
$\set{z_1 =\omega}$. If $\abs{\kappa_2}<\abs{\kappa_1}$, then as in Proposition  \ref{prop:critical} 
we make a linear change of coordinates close to $p$ such that in the new coordinates, $f$ expresses as 
$f(x,y )  = (\kappa_1 x, \kappa_2y)+ h.o.t.$
We see that in these coordinates, $\mathrm{pr_1}(f^n(E))$ is a set of diameter 
$\lesssim \kappa_2^n$ about $x_n\sim c \kappa_1^n$, which leads to  a contradiction with Theorem \ref{thm:fuzzy DCA}, 
exactly as in Proposition  \ref{prop:critical}.

\medskip

It remains to treat the case where $f$ is hyperbolic, $p$ is linearizable, 
and we look for tangencies with any of the invariant coordinate foliations. 
We argue  exactly as before, with $(\pi_1, \pi_2)$ 
being  the linearizing  coordinate projections, in any order, and keep the same notation. 
The case where $\pi_2\circ \psi^u(\gamma)$ is unbounded is dealt with exactly as above. If now
 $\pi_2\circ \psi^u(\gamma)$ is bounded and $E$ denotes its cluster set in the leaf $\set{z_1 = \omega}$,  
  we observe  that   as in the unbounded case, the laminar 
structure of $J^-$ outside $p$ forces $E$ to be reduced to a point. Therefore $\psi^u$ admits an asymptotic value in $\mathcal{B}
\setminus\set{p}$ and the contradiction arises by iterating and applying the   ordinary Denjoy-Carleman-Ahlfors theorem.
\end{proof}

\end{document}